\documentclass[11pt,english,reqno]{amsart}
\pdfoutput=1
\usepackage[foot]{amsaddr}
\usepackage{todonotes}
\usepackage{tikz,todonotes}
\usepackage{pgfplots}
\pgfplotsset{compat=1.18}

\allowdisplaybreaks
\usepackage{graphicx}
\usepackage{subcaption} % for subfigures
\usepackage{caption}    % for better control of captions

\usepackage[linktocpage,colorlinks,linkcolor=blue,anchorcolor=blue,citecolor=blue,urlcolor=blue,pagebackref]{hyperref}
\RequirePackage[authoryear]{natbib}%% uncomment this for author-year citations
\setcitestyle{authoryear,open={(},close={)}}

\usepackage{tikz-cd}
\usepackage{comment}
\usepackage{tikzsymbols}

\usepackage[letterpaper]{geometry}
\usepackage{fancyhdr,listings,amsfonts,amsmath,color}
\usepackage{latexsym,amssymb,amsmath,amsfonts,amsthm,xcolor,graphicx,longtable,wrapfig}
\usepackage[utf8]{inputenc}
\usepackage{geometry}
\usepackage{textgreek}
\usepackage[capitalise]{cleveref}
\usepackage{stmaryrd,mathtools}

\allowdisplaybreaks
\renewcommand*{\backref}[1]{\ifx#1\relax \else Page #1 \fi}
\renewcommand*{\backrefalt}[4]{%
  \ifcase #1 \footnotesize{(Not cited.)}%
  \or        \footnotesize{(Cited on page~#2.)}%
  \else      \footnotesize{(Cited on pages~#2.)}%
  \fi
}

\usepackage{amsmath}
\usepackage{graphicx}
\usepackage{float}
\usepackage{color}
\usepackage{pifont}
\usepackage{latexsym,amssymb,amsmath,amsfonts,amsthm,xcolor,graphicx}
\usepackage[utf8]{inputenc}
\usepackage{amsmath}
\usepackage{graphicx,enumitem}
\usepackage{float}
\usepackage{color}
\usepackage{pifont}
\usepackage{amssymb}
\usepackage{multirow}
\usepackage{diagbox}
\usepackage{cancel}
\usepackage{bbm}
\usepackage{xcolor}
\usepackage[ruled,vlined]{algorithm2e}
\usepackage{algpseudocode}
\usepackage{mathrsfs}
\usepackage{amsfonts}
\numberwithin{equation}{section}
\renewcommand{\d}{\mathrm{d}}

\newtheorem{theorem}{Theorem}[section]

\newtheorem{lemma}{Lemma}[section]
\newtheorem{corollary}{Corollary}[section]
\newtheorem{definition}{Definition}[section]
\newtheorem{proposition}{Proposition}[section]

\newtheorem{remark}{Remark}[section]

\theoremstyle{definition}
\title[]{On the Structure of Stationary Solutions to McKean-Vlasov Equations with Applications to Noisy Transformers} 

\author{Krishnakumar Balasubramanian$^1$}
\address{$^1$Department of Statistics, University of California, Davis, Email: \texttt{kbala@ucdavis.edu} }
\author{Sayan Banerjee$^2$}
\address{$^2$Department of Statistics and Operations Research, University of North Carolina, Chapel Hill. Email: \texttt{sayan@email.unc.edu} }
\author{Philippe Rigollet$^3$}
\address{$^3$Department of Mathematics, Massachusetts Institute of Technology. Email: \texttt{rigollet@mit.edu}}

\begin{document}
\begin{abstract}
We study stationary solutions of McKean-Vlasov equations on the circle. Our main contributions stem from observing an exact equivalence between solutions of the stationary McKean-Vlasov equation and an infinite-dimensional quadratic system of equations over Fourier coefficients, which allows explicit characterization of the stationary states in a sequence space rather than a function space. This framework provides a transparent description of local bifurcations, characterizing their periodicity, and resonance structures, while accommodating singular potentials. We derive analytic expressions that characterize the emergence, form and shape (supercritical, critical, subcritical or transcritical) of bifurcations involving possibly multiple Fourier modes and connect them with discontinuous phase transitions. We also characterize, under suitable assumptions, the detailed structure of the stationary bifurcating solutions through leading small-amplitude coefficients for arbitrarily many fixed Fourier modes. At the global level, we establish regularity and concavity properties of the free energy landscape, proving existence, compactness, and coexistence of globally minimizing stationary measures, further identifying discontinuous phase transitions with points of non-differentiability of the minimum free energy map. As an application, we specialize the theory to the \emph{Noisy Mean-Field Transformer} model, where we show how changing the inverse temperature parameter $\beta$ affects the geometry of the infinitely many bifurcations from the uniform measure. We also discuss how increasing $\beta$ motivates the study of approximate multi-mode stationary profiles as candidates for a separate metastability analysis; persistence times require additional dynamical estimates. Moreover, we determine the exact threshold
$\beta_0\simeq2.447$ separating continuous and discontinuous
phase transitions in this model. The same threshold marks
the change of the first bifurcation from supercritical
to subcritical.
\noindent\newline

% \noindent \textbf{AMS 2020 subject classifications:} 35Q83, 35Q70, 34K18, 60H50, 82C22, 35B27.\newline

\end{abstract}

\maketitle
%\tableofcontents

\section{Introduction}

Consider the following McKean-Vlasov equation on the circle $\mathbb{S}^1 = [0, 2\pi)$:
\begin{align}\label{eq:mveq}
    \partial_t p  = \partial_{\theta\theta} p - \kappa \partial_\theta (p\, \partial_\theta(W\star p)), 
    \qquad p \in C^{1,2}([0,\infty) \times \mathbb{S}^1),
\end{align}
where $W : [0, 2\pi) \to \mathbb{R}$ is the \emph{interaction potential},
\[
(W \star p)(\theta) := \frac{1}{2\pi}\int_0^{2\pi} W(\theta - \phi)\, p(\phi)\, \d\phi,
\]
and we impose the normalization and positivity conditions 
\[
p(\cdot,\cdot) \ge 0, 
\qquad 
\frac{1}{2\pi}\int_0^{2\pi} p(\cdot,\theta)\, \d\theta = 1,
\]
so that $p(t,\cdot)$ defines a probability density on $\mathbb{S}^1$ for each $t \ge 0$, and $\kappa > 0$ represents the \emph{interaction intensity}. Throughout this article, $W$ is assumed to be even, square integrable, and have the Fourier expansion $W(\theta) = \sum_{\ell=1}^\infty a_\ell \cos \ell\theta$. \cref{eq:mveq} is a canonical nonlinear diffusion equation that arises as the mean-field limit, as $N \to \infty$, of systems of weakly interacting particles on $\mathbb{S}^1$ whose angular dynamics evolve according to
\begin{equation}\label{eq:partsyst}
    d\theta_i(t) = \frac{\kappa}{N}\sum_{j=1}^N W'(\theta_i(t) - \theta_j(t))\,d t + \sqrt{2}\,d B_i(t), \quad 1 \le i \le N,
\end{equation}
where $B_1,\dots,B_N$ are independent standard Brownian motions. The term $\partial_\theta W(\theta_i - \theta_j)$ represents the interaction force exerted by particle $j$ on particle $i$, governing the strength and direction of coupling: attractive potentials promote synchronization, while repulsive ones lead to dispersion or multimodal states. In the mean-field limit, these pairwise forces collectively produce the nonlinear drift term $- \kappa \partial_\theta (p\, \partial_\theta(W\star p))$ in~\cref{eq:mveq}, which combines with the diffusive term $\partial_{\theta\theta} p$ to describe the evolution of the population density on $\mathbb{S}^1$. The McKean-Vlasov equation arises in a wide range of mean-field models, including synchronization dynamics \citep{kuramoto1981rhythms,bertini2010dynamical}, granular media \citep{bolley2013uniform}, collective behavior in biological systems \citep{keller1971model}, and opinion dynamics \citep{hegselmann2015opinion}, to name a few.

\medskip

Our primary interest lies in explicitly characterizing the \emph{stationary} weak solutions to~\eqref{eq:mveq}, which satisfy
\begin{align}\label{eq:stat}
    \partial_{\theta\theta} p = \kappa \partial_\theta (p\, \partial_\theta(W\star p)).
\end{align}
Such stationary solutions approximate equilibria of the interacting particle system \eqref{eq:partsyst} and coincide with the critical points of the associated free energy functional $\rho \mapsto \mathcal{F}(\rho, \kappa)$, where
\begin{align*}
    \mathcal{F}(\rho, \kappa) := \frac{1}{2\pi}\int_0^{2\pi} \rho(\theta) \log \rho(\theta) \d \theta - \frac{\kappa}{8\pi^2}\int_0^{2\pi}\int_0^{2\pi}W(\theta - \phi)\rho(\theta)\rho(\phi)\d\theta\d\phi.
\end{align*}
The uniform distribution on $\mathbb{S}^1$ is a trivial stationary solution as both the left and right hand sides of \eqref{eq:stat} vanish with this choice of $p$. Our primary focus is on the existence, shape and bifurcation structure of other, nontrivial, stationary solutions. A precise, explicit understanding of these equilibrium densities is crucial for several reasons. First, the structure and multiplicity of stationary solutions determine the possible long-time behaviors and phase transitions of the underlying mean-field dynamics. Second, the stability of these equilibria governs whether the system converges to the uniform distribution, organizes into coherent patterns, or exhibits abrupt discontinuous transitions as the interaction strength~$\kappa$ varies. Finally, obtaining detailed representations of $p$ in terms of $W$ and $\kappa$ enables connecting qualitative features of the equilibrium density, such as periodicity and multi-modality, to the spectral properties of~$W$.  

Analysis of stationary solutions to the McKean-Vlasov equation has traditionally been tackled in function spaces, often requiring advanced tools from calculus of variations; see, for example, \cite{carrillo2020long}. A key feature that distinguishes the approach taken in this paper from previous ones is the observation that the problem of finding stationary solutions can be transformed exactly into finding solutions to an infinite-dimensional system of quadratic equations in a Fourier sequence space. This shift from a function space to a sequence space enables us to do a more transparent and elaborate analysis. 

\subsection{Background and challenges}\label{sec:bgandchallenges}
Long-time behavior for McKean-Vlasov equations is a subject of extensive current research. Many results in this direction, for McKean-Vlasov equations on $\mathbb{R}^d$, assume some form of \emph{convexity}, either of the (negative) interaction potential $-W$, or by introducing a uniformly convex \emph{confinement potential} $V$ as an additional term $+\partial_\theta(p\partial_\theta V)$ in the RHS of \eqref{eq:mveq} which overcomes the lack of convexity of $-W$. Other results exploit weaker notions of convexity via functional inequalities like the log-Sobolev inequality. See, for example, \cite{carrillo2003kinetic,malrieu2003convergence,cattiaux2008probabilistic,guillin2021uniform}. In this case, the free energy functional inherits convexity properties and, in particular, admits a unique global minimizer, sometimes up to translations. This greatly simplifies the energy landscape and, in addition, one also obtains rates of convergence to stationarity by viewing \eqref{eq:mveq} as a gradient flow for the free energy on the $L^2$-Wasserstein space and applying log-Sobolev inequalities. In our setting, there is no confinement potential and a \emph{nonconstant periodic $-W$ cannot be globally convex}. The resulting free-energy landscape can therefore have
multiple stationary states and nontrivial phase transitions.

For the non-convex case, early works on stationary solutions and phase transitions for McKean–Vlasov (and related Vlasov–Fokker–Planck) equations focus on existence questions and qualitative behavior for specialized interaction functions. \citet{dawson1983critical} establishes critical dynamics and fluctuation results for mean-field cooperative models with double-well confinement and Curie–Weiss-type interaction potentials. \citet{yozo1984asymptotic} analyzes bifurcations, showing non-uniqueness of stationary solutions for the McKean–Vlasov equation on $\mathbb{R}^d$. For the double-well confinement potential case, \citet{tugaut2014phase} provides a quantitative description of bifurcations and the multiplicity of invariant measures, clarifying how the double-well structure drives phase transitions.

Work on phase transitions and bifurcation structures for spatially periodic or finite-volume problems advances rapidly in recent years. \citet{chayes2010mckean} analyze the free energy functional for general interaction potentials on $\mathbb{T}^d$ and characterize the continuity of phase transitions for its global minimizers as $\kappa$ varies. In particular, they show that when all Fourier coefficients of $W$ are negative, the free energy becomes convex in a suitable sense and the uniform distribution on $\mathbb{T}^d$ is the unique minimizer of $\mathcal{F}(\kappa, \cdot)$ for all $\kappa$. This implies that non-trivial stationary distributions arise only when $W$ admits at least one positive Fourier mode. Closest in theme to our work is \citet{carrillo2020long}, who provide criteria for local bifurcations from the uniform distribution based on positive Fourier modes of $W$ and give sufficient conditions for both continuous and discontinuous phase transitions on the torus. Their analysis requires intricate calculus on function spaces: the bifurcation study involves fixed points of a map on $L^2(\mathbb{T}^d)$, where the associated Fréchet derivatives make the approach difficult to extend beyond single-mode bifurcations (see \cref{rem:CRrem}). This framework is recently extended to the sphere $\mathbb{S}^d$ by \citet{shalova2024solutions}.

%Work on phase transitions and bifurcation structure for spatially periodic or finite-volume problems has advanced considerably in recent years.~\cite{chayes2010mckean} analyzed the free energy functional for general interaction potentials on $\mathbb{T}^d$ and characterized the continuity of phase transitions for its global minimizers as $\kappa$ varies. In particular, they showed that when all the Fourier coefficients of $W$ are negative, then the free energy becomes convex in a suitable sense and the trivial uniform distribution on $\mathbb{T}^d$ is the unique minimizer for $\mathcal{F}(\kappa, \cdot)$ for all $\kappa$. This indicates that non-trivial stationary distributions can only be expected when there is a positive Fourier mode for $W$. Closest in theme to our paper is the work of~\cite{carrillo2020long} which provides criteria for local bifurcations from the \SB{uniform distribution} based on positive Fourier modes of $W$, and gives sufficient conditions for continuous and discontinuous phase transitions on the torus. The analysis in~\cite{carrillo2020long} requires involved calculus on function spaces, as aforementioned. For instance, the bifurcation analysis involves investigating the fixed point of a map on $L^2(\mathbb{T}^d)$ and the complicated Fr\'echet derivatives required in the process make this approach hard to extend beyond bifurcating solutions with only one dominant mode; see \cref{rem:CRrem} below. This analysis was recently adapted for the sphere $\mathbb{S}^d$ by \cite{shalova2024solutions}. 

\subsection{Our contributions} 

\subsubsection{General results on stationary McKean-Vlasov equations on $\mathbb{S}^1$}
In this work, we develop a Fourier-analytic framework for analyzing~\eqref{eq:stat}, which transforms the nonlinear PDE into an infinite-dimensional system of quadratic equations in a Fourier space; see \cref{lem:lemma1}. This representation provides a transparent description of stationary solutions as the zeroes of the map $F: \ell^2_w \times \mathbb{R}_+ \rightarrow \ell^2$ given by
\begin{align*}
    F_{\ell}(\underline{p}, \kappa) = \ell(2-\kappa a_\ell) p_\ell - \kappa \sum_{j < \ell}  j a_j p_j p_{\ell-j} - \kappa \sum_{j > \ell} (ja_j - (j-\ell)a_{j-\ell})p_jp_{j-\ell}, \ \ell \ge 1,
\end{align*}
where $\ell^2_w := \{\underline{p} \in \mathbb{R}^\infty: \sum_{\ell=1}^\infty (1 + \ell^2)p_\ell^2 < \infty\}$ and $\ell^2 := \{\underline{p} \in \mathbb{R}^\infty: \sum_{\ell=1}^\infty p_\ell^2 < \infty\}$. Technically, this transfers the analysis from function spaces treated in~\cite{carrillo2020long} to more tractable sequence spaces, thereby enabling the investigation of more detailed properties of stationary solutions. As a first step, we adapt the Crandall-Rabinowitz \citep{crandall1971bifurcation} setup used in \cite{carrillo2020long} to our function $F$ in \cref{thm:CRthm} that gives rise to non-trivial bifurcating solutions at intensity values $\kappa = 2/a_{\ell^*}$, where $a_{\ell^*}$ is a positive Fourier mode of $W$ such that $a_j \neq a_{\ell^*}$ for $j \neq \ell^*$. In comparison to \citet[Theorem 4.2]{carrillo2020long}, our result \emph{applies to more singular potentials $W$} and also \emph{clarifies how the bifurcating branches `bend'} forward or backward (supercritical or subcritical bifurcations; see \cref{fig:bifurcation}) based on the sign of the ratio 
\begin{align}\label{eq:lstarsignature}
R_{\ell^*}(W)\coloneqq \frac{a_{\ell^*} - 2a_{2\ell^*}}{a_{\ell^*} - a_{2\ell^*}},
\end{align}
which we refer to as the $\ell^*$-\emph{signature} of $W$. In \cref{thm:period}, we characterize the \emph{periodicity of non-trivial bifurcating solutions} and, in particular, show that if a bifurcation point arises from possibly multiple coinciding Fourier modes of $W$, then any non-trivial bifurcating solution has to be $g$-periodic, where $g$ is the greatest common divisor of the indices of the coinciding Fourier modes. As a consequence, if $a_1> a_2> \dots>0$, the periodicity of bifurcating branches at the subsequent bifurcation points $2/a_1 < 2/a_2 < \dots$ increases.

In the setting of \citet[Theorem 4.2]{carrillo2020long} and our \cref{thm:CRthm}, the uniqueness assertion $a_j \neq a_{\ell^*}$ for $j \neq \ell^*$ leads to single mode bifurcating branches around $\kappa = 2/a_{\ell^*}$ where, locally, only one Fourier mode of the solution dominates.  In \cref{thm:high}, we characterize multi-mode branches when multiple Fourier coefficients of $W$ coincide. This leads to a rich class of bifurcations. As a special case, under the `resonance assumption' $a_\ell = a_m = a_{\ell + m}$, we obtain \emph{transcritical bifurcations} under suitable assumptions (see \cref{fig:bifurcation}). This analysis motivates approximate stationary profiles when multiple Fourier modes of $W$ are close to each other (see \cref{rem:metahigh}). These profiles are candidates for a separate metastability analysis. In \cref{thm:main_thm}, under slightly stronger assumptions on $W$ than in \cref{thm:CRthm}, we give a very detailed structural characterization of the stationary solutions through leading small-amplitude coefficients for arbitrarily many fixed Fourier modes in the exponent.

In \cref{sec:dpt}, we study phase transitions which correspond to a shift in the energy minimizing branch from the uniform distribution to a non-uniform one as $\kappa$ increases, see \cref{def:contpt}.
We \emph{connect the shape of bifurcations to the continuity of phase transitions} and show that bifurcations that have a branch in the decreasing $\kappa$ direction (e.g. subcritical and transcritical bifurcations) lead to discontinuous phase transitions of the free energy minimizers; see \cref{thm:bifdpt} and~\cref{fig:kappa-quadratic-overlay}. \cref{cor:dpt} exhibits instances when this happens and, in particular, shows that the sufficient condition for discontinuous phase transition in \citet[Theorem 5.11]{carrillo2020long} is one of several instances of creating such a bifurcation that `turns back'. This analysis also paints a geometric picture for such transitions visualized in~\cref{fig:kappa-quadratic-overlay}: the \emph{observed discontinuity} in the dynamics arising in the increasing $\kappa$ direction \emph{emerges, in fact, in a continuous fashion} through a bifurcation in the decreasing $\kappa$ direction. For the Noisy Transformer, \cref{thm:phasetr} (discussed below in~\cref{sec:122}) makes this
correspondence fully precise: the phase transition of the global
free energy minimizers is discontinuous if and only if the
first bifurcation from the uniform state is subcritical.

In \cref{sec:global}, we present some global properties of the energy landscape. As \emph{explicitly computable special cases} for the zeros of the map $F$ above, we present a global bifurcation picture involving Poisson kernels for the singular potential $W(\theta) = -\log(2\sin(\theta/2))$ (\cref{cor:explicitbif}) and a global `Kuramoto type' branch when $W$ has a finite number of non-zero Fourier modes, all positive (\cref{thm:finkur}). Finally, in \cref{thm:global}, we perform a detailed analysis of the \emph{free energy minimizing branch} of solutions and, among other things, show that \emph{discontinuous phase transitions are connected to points of non-differentiability of the minimum free energy map} $\kappa \mapsto m(\kappa) := \min_\rho \mathcal{F}(\kappa, \rho)$.

We remark here that, to cleanly demonstrate our main results, our main focus in this work is for the case of McKean-Vlasov equations on the circle. Nevertheless, our approach suggests extensions to
higher-dimensional tori $\mathbb{T}^d$ with translation-invariant interactions and spheres $\mathbb{S}^d$ with rotationally invariant interactions. In these settings, Fourier modes and spherical harmonics,
respectively, jointly diagonalize the Laplace--Beltrami operator and the interaction operator;
see \citep{carrillo2020long,shalova2024solutions}.

\subsubsection{Applications to noisy mean-field transformers}\label{sec:122} Having established the general Fourier-analytic theory, we next demonstrate its utility by applying it to the \emph{noisy mean-field transformer model}, an important modern instance of McKean–Vlasov dynamics. Recent works have focused on describing  signal propagation in attention-based transformer architectures in the zero-noise setup \citep{geshkovski2024dynamic, geshkovski2025mathematical,bruno2025emergence,chen2025quantitative,burger2025analysis,castin2025unified,bruno2025multiscale}. 
Our focus in this work is on the noisy setup, also considered recently in~\cite{shalova2024solutions}. 
Studying the noisy regime provides a principled first step to understand how randomness in the intermediate multi-layer perceptron layers of transformer models affects stability, phase transitions, and emergent collective behavior in the overall large-scale transformer dynamics. In particular, it provides insight into how feature representations stabilize across layers in large-depth transformer models. 

The Noisy mean-field Transformer potential is given by
\begin{align}\label{eq:NMFTpotential}
W_\beta(\theta) \coloneqq \frac{e^{\beta \cos\theta} - I_0(\beta)}{\beta} = \sum_{\ell=1}^\infty a_\ell(\beta)\cos(\ell\theta), \qquad \beta>0,
\end{align}
where we have the explicit Fourier coefficients
\begin{align*}
a_\ell(\beta) = \frac{2I_\ell(\beta)}{\beta}, \quad \ell \ge 1,\quad\text{with}\quad I_\ell(\beta)\coloneqq \frac{1}{2\pi}\int_0^{2\pi} e^{\beta \cos\phi} \cos(\ell\phi) \, \mathrm{d}\phi, \quad \ell \ge 0,
\end{align*}
denoting the modified Bessel function of the first kind.

By applying our general Fourier-analytic framework to this model, we uncover a rich phase structure dependent on the noise parameter $\beta$. In~\cref{sec:first_order} we show that the uniform solution in the noisy mean-field Transformer model loses stability when $\kappa > \beta / I_1(\beta)$, marking the first bifurcation point. More generally, there are infinitely many bifurcations at $\kappa_\ell^*(\beta) = \beta / I_\ell(\beta)$, each corresponding to the $\ell$-th Fourier mode, whose type (supercritical or subcritical) is governed by the $\ell$-signature of $W_\beta(\theta)$, i.e., the sign of 
$R_\ell(W_\beta(\theta)) \coloneqq R_\ell(\beta) = (I_\ell(\beta) - 2I_{2\ell}(\beta)) / (I_\ell(\beta) - I_{2\ell}(\beta))$. For small $\beta$, all bifurcations are supercritical (bend right). As $\beta$ increases, the first few bifurcations become subcritical and, for each fixed mode, the curvature in the original $\kappa$-parameter tends to zero from below as $\beta\to\infty$.

Another consequence of our general results, as shown in~\cref{rem:mettran}, is that in the large-$\beta$ regime, the Fourier coefficients of the Noisy Transformer potential, 
$a_\ell(\beta) = 2I_\ell(\beta)/\beta$, become increasingly clustered: for fixed $L$, 
the relative difference between $a_{l_1}(\beta)$ and $a_{l_2}(\beta)$ scales as 
$(\ell_2^2 - \ell_1^2)/(2\beta)$. Consequently, for $\beta \gtrsim L^2/(2\varepsilon)$, 
the first $L$ modes are nearly degenerate, and the bifurcation thresholds 
$\kappa_\ell^*(\beta) = \beta/I_\ell(\beta)$ accumulate exponentially fast as $\beta \to \infty$. 
This clustering motivates the study of nearly resonant multimode interactions and candidate metastable profiles, analogous to those observed in the noiseless Transformer~\citep{geshkovski2024dynamic,bruno2025emergence}. In \cref{rem:transformer_explicit}, we use the refined approximation described in \cref{thm:main_thm} to dive deeper into the structure of the bifurcating solutions. We exhibit how the primary mode of the stationary density interacts with the other modes and offer an interactive way to visualize the different possible shapes that emerge as $\beta, \kappa$ vary (see \cref{fig:denvis}). 

Finally, in~\cref{thm:phasetr}, we identify the exact threshold
$\beta_0\simeq2.447$, characterized by the unique solution to
$I_2(\beta_0)/I_1(\beta_0)=1/2$, separating continuous and
discontinuous phase transitions for the Noisy Transformer.
For $0<\beta\le\beta_0$, the transition is continuous and
occurs at
$\kappa_c(\beta)=\kappa^*(\beta)=\beta/I_1(\beta)$.
For $\beta>\beta_0$, it is discontinuous and occurs at
$\kappa_c(\beta)<\kappa^*(\beta)$.
Thus the threshold at which the first local bifurcation
changes from supercritical to subcritical also determines
exactly the change in continuity of the global phase
transition. At $\beta=\beta_0$, the transition remains
continuous, while the local branch satisfies
$\kappa(s)-\kappa^*(\beta_0)\sim Cs^4$, with $C>0$.
Consequently, the bifurcation is degenerate and supercritical and the bifurcating amplitude is of order
$(\kappa-\kappa^*)^{1/4}$.

\medskip
\textbf{Utility of adding `small noise' to Transformer dynamics:} Self-attention based Transformer dynamics are noiseless. The above discussion indicates the possible utility of adding noise, potentially via the multilayer perceptrons interspersed with attention in actual Transformer models, and tuning its strength by varying $\kappa$ (here, large $\kappa$ is equivalent to `small noise'). As noted in \cite{geshkovski2025mathematical}, all points eventually collapse on running the noiseless Transformer dynamics long enough. Without noise, the undulations in $W$ arising from different Fourier modes translates into a highly uneven landscape for the free energy functional, now without the entropy term, which gives rise to \emph{metastable states at several scales}. As a result, although there is eventual collapse, there could be intermediate clustering in intricate ways, depending on the initial density profile of the points, and a full understanding of this phenomenon is quite involved; see \cite{geshkovski2024dynamic,karagodin2024clustering,bruno2025emergence} for some results in this direction. Adding and tuning the noise helps \emph{regularize the energy profile at small enough scales while retaining its features at large enough scales}. This gives control on the \emph{expressivity} of the Transformer. Moreover, the noise \emph{prevents eventual collapse} and stationary states are `truly stationary'. Further, \cref{thm:phasetr} shows that, as $\kappa$ increases, the onset of clustering in the global free energy minimizers
is gradual for $0<\beta\le\beta_0$ and abrupt for $\beta>\beta_0$.

\subsection{Additional related works}

Among other related works, recently~\cite{bertoli2024stability} study related notions of stability when $W$ has a finite number of non-zero Fourier modes and present heuristic, numerical analyses for the existence of multimodal stationary states.
~\cite{cormier2022stability} provides conditions guaranteeing local exponential stability of invariant measures using Lions-derivative techniques and analytic-function/Fourier-based criteria in different settings, giving explicit, checkable tests that are immediately useful in applications such as neuronal mean-field models. \cite{vukadinovic2023phase} analyzes bifurcation thresholds for a two-component weakly-coupled Hodgkin–Huxley neurons and produces a finite-dimensional reduction that makes the bifurcation explicit in that application. 
~\cite{zhang2025local} gives a rigorous local bifurcation criterion for a large class of gradient-type McKean-Vlasov problems using regularized determinants and operator-theoretic tools; this provides a direct route to relate linearization spectra on function spaces to phase transition thresholds. 

Furthermore,~\cite{degond2015phase} study phase transitions and long-time behavior in kinetic models of self-propelled particles and alignment-type interactions; their analyses emphasize kinetic, as opposed to purely gradient, structure and connect phase transitions to changes in macroscopic order parameters and, in certain settings, hysteresis. These kinetic perspectives are relevant when one moves beyond overdamped gradient flows to retain inertial or alignment dynamics.

While finalizing this work, we became aware of the concurrent study by \cite{gerber2025formationclusterscoarseningweakly}, which investigates the dynamical behavior and coarsening phenomena in weakly interacting diffusions in a rich family of models that includes noisy Transformers. Their results are complementary to ours: whereas we investigate the structure and bifurcation of stationary solutions at the mean-field PDE level, their focus is primarily on the evolution and metastability of clustered states at the particle dynamics level.

\subsection{Solution spaces}\label{sec:solutinspace}
In this work, it will be convenient to consider weak solutions for~\eqref{eq:stat} in more general spaces. We denote by $L_s^2(\mathbb{S}^1)$ the space of square integrable functions $f$ on $\mathbb{S}^1$ such that $f(\theta - \pi) = f(\pi-\theta), \forall \theta \in [0, 2\pi)$ (even function). We consider solutions in the Sobolev space
\[
H_s^1(\mathbb{S}^1) := \mathcal{W}_s^{1,2}(\mathbb{S}^1)
    = \Bigl\{\, p : \mathbb{S}^1 \to \mathbb{R} \,\Big|\, 
        p \in L_s^2(\mathbb{S}^1),\, p' \in L^2(\mathbb{S}^1)
    \Bigr\}.
\]
consisting of square-integrable, even functions on the circle whose weak derivative is also square-integrable, equipped with the usual Sobolev norm $\|p\|_{H^1}^2 := \smallint_0^{2\pi} \big( p^2(\theta) + |p'(\theta)|^2 \big)\, d\theta$. The symmetry assumption is both natural and nonrestrictive: it corresponds to considering equilibria invariant under reflection, a property inherited from our assumption that $W$ is even. %The evenness assumption corresponds to reflection symmetry and will be preserved by all interaction potentials $W$ considered in this work.  
%Equipping $H_s^1(\mathbb{S}^1)$ with the usual Sobolev norm $\|p\|_{H^1}^2 := \smallint_0^{2\pi} \big( p^2(\theta) + |p'(\theta)|^2 \big)\, d\theta$, we obtain a Hilbert space setting in which weak solutions to~\eqref{eq:stat} are well defined and Fourier-analytic methods apply. 
Densities are normalized by $(2\pi)^{-1}\int_0^{2\pi}p=1$; hence the uniform density, denoted by $p_0$ when referring to a density, is the constant $1$.

\section{Local properties}\label{sec:local} 
\subsection{A Fourier representation}
The starting point of our analysis is the following equivalent representation of a solution $p$ to \eqref{eq:stat} purely in terms of Fourier modes of $p$ and $W$. This enables us to perform a more tractable and detailed analysis of stationary solutions on \emph{sequence spaces} in comparison to analysis on \emph{function spaces} in \cite{carrillo2020long}. The generality of the following theorem also helps us analyze stationary solutions for more singular potentials $W$.

\begin{theorem}\label{lem:lemma1}
Fix $\kappa>0$. Let $p$ be any solution to~\eqref{eq:stat} in $H_s^1(\mathbb{S}^1)$. Write 
\begin{align*}
    p(\theta) = 2\sum_{\ell=0}^\infty p_\ell\cos \ell\theta, \quad \theta \in [0,2\pi).
\end{align*}
Then, $\{ p_\ell: \ell \geq 0\}$ satisfy $p_0=1/2$ and for $\ell\geq 1$, 
\begin{align}\label{eq:Fid}
\ell(2-\kappa a_\ell) p_\ell = \kappa \sum_{j < \ell}  j a_j p_j p_{\ell-j} + \kappa \sum_{j > \ell} (ja_j - (j-\ell)a_{j-\ell})p_{j}p_{j-\ell}.
\end{align}
Conversely, any $p \in H_s^1(\mathbb{S}^1)$ whose Fourier coefficients satisfy the above is a solution to~\eqref{eq:stat}. Moreover, any such solution satisfies the fixed point equation
$$
p(\theta) = \frac{1}{Z}\exp\left(\kappa \sum_{\ell=1}^\infty p_\ell a_\ell \cos \ell\theta\right), \ \theta \in \mathbb{S}^1,
$$
where $Z = \frac{1}{2\pi}\int_0^{2\pi}\exp\left(\kappa \sum_{\ell=1}^\infty p_\ell a_\ell \cos \ell\psi\right) \d\psi$.
\end{theorem}
Since $p$ will always satisfy $\frac{1}{2\pi}\int_0^{2\pi}p(\cdot,\theta)\d\theta = 1$, we suppress the zeroth mode and simply call $\underline{p} = \left(p_\ell := \frac{1}{2\pi}\int_0^{2\pi}p(\cdot,\theta)\cos \ell\theta \d\theta\right)_{\ell \in \mathbb{N}}$ the Fourier modes of $p$.
Observe that the Fourier modes of $p \in H_s^1(\mathbb{S}^1)$ correspond exactly with the weighted $\ell^2$ space
$$
\ell^2_w := \{\underline{x} \in \mathbb{R}^\infty: \sum_{\ell=1}^\infty (1 + \ell^2)x_\ell^2 < \infty\}.
$$
Suppose $\sup_{\ell \ge 1} \ell |a_\ell| < \infty$. The above theorem states that $p \in H_s^1(\mathbb{S}^1)$ is a stationary solution for $\kappa \in \mathbb{R}_+$ if and only if the Fourier modes $\underline{p} \in \ell^2_w$ of $p$ satisfy 
\begin{equation}
    \label{eq:F=0}
    F(\underline{p}, \kappa)=\underline{0},
\end{equation}
where $F: \ell^2_w \times \mathbb{R}_+ \rightarrow \ell^2$ is given by
\begin{align*}
    F_{\ell}(\underline{x}, \kappa) = \ell(2-\kappa a_\ell) x_\ell - \kappa \sum_{j < \ell}  j a_j x_j x_{\ell-j} - \kappa \sum_{j > \ell} (ja_j - (j-\ell)a_{j-\ell})x_jx_{j-\ell}, \ \ell \ge 1.
\end{align*}
The well-definedness, regularity and derivatives of $F$ are recorded in the following result.
\begin{lemma}\label{lem:Freg}
     Suppose $\sup_{\ell \ge 1} \ell |a_\ell| < \infty$. Then $F: \ell^2_w \times \mathbb{R}_+ \rightarrow \ell^2$ is a $C^\infty$ function. The Fr\'echet derivatives of $F$ are given by
     \begin{align*}
         \left[D_{\underline{p}}F(\underline{x}, \kappa)\underline{h}\right]_\ell &= \ell(2-\kappa a_\ell) h_\ell - \kappa \sum_{j < \ell}  j a_j (x_j h_{\ell-j} + h_jx_{\ell-j})\\
         &\qquad - \kappa \sum_{j > \ell} (ja_j - (j-\ell)a_{j-\ell})(x_jh_{j-\ell} + h_jx_{j-\ell}),\\
         \left[D^2_{\underline{p}\underline{p}}F(\underline{x}, \kappa)[\underline{h},\underline{k}]\right]_\ell &= - \kappa \sum_{j < \ell}  j a_j (h_j k_{\ell-j} + k_jh_{\ell-j}) - \kappa \sum_{j > \ell} (ja_j - (j-\ell)a_{j-\ell})(h_j k_{j - \ell} + k_jh_{j- \ell}),\\
         \left[D^2_{\underline{p}\kappa}F(\underline{x}, \kappa)\underline{h}\right]_\ell &= -\ell a_\ell h_\ell - \sum_{j < \ell}  j a_j (x_j h_{\ell-j} + h_jx_{\ell-j}) - \sum_{j > \ell} (ja_j - (j-\ell)a_{j-\ell})(x_jh_{j-\ell} + h_jx_{j-\ell}),
     \end{align*}
     where $\ell \in \mathbb{N}$ and $\underline{h},\underline{k} \in \ell^2_w$.
\end{lemma}

We now define the notions of bifurcation point and bifurcation curve.

\begin{definition}
A point $(\underline{0},\kappa_b)$ is called a bifurcation point of $F$ from the uniform branch if
$
(\underline{0},\kappa_b)\in
\overline{\{(\underline{p},\kappa)\in\ell^2_w\times\mathbb{R}_+:
F(\underline{p},\kappa)=\underline{0},\ \underline{p}\ne\underline{0}\}}.
$
 
    Further, if there exists $\delta>0$ and a continuous curve $s \mapsto (\underline{p}(s),\kappa(s))$ mapping $(-\delta, \delta)$ to $\ell^2_w \times \mathbb{R}_+$ such that $(\underline{p}(0), \kappa(0)) = (\underline{0}, \kappa_b)$ and $F(\underline{p}(s),\kappa(s)) = 0$ for all $s \in (-\delta, \delta)$, then this curve is called a bifurcation curve. If the bifurcation curve meets the $\kappa$-axis $\{(\underline{0}, \kappa) : \kappa \in \mathbb{R}_+\}$ at at most one $s$, we call it a non-trivial branch of solutions.
\end{definition}

Clearly, $\underline{p} \equiv \underline{0}$ is a solution for all $\kappa$. This corresponds to $p_0 \coloneqq \text{Unif}(\mathbb{S}^1)$, which is always a stationary solution. We call $\kappa \mapsto (\underline{0}, \kappa)$ the `trivial' branch of stationary solutions.

\subsection{First order bifurcations} \label{sec:first_order}The next result characterizes the bifurcation points of $F$ about this trivial branch that are of first order, namely, they are related to a unique Fourier mode of $W$.
\begin{theorem}\label{thm:CRthm}
    Suppose $\sup_{\ell \ge 1} \ell |a_\ell| < \infty$. Assume, for some $\ell^* \in \mathbb{N}$, $a_{\ell^*} >0$ and $a_\ell \neq a_{\ell^*}$ for all $\ell \neq \ell^*$. Then, $\kappa^* := 2/a_{\ell^*}$ is a \emph{pitchfork} bifurcation point of $F$ from the uniform solution branch in the sense that there exists $\delta>0$ 
    %and an open neighborhood $U$ containing $\underline{0}$ in $\ell^2_w$
    and a non-trivial branch of solutions $(\underline{p}(s), \kappa(s)) : (-\delta, \delta) \rightarrow  \ell^2_w \times \mathbb{R}_+$ such that
    \begin{align*}
        \underline{p}(s) = se_{\ell^*} + o(s),
    \end{align*}
where $\{e_j\}_j$ is the canonical basis of $\ell^2$, and $\kappa(\cdot)$ is twice differentiable and satisfies $\kappa(0) = \kappa^*, \kappa'(0) = 0$ and
    \begin{align}\label{eq:curvature}
        \kappa''(0) = \frac{2}{a_{\ell^*}}\cdot R_{\ell^*}(W) = \frac{2}{a_{\ell^*}}\cdot \frac{a_{\ell^*} - 2a_{2\ell^*}}{a_{\ell^*} - a_{2\ell^*}}.
    \end{align}
    Moreover, $(\underline{p}(s), \kappa(s))$ comprise the only even non-trivial solution to~\eqref{eq:F=0} in a neighborhood of $(\underline{0},\kappa^*)$ in $\ell^2_w \times \mathbb{R}_+$. Further, if $\kappa \notin \{ 2/a_\ell\,:\, a_\ell>0, \, \ell \in \mathbb{N}\}$, then $\kappa$ cannot be a bifurcation point. 
    \end{theorem}

\begin{figure}[t]
\centering
\begin{tikzpicture}
\begin{axis}[
    width=8cm,
    height=8cm,
    axis lines=center,
    axis line style={gray!30},
    xlabel={s},
    ylabel={$\kappa - \kappa^*$},
    xmin=-2.5, xmax=2.5,
    ymin=-2, ymax=2,
    xtick=\empty,
    ytick=\empty,
    samples=100,
    thick,
]

% --- Subcritical (unchanged) ---
\addplot[blue, solid] coordinates {(0,0) (0,2)};
\addplot[blue, solid, domain=-2:0] {-x^2};
\addplot[blue, solid, domain=0:2]  {-x^2};

% --- Supercritical (unchanged) ---
\addplot[red, dotted, very thick] coordinates {(0,-2) (0,0)};
\addplot[red, dotted, very thick, domain=-2:2] {x^2};

% --- Transcritical (rotated S): (\kappa-\kappa^*) = sgn(s)*|s|^{1/3} ---
\addplot[black!60!green, dashdotdotted, very thick, domain=-2.5:2.5, samples=200]
  { (x<0 ? -pow(-x,1/3) : (x>0 ? pow(x,1/3) : 0)) };

\end{axis}
\end{tikzpicture}
\caption{Local pitchfork (i.e., subcritical and supercritical) and transcritical bifurcation types for the McKean-Vlasov stationary equation near the trivial solution at $(p,\kappa)=(0,\kappa^*)$: Solid (blue) line corresponds to the subcritical case. Dashed (red) line corresponds to the supercritical case. Dashed and dotted (dark green) line corresponds to the transcritical case. The vertical axis is the dimensionless deviation $\kappa-\kappa^*$ (zero at the bifurcation point) and the horizontal axis is the modal amplitude $s \;=\; \langle e_{\ell^*},\, p\rangle$. The vertical axis corresponds to the uniform solution $p_0$. The transcritical branch has finite nonzero slope at the origin.}
\label{fig:bifurcation}
\end{figure}
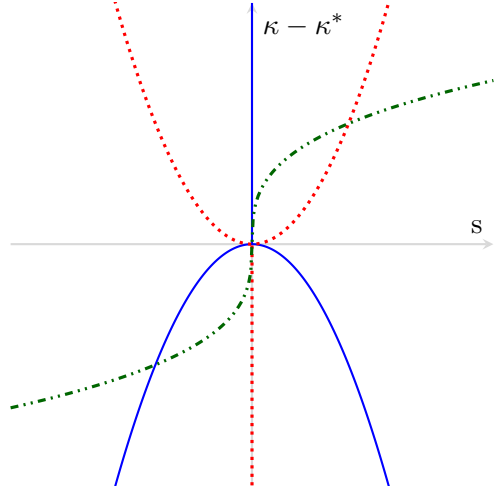

    \begin{remark}\label{rem:CRrem}
        \cref{thm:CRthm} closely resembles \citet[Theorem 4.2 and Corollary 4.3]{carrillo2020long}. However, there are several key differences, as highlighted below.
        \begin{itemize}
            \item[(i)]\citet[Theorem 4.2 and Corollary 4.3]{carrillo2020long} study bifurcation points and associated solutions for the equation $F_\textsc{cgps}(\rho,\kappa) = 0$, where $F_\textsc{cgps}: L_s^2(\mathbb{S}^1) \times \mathbb{R}_+ \rightarrow L_s^2(\mathbb{S}^1)$ corresponds to the \emph{fixed point map} for the associated McKean-Vlasov equation given by
            \begin{align}\label{eq:CGPSfpe}
                F_\textsc{cgps}(\rho,\kappa) := \rho - \frac{1}{Z}e^{\kappa W\star \rho},
            \end{align}
            where $Z := \frac{1}{2\pi}\int_0^{2\pi}e^{\kappa (W\star \rho)(s)}\d s$ is the normalizing factor. Although their results exhibit bifurcations for $F_\textsc{cgps}$ when $W \in  L_s^2(\mathbb{S}^1)$, it can be directly related to non-trivial solution branches of the stationary McKean-Vlasov equation \eqref{eq:stat} only when $W$ has higher regularity, namely, $W \in H_s^1(\mathbb{S}^1)$ (see \citet[Theorem 2.3 and Proposition 2.4]{carrillo2020long}). For the Keller-Segel model treated in \citet[Section 6.5]{carrillo2020long}, which has a potential not in $H_s^1(\mathbb{S}^1)$, a case-specific argument is required. In comparison, the bifurcation theory for $F$, treated in \cref{thm:CRthm}, is in one-to-one correspondence with that for solutions of \eqref{eq:stat} for any $W$ satisfying $\sup_{\ell \ge 1} \ell |a_\ell| < \infty$, which allows for more singular $W$ in $L_s^2(\mathbb{S}^1)$. Moreover, the conclusion that there are no bifurcation points other than the ones of the form $2/a_\ell, \, \ell \in \mathbb{N}$, carries over to solutions of \eqref{eq:stat} via this equivalence to zeros of $F$.\\
            \item[(ii)] $F_\textsc{cgps}$ is defined on a function space and involves exponential mappings, which makes the associated calculus quite intricate and prone to errors. In comparison, $F$ is a quadratic function on a sequence space which makes it significantly easier to analyze. In \citet[Theorem 4.2]{carrillo2020long}, working with  Fr\'echet derivatives of $F_\textsc{cgps}$ led to an error in computing $\kappa''(0)$, the curvature of the bifurcation curve (\citet[Equation (I.6.11)]{kielhofer2012bifurcation} was erroneously applied), which gave a curvature $\kappa''(0) = 4\pi\kappa^*/3$ that is always positive irrespective of the choice of $W$. \cref{eq:curvature} gives the correct curvature. We see that the bifurcation can be subcritical (left-sided), critical (vertical in the sense $\kappa''(0)=0$) or supercritical (right-sided) depending on the behavior of the quantity $\frac{a_{\ell^*} - 2a_{2\ell^*}}{a_{\ell^*} - a_{2\ell^*}}$. When $R_{\ell^*}(W)=0$, only the quadratic curvature vanishes; higher-order bending is not determined by this formula. This will also be useful later in quantifying discontinuous phase transitions (see \cref{sec:dpt}).\footnote{After submitting this article to arXiv, we became aware of \cite{carrillo2026long}, which independently corrects the bifurcation curvature formula in \cite{carrillo2020long}.}
        \end{itemize}
    \end{remark}
\subsubsection{Specialization to noisy mean-field Transformers}\label{sec:221} 
We now specialize to the \emph{noisy mean-field Transformer}. Recall that this model employs the potential $W_\beta$ defined in~\eqref{eq:NMFTpotential}. From results in \cite{gates1970van} (see also \citet[Proposition 2.9]{chayes2010mckean}), it follows that the uniform solution loses stability when \(\kappa > 2/a_1 = \beta/I_1(\beta)\). This also corresponds to the first bifurcation point.
It can be checked that the map $\ell \mapsto I_\ell(\beta)$ is strictly decreasing for any $\beta>0$. Thus, substituting the corresponding coefficients into~\cref{thm:CRthm} yields that, 
for any \(\ell \in \mathbb{N}\), there is a bifurcation point at
\[
    \kappa_{\ell}^*(\beta) = \frac{2}{a_{\ell}(\beta)} 
    = \frac{\beta}{I_{\ell}(\beta)},
\]
corresponding to the $\ell$-th Fourier mode of $W_\beta$. This fact is likewise observed for Noisy Transformers on the sphere in \citet[Proposition~6.1]{shalova2024solutions}.

It can be checked that $\lim_{\beta \to 0}\frac{\beta}{I_1(\beta)} = 2$ and $\lim_{\beta \to 0}\frac{\beta}{I_\ell(\beta)} = \infty$ for $\ell \ge 2$, and $\lim_{\beta \to \infty}\frac{\beta}{I_\ell(\beta)} = 0$ for all $\ell \in \mathbb{N}$. Thus, as $\beta \to 0$, the first bifurcation point approaches $2$ and the subsequent points diverge. This is consistent with the fact that, in the Kuramoto model with potential $W(\theta) = \cos \theta = \lim_{\beta \to 0}W_\beta(\theta)$, there is a unique bifurcation point.

Using the following Laplace/large-\(\beta\) expansion for fixed $\ell \in \mathbb{N}$,
\begin{align}\label{largebetaIbeta}
I_\ell(\beta)=\frac{e^\beta}{\sqrt{2\pi\beta}}\Big(1-\frac{4\ell^2-1}{8\beta}+O(\beta^{-2})\Big),
\qquad \beta\to\infty,
\end{align}
we conclude that, for large $\beta$,
$$
\kappa_{\ell}^*(\beta) =\frac{\beta}{I_\ell(\beta)} = \sqrt{2\pi}\beta^{3/2}e^{-\beta}(1+ O(\beta^{-1})), \qquad \beta\to\infty.
$$
Thus, as $\beta \to \infty$, all the bifurcation points accumulate near $0$.

%This fact was recently also noted for Noisy Transformers on the sphere in \citet[Proposition 6.1]{shalova2024solutions}.

By~\cref{thm:CRthm}, the curvature of the bifurcating branch at the $\ell$-th bifurcation point is given explicitly by
\[
    \kappa''(0) 
    = \frac{\beta\bigl(I_{\ell}(\beta) - 2 I_{2\ell}(\beta)\bigr)}
    {I_{\ell}(\beta)\bigl(I_{\ell}(\beta) - I_{2\ell}(\beta)\bigr)}.
\]
Denoting $R_\ell(W_\beta(\theta))$ by $R_\ell(\beta)$ for simplicity, the bifurcation type is thus determined by the sign of
\[
R_\ell(\beta) \;=\; \frac{a_\ell-2a_{2\ell}}{a_\ell-a_{2\ell}}
\;=\;
\frac{I_\ell(\beta)-2I_{2\ell}(\beta)}{I_\ell(\beta)-I_{2\ell}(\beta)}.
\]
Note that $R_\ell(\beta)$ is a strictly decreasing function of $\frac{I_{2\ell}(\beta)}{I_\ell(\beta)}$.  From Tur\'an-type inequalities (see \cite{baricz2010turan}),
$$
I_k^2(\beta) \ge I_{k-1}(\beta)I_{k+1}(\beta)
$$
for all $k \in \mathbb{N}$ and $\beta>0$. This implies that the map $k \mapsto \frac{I_{k+1}(\beta)}{I_k(\beta)}$ is decreasing in $k$. Hence,
$$
\frac{I_{2\ell}(\beta)}{I_\ell(\beta)} = \prod_{k=\ell}^{2\ell-1}\frac{I_{k+1}(\beta)}{I_k(\beta)}
$$
is decreasing in $\ell$. Combining these observations, we conclude that $\ell \mapsto R_\ell(\beta)$ is increasing in $\ell$ for every $\beta>0$. Moreover, as $\beta \mapsto \frac{I_{2\ell}(\beta)}{I_\ell(\beta)}$ is strictly increasing in $\beta$, which can again be verified using Tur\'an-type inequalities, we have $\beta \mapsto R_\ell(\beta)$ is strictly decreasing in $\beta$ for fixed $\ell \in \mathbb{N}$. 
Note that $\lim_{\beta \to 0}\frac{I_{2\ell}(\beta)}{I_\ell(\beta)} = 0$, and hence $\lim_{\beta \to 0} R_\ell(\beta) = 1$, for all $\ell \in \mathbb{N}$. Moreover, using \eqref{largebetaIbeta}, we conclude for any $\ell \in \mathbb{N}$,
$$
R_\ell(\beta) = -\frac{2\beta}{3\ell^2} + O(1) \rightarrow -\infty, \qquad \text{ as }  \beta\to\infty.
$$
See~\cref{fig:rellbeta} for a visualization.

\begin{figure}[t]
    \centering
    \includegraphics[width=0.5\linewidth]{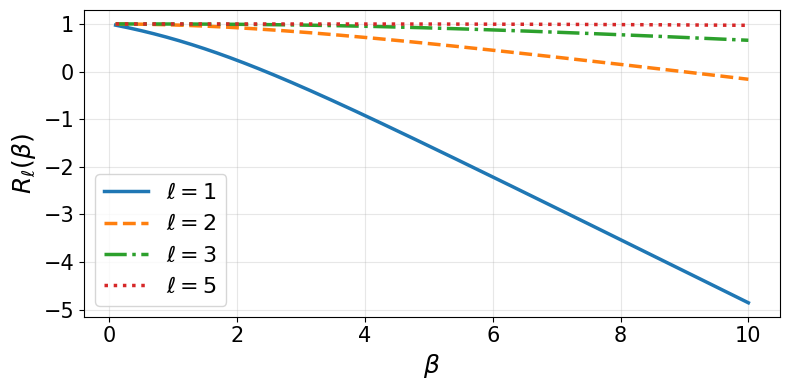}
    \caption{Plot of $R_\ell(\beta)$ vs. $\beta$: For fixed $\ell$, $R$ is strictly decreasing in $\beta$ and for every $\beta$, $R$ is increasing in $\ell$. }
    \label{fig:rellbeta}
\end{figure}

Thus, for any $\ell \in \mathbb{N}$, there exists a unique $\beta_\ell>0$ such that the associated bifurcation at $\kappa^*_\ell(\beta)$ is supercritical for $\beta< \beta_\ell$, critical for $\beta=\beta_\ell$ and subcritical for $\beta>\beta_\ell$. For $\ell=1$, $\beta_1 \approx 2.447$; see also \cref{thm:phasetr}. For $\beta < \beta_1$, all bifurcations are supercritical. For $\beta> \beta_1$, the first few bifurcations are subcritical, followed by possibly some critical ones and the subsequent bifurcations are all of supercritical type.

Further, for fixed $\beta>0$, $\frac{\beta}{I_\ell(\beta)} \to \infty$ and $\frac{I_{2\ell}(\beta)}{I_\ell(\beta)} \to 0$, and hence $R_\ell(\beta) \to 1$, as $\ell \to \infty$. Thus, the curvature of subsequent bifurcations increases to infinity with $\ell$. 

Recalling $\lim_{\beta \to 0} R_\ell(\beta) = 1$, for all $\ell \in \mathbb{N}$, we observe that, as $\beta \to 0$, the curvature at the first bifurcation point approaches $2$ and that at subsequent ones approaches $\infty$.

Using the large-$\beta$ approximations for $\kappa_{\ell}^*(\beta)$ and $R_\ell(\beta)$ above,
$$
\lim_{\beta \to \infty} \kappa''(0) = \lim_{\beta \to \infty} \kappa_{\ell}^*(\beta)R_\ell(\beta) = 0.
$$
Thus, for each fixed mode, the curvature in the original $\kappa$-parameter tends to zero from below as $\beta\to\infty$, while the normalized curvature satisfies $R_\ell(\beta)\to-\infty$.

%For small \(\beta\) one typically has \(I_m(\beta)\gg I_{2m}(\beta)\) and hence supercritical behavior at all modes $m \in \mathbb{N}$; for larger \(\beta\) the first few modes are subcritical and the bifurcations eventually become all of supercritical type.

%From standard results on Bessel functions, it follows that, for any $\ell \in \mathbb{N}$ the map $\beta \mapsto \frac{I_{2\ell}(\beta)}{I_\ell(\beta)}$ is strictly increasing in $\beta>0$ and $\lim_{\beta \to 0}\frac{I_{2\ell}(\beta)}{I_\ell(\beta)} = 0$, $\lim_{\beta \to \infty}\frac{I_{2\ell}(\beta)}{I_\ell(\beta)} = 1$. Thus, there exists a unique $\beta_\ell>0$ such that the associated bifurcation at $\kappa^*_\ell(\beta)$ is supercritical for $\beta< \beta_\ell$, critical for $\beta=\beta_\ell$ and subcritical for $\beta>\beta_\ell$. For $\ell=1$, $\beta_1 \approx 2.447$ (see \cref{thm:phasetr}).

%Hence, there are infinitely many pitchfork bifurcation points at \(\kappa_\ell^*(\beta) = \beta/I_\ell(\beta)\), $\ell \in \mathbb{N}$, and the sign of \(\kappa''(0)\) determines whether the pitchfork is supercritical or subcritical.

    \subsection{Periodicity} The next result gives periodicity properties of non-trivial solutions.
    \begin{theorem}\label{thm:period}
        Suppose $\sup_{\ell \ge 1} \ell |a_\ell| < \infty$. 
        
        (a) Assume, for some $m \in \mathbb{N}$, $a_{m} >0$ and $a_{\ell m} \neq a_{m}$ for all $\ell >1$. Then, there exists a non-trivial \textbf{$m$-periodic} branch of solutions $(\underline{p}(s), \kappa(s)) : (-\delta, \delta) \rightarrow  \ell^2_w \times \mathbb{R}_+$ at bifurcation point $2/a_m$ that has the following description. Let $W^{(m)}(\theta) := \sum_{\ell=1}^\infty a_{\ell m} \cos \ell\theta, \ \theta \in [0,2\pi),$ and consider the unique non-trivial branch of solutions $(\underline{p}^{(m)}(s), \kappa^{(m)}(s)) : (-\delta^{(m)}, \delta^{(m)}) \rightarrow  \ell^2_w \times \mathbb{R}_+$ at bifurcation point $\kappa_m := 2/a_m$ for \eqref{eq:stat} with potential $W^{(m)}$ whose existence is guaranteed by \cref{thm:CRthm}. Then, $\delta = \delta^{(m)}$ and for any $s \in (-\delta, \delta)$,
        \begin{align*}
            \kappa(s) &= \kappa^{(m)}(s),\\
            p_j(s) &= 0 \ \forall \ j \ \text{ such that } m \nmid j,\\
            p_{\ell m}(s) &= p^{(m)}_{\ell}(s), \ \ell \in \mathbb{N}.
        \end{align*}
        The above relations connect the associated bifurcating branches of stationary solutions with potentials $W$ and $W^{(m)}$ in terms of their Fourier coefficients. In particular, if $a_\ell \neq a_{m}$ for all $\ell \neq m$, this gives the unique non-trivial branch of solutions at bifurcation point $2/a_m$.

        (b) Assume that, for some $m \in \mathbb{N}$, $a_{m} >0$. 
        %and, in addition, $$\operatorname{gcd}\{\ell \le m : a_\ell = a_m\} = g>1.$$ 
        Let $S:= \{\ell \in \mathbb{N} : a_\ell = a_m\}$.
        Further, assume $\sum_{j=1}^\infty j^2a_j^2 < \infty$ (that is, $W \in H^1_s(\mathbb{S}^1)$) and
        $$
        \operatorname{gcd}(S) = g>1.
        $$
        Then, for any non-trivial branch of solutions $(\underline{p}(s), \kappa(s)) : (-\delta, \delta) \rightarrow  \ell^2_w \times \mathbb{R}_+$ around the bifurcation point $2/a_m$, we must necessarily have $\delta' \in (0,\delta]$ such that
        $$
        p_{\ell}(s) = 0 \ \forall s \in (-\delta', \delta') \ \text{ if } \ g \nmid \ell.
        $$
        In other words, the solution is \textbf{$g$-periodic}.
    \end{theorem}

\begin{remark}\label{rem:singpot}
The assumption $\sup_{\ell\ge1} \ell |a_\ell| < \infty$ ensures that the Fourier coefficients decay at least as $O(1/\ell)$, implying that 
$$
W \in H_s^{\alpha}(\mathbb{S}^1) := \{f(\theta) = \sum_{\ell = 1}^\infty f_\ell \cos \ell \theta : \sum_{\ell=1}^\infty (1 + \ell^2)^\alpha f_\ell^2 < \infty\}
$$ 
for every $\alpha < \tfrac{1}{2}$, but it does \emph{not} require continuity or boundedness of $W$. 
For example, consider the potential
\[
    W(\theta) = \sum_{\ell=1}^\infty \frac{\cos(\ell\theta)}{\ell}
    = -\log\!\big(2\sin(\theta/2)\big), \ \theta \in [0, 2\pi),
\]
which we revisit in~\cref{sec:global} in more detail. It satisfies $\sup_{\ell \ge 1}\ell|a_\ell|<\infty$ but diverges logarithmically at $\theta = 0$. 
In particular, this assumption is strictly weaker than the Zygmund class condition~(\citet[Chapter~V]{zygmund2002trigonometric}), 
which additionally requires uniform control on the discrete derivative $a_{\ell+1}-a_\ell$. 

%\textcolor{blue}{Formally, one may define the weak derivative $W'(\theta) = -\sum_{\ell\ge1} \ell a_\ell \sin(\ell\theta)$, which exists as a periodic distribution belonging to $\mathrm{BMO}_{\mathrm{per}}([0,2\pi))$, 
%the space of functions of \emph{bounded mean oscillation}:
%\[
%    \|f\|_{\mathrm{BMO}} 
%    := \sup_{I \subset [0,2\pi)} 
%    \frac{1}{|I|} \int_I |f(x) - f_I|\,dx < \infty,
%    \qquad f_I := \frac{1}{|I|}\int_I f(x)\,dx.
%\]
%Intuitively, functions in $\mathrm{BMO}$ have uniformly bounded local mean deviations, 
%so the condition $\sup_{\ell}\ell|a_\ell|<\infty$ may be interpreted as a bounded-oscillation constraint on the weak derivative $W'$. 
%Consequently,
%\[
%    W \in \bigcap_{\varepsilon > 0} H^{\frac{1}{2} - \varepsilon}(\mathbb{S}^1)
%    \quad \text{but} \quad 
%    W \notin H^{\frac{1}{2}}(\mathbb{S}^1),
%\]
%i.e., $W$ has Sobolev regularity strictly below the critical exponent $s = {1}/{2}$.}
\end{remark}

    \begin{remark}\label{rem:per}
    (a) Suppose $a_1>a_2>\dots>0$. By \cref{thm:period}(a), as the interaction intensity $\kappa$ grows (equivalently, if the noise becomes smaller), the periodicity of the non-trivial stationary solutions at subsequent bifurcation points increases. These profiles are candidates for a separate metastability analysis; persistence times require additional dynamical estimates. 

    (b) The hypotheses of \cref{thm:period}(b) allow for the null-space of the derivative map $D_{\underline{p}}F(\underline{0}, \frac{2}{a_m})$ to have (finite) dimension greater than one, leading to possibly multiple bifurcating non-trivial solutions around $\kappa = \frac{2}{a_m}$. \cref{thm:period}(b) says that, provided $g>1$, any such solution is $g$-periodic. \cref{thm:period}(a) exhibits the existence of an $m$-periodic solution.
\end{remark}

\subsection{Higher order bifurcations}

The next result investigates bifurcations at points corresponding to multiple identical Fourier modes of $W$, which result in bifurcating stationary solutions with possibly multiple dominant Fourier modes. Even when multiple Fourier modes of $W$ are not equal but sufficiently close to each other, this result can be used to provide an informal explanation for how this recipe creates `approximate' stationary solutions as candidates for metastable states; see \cref{rem:metahigh}.

\begin{theorem}\label{thm:high}
Assume $\sup_{\ell \ge 1} \ell |a_\ell| < \infty$.

    (a) (No resonance case) Suppose that there exists $a>0, k \in \mathbb{N}$ and a finite set $\Lambda = \{\ell_1, \dots, \ell_k\} \subset \mathbb{N}$ with $\ell_1 < \ell_2 < \dots < \ell_k$ such that $a_j = a$ for all $j \in \Lambda$ and $a_j <a$ for all $j \notin \Lambda$. Moreover, assume $\Lambda$ satisfies the following: 
    \begin{itemize}
        \item[(i)] $\ell_i + \ell_j \notin \Lambda$ for any $1 \le i,j\le k$ (`no (destructive) resonance' assumption);
        \item[(ii)] For any $1 \le i,j, r \le k$, $\ell_i + \ell_j + \ell_r \notin \Lambda$;
        \item[(iii)] For any $1 \le i,j, r, t \le k$, $\ell_i + \ell_j = \ell_r + \ell_t$ if and only if $\{i,j\} = \{r,t\}$.
    \end{itemize}
    (In the above, $i,j,r,t$ need not be distinct.)
    
 Define the matrix for $1 \le i,j \le k$:
    \begin{align*}
    B_{\ell_i \ell_i} &:= \frac{a-2a_{2\ell_i}}{a-a_{2\ell_i}},\\ 
    B_{\ell_i\ell_j} &:= \frac{2}{\ell_i}\left[\frac{\ell_j a - (\ell_i + \ell_j)a_{\ell_i + \ell_j}}{a - a_{\ell_i + \ell_j}} + \frac{(\ell_j - \ell_i)a_{\ell_j - \ell_i} - \ell_j a}{a-a_{\ell_j - \ell_i}}\right], \quad i < j,\\
    B_{\ell_i\ell_j} &:= \frac{2}{\ell_i}\left[\frac{\ell_j a - (\ell_i + \ell_j)a_{\ell_i + \ell_j}}{a - a_{\ell_i + \ell_j}} - \frac{(\ell_i - \ell_j)a_{\ell_i - \ell_j} + \ell_j a}{a-a_{\ell_i - \ell_j}}\right], \quad i > j.
    \end{align*}
    Suppose, for some $m \le k$, there exists a sub-collection of indices $\{i_1,\dots, i_m\} \subset \{1,\dots,k\}$ such that the submatrix $\tilde{B}_{rt} := B_{\ell_{i_r} \ell_{i_t}}, 1 \le r,t \le m$, is invertible and all the entries of $\tilde{B}^{-1}\mathbf{1}$ are non-zero and of the same sign $\operatorname{sgn}\left(\tilde{B}^{-1}\mathbf{1}\right)$. Assume in addition that
    \begin{equation}\label{eq:nondege}
    \sum_{r=1}^m B_{\ell\ell_{i_r}}
    \left(\tilde{B}^{-1}\mathbf{1}\right)_r \neq 1,
    \qquad \ell \in \Lambda \setminus
    \{\ell_{i_1},\dots,\ell_{i_m}\}.
    \end{equation}
    (This condition is vacuous when $m=k$.)
    
    Denote by $b_r$ the square root of the absolute value of the $r$-th entry of $\tilde{B}^{-1}\mathbf{1}$. Then there is $\delta>0$ and a non-trivial branch of solutions $(\underline{p}(s), \kappa(s)) : (-\delta, \delta) \rightarrow  \ell^2_w \times \mathbb{R}_+$ around $\kappa = 2/a$ such that
    \begin{align*}
        \underline{p}(s) = s\sum_{r=1}^m b_r e_{\ell_{i_r}} + O(s^2), \qquad \kappa(s) = \frac{2}{a}\left(1 + \operatorname{sgn}\left(\tilde{B}^{-1}\mathbf{1}\right)\frac{s^2}{2}\right).
    \end{align*}
    %where $b: (-\delta, \delta) \rightarrow \mathbb{R}_+^m$ is a continuous curve with $b_r(0) = b_r$.
    We also have $\underline{p}'(s) = \sum_{r=1}^m b_r e_{\ell_{i_r}} + O(s)$, for $s \in (-\delta, \delta)$.
    In particular, all such bifurcations are of \textbf{pitchfork type} ($\kappa(\cdot)$ convex or concave with $\kappa'(0)=0$). The bifurcation is supercritical if $\operatorname{sgn}\left(\tilde{B}^{-1}\mathbf{1}\right) = +1$ and subcritical if $\operatorname{sgn}\left(\tilde{B}^{-1}\mathbf{1}\right) = -1$. %Moreover, all non-trivial bifurcations around $\kappa = 2/a$ are obtained through this recipe.

    (b) (Resonance case) Suppose there exists $a>0$ and $l,m \in \mathbb{N}$, with $m < l$ and $l \neq 2m$, such that $a_l = a_m = a_{l+m} = a$ and $a_j <a$ for all $j \neq l,m, l+m$. Then we have four non-trivial bifurcating branches around $\kappa = 2/a$ given by $(\underline{p}^{\sigma}(s), \kappa(s)) : (-\delta, \delta) \rightarrow  \ell^2_w \times \mathbb{R}_+$, where $\sigma = (\sigma_1, \sigma_2, \sigma_3)$ with $\sigma_i \in \{+1,-1\}$ for $i=1,2,3,$ with $\sigma_1 \sigma_2\sigma_3 = 1$, with the following description. For each $\sigma$, %there exists a continuous curve $u^\sigma: (-\delta, \delta) \rightarrow \mathbb{R}^3$ with $u^\sigma(0) = \sigma/2$ such that
    \begin{align*}
        \underline{p}^\sigma(s) = s\left(\sigma_1 e_l + \sigma_2 e_m + \sigma_3 e_{l+m}\right) + O(s^2), \qquad \kappa(s) = \frac{2}{a}(1- s), \quad s \in (-\delta, \delta).
    \end{align*}
    We also have $\underline{p}'(s) = \sigma_1 e_l + \sigma_2 e_m + \sigma_3 e_{l+m} + O(s)$, for $s \in (-\delta, \delta)$.
    In particular, such bifurcating solutions are \textbf{transcritical}, that is, $\kappa'(0) \neq 0$.
    %Moreover, for any transcritical bifurcation around $\kappa = 2/a$, we must necessarily have resonance, namely, $l,m \in \mathbb{N}$, with $m \le l$ and $l \neq 2m$, such that $a_\ell = a_m = a_{l+m}$.
\end{theorem}

\begin{remark}\label{rem:high}
In the setup of \cref{thm:high}(a), we make the following observations. In (i)--(iii), condition \eqref{eq:nondege} is required for each chosen proper subset of $\Lambda$.

    (i) For $\ell_i \in \Lambda$ such that $B_{\ell_i\ell_i} \neq 0$, we can take $\tilde{B}$ to be the corresponding $1 \times 1$ submatrix and obtain non-trivial unimodal bifurcations (locally $\ell_i$ gives the only dominant mode in the Fourier expansion of $p$) for such $\ell_i$, around $\kappa = 2/a$, as given by \cref{thm:CRthm}.

    (ii) If $B_{\ell_{i_r}\ell_{i_r}}$ are non-zero and of the same sign for some subcollection of indices $\{i_1,\dots, i_m\} \subset \{1,\dots,k\}$ and the off-diagonal terms $B_{\ell_{i_r}\ell_{i_t}}, r \neq t$ are sufficiently small compared to the diagonal terms, \cref{thm:high}(a) furnishes a mixed branch of non-trivial solutions where $\{\ell_{i_1}, \dots, \ell_{i_m}\}$ all correspond to (local) dominant modes.

    (iii) Suppose $\ell,m \in \Lambda$ such that $B_{\ell\ell}>0, B_{mm}>0, B_{\ell m}<0, B_{m\ell}<0$ and $B_{\ell\ell}B_{mm} - B_{\ell m}B_{m\ell}<0$. Then, the $2 \times 2$ submatrix $\tilde{B}$ constructed from $\ell,m$ satisfies $\tilde{B}^{-1}\mathbf{1}<0$ coordinate-wise. This exhibits instances where there are multiple Fourier modes that all correspond to supercritical (subcritical) unimodal bifurcations but jointly produce a mixed subcritical (supercritical) bifurcation around $2/a$.

    A transcritical branch requires a nonzero quadratic contribution in the reduced Lyapunov--Schmidt system (see \cref{subsection:LSdesc} and \cref{lem:resLS}), hence an additive resonance $i+j=k$ among critical indices; repeated indices are allowed. \cref{thm:high}(b) treats a three-distinct-mode instance. A relaxation of this resonance condition is given as a sufficient condition for a discontinuous phase transition in \citet[Section 5.1]{carrillo2020long}. In \cref{sec:dpt}, we show that a branch of non-trivial solutions with $\kappa$ values less than $2/a$ is sufficient to cause such phenomena. Resonance is thus `not specially related' to discontinuous phase transitions, however, resonance is indeed the special structure behind transcritical bifurcations.
\end{remark}

\begin{remark}[\textbf{Metastability and higher order bifurcations}]\label{rem:metahigh}
    Although \cref{thm:high} gives multi-mode bifurcating solutions only when multiple Fourier coefficients of $W$ agree, at an intuitive level, it also gives useful information when $\{a_\ell : \ell \in \Lambda\}$ are sufficiently close to each other for some finite set $\Lambda$. If $\Lambda$ satisfies the hypotheses of \cref{thm:high}(a), then for any nonempty $\tilde{\Lambda} \subseteq \Lambda$, one can compute the associated matrix \(\tilde B\) and
  evaluate \(\tilde B^{-1}\mathbf1\), provided $\tilde B$ is invertible; if the entries of
  \(\tilde B^{-1}\mathbf1\) are nonzero and have the same sign, `approximate' stationary distributions can be computed using the explicit multi-mode pitchfork bifurcating solutions given in \cref{thm:high}(a). This approximate stationarity heuristically implies that the solution to the McKean-Vlasov equation \eqref{eq:mveq} started from such a bifurcating solution remains relatively unchanged over long periods of time, which can be viewed as a form of metastability. A similar analysis can be conducted in the setting of \cref{thm:high}(b). See also \cref{cor:dpt} for a connection between such clustering and discontinuous phase transitions.
\end{remark}

\subsubsection{Quantifying clustering and metastability in Noisy Transformer for large $\beta$} \label{rem:mettran}
Now we apply the above observations to the Noisy Transformer in the large $\beta$ setting.

\textbf{Clustering of Fourier modes of $W_\beta$: }Recall for the Noisy
Transformer coefficients
\[
a_\ell(\beta)=\frac{2I_\ell(\beta)}{\beta},\qquad
\kappa^*_\ell(\beta)=\frac{\beta}{I_\ell(\beta)}.
\]
%The following qualitative asymptotic facts describe the transition from a single-mode regime (small \(\beta\)) to a near-degenerate, multi-mode regime (large \(\beta\)).
\medskip
\iffalse 
\noindent\textbf{(A) Small-\(\beta\) regime: explicit crossing estimate.} \\
Using the power-series asymptotics
\[
I_\ell(\beta)\sim \frac{1}{\ell!}\Big(\frac{\beta}{2}\Big)^\ell\qquad(\beta\downarrow 0),
\]
we obtain the leading-order approximation
\[
a_\ell(\beta)\sim\frac{\beta^{\,\ell-1}}{2^{\ell-1}\ell!}.
\]
Thus the ratio of the two coefficients satisfies
\[
\frac{a_{l_1}(\beta)}{a_{l_2}(\beta)}
\sim
\frac{\beta^{\,l_1-l_2}}{2^{\,l_1-l_2}}\frac{l_2!}{l_1!}.
\]
Hence the approximate solution of \(a_{l_1}(\beta)=a_{l_2}(\beta)\) in the small-\(\beta\)
(asymptotic) regime is
\begin{align}\label{eq:beta_cross_small}
\beta_{\mathrm{cross}}\approx 2\Big(\frac{l_1!}{\,l_2!\,}\Big)^{\!1/(\ell_1-l_2)}.
\end{align}
For \(\beta\ll \beta_{\mathrm{cross}}\) we have \(a_{l_1}(\beta)\gg a_{l_2}(\beta)\) and so the
lower mode \(\ell_1\) dominates (typically the first bifurcation is that mode).  
Example: for \((\ell_1,l_2)=(1,2)\) the formula \eqref{eq:beta_cross_small} gives
\(\beta_{\mathrm{cross}}\approx 4\), i.e. for \(\beta\lesssim 4\) the \(\ell=1\) mode dominates
under the small-\(\beta\) approximation.

\medskip
\fi 
%\noindent\textbf{(B) Large-\(\beta\) regime: clustering and fractional differences.} \\
Fix integers $L \in \mathbb{N}$. Using~\cref{largebetaIbeta}, one obtains for $1 \le \ell \le L$,
\begin{align}\label{aellbetalargebeta}
a_\ell(\beta)
=\frac{2e^\beta}{\beta^{3/2}\sqrt{2\pi}}\Big(1 - \frac{4\ell^2-1}{8\beta} + O(\beta^{-2})\Big).
\end{align}
Consequently the absolute difference between any two coefficients $a_{l_1}, a_{l_2}$ for $1 \le \ell_1 < \ell_2 \le L$ satisfies, to leading order,
\[
\begin{aligned}
a_{\ell_1}(\beta)-a_{\ell_2}(\beta)
&\approx \frac{e^\beta}{\sqrt{2\pi}\,\beta^{5/2}}\,(\ell_2^2-\ell_1^2),
\end{aligned}
\]
while a representative magnitude is
\[
a_\ell(\beta)\approx \frac{2e^\beta}{\beta^{3/2}\sqrt{2\pi}}, \quad 1 \le \ell \le L.
\]
Therefore the \emph{relative} difference is
\begin{align}\label{eq:frac_diff} 
\frac{a_{\ell_1}(\beta)-a_{\ell_2}(\beta)}{a_\ell(\beta)}
\approx \frac{\ell_2^2-\ell_1^2}{2\beta}\quad(\beta\to\infty). 
\end{align}
In particular, for any tolerance \(\varepsilon>0\), \(a_{\ell_1}(\beta)\) and \(a_{\ell_2}(\beta)\) are \(\varepsilon\)-close for any $1 \le \ell_1 < \ell_2 \le L$
(relative to the common scale) provided
\[
\beta \gtrsim \frac{L^2}{2\varepsilon}.
\]
This quantifies the clustering of the Fourier coefficients \(\{a_{l}(\beta) : 1 \le \ell \le L\}\) in relative scale
as \(\beta\to\infty\) which motivates the study of nearly resonant multimode interactions. Moreover, for large \(\beta\) the bifurcation thresholds satisfy the rapid-decay estimate
\[
\kappa^*_\ell(\beta)=\frac{\beta}{I_\ell(\beta)} \sim \sqrt{2\pi}\,\beta^{3/2}e^{-\beta},
\qquad \beta\to\infty.
\]
In particular, the bifurcation points accumulate exponentially fast as $\beta \to \infty$.

\medskip

\textbf{Metastability and higher order bifurcations:} For $\beta\gtrsim L^2/(2\varepsilon)$, this clustering motivates the study of nearly resonant multimode interactions. \cref{thm:high}(a) does not directly apply to $\Lambda=\{1,\ldots,L\}$: the coefficients are distinct, and for $L\ge2$ the set contains the resonance $1+1=2$.
  %For any $\tilde{\Lambda} \subseteq \Lambda$, one can compute the associated matrix \(\tilde B(\beta)\), its determinant \(\tilde D(\beta)\) and
  %evaluate \(\tilde B(\beta)^{-1}\mathbf1\); if \(\tilde D(\beta)\neq0\) and the entries of
 % \(\tilde B(\beta)^{-1}\mathbf1\) are nonzero and have the same sign, 
  The resulting approximate stationary distributions provide a possible approach to quantify metastability for the Noisy Transformer analogous to the noiseless case observed in \cite{geshkovski2024dynamic}.

%\noindent\textbf{(D) When the two-mode ($k=2$) mechanism is relevant.} We can combine (A)–(C) into practical quantitative criteria:

%\begin{itemize}
%  \item \emph{Single-mode dominance:} If \(\beta \ll \beta_{\mathrm{cross}}\) given by\eqref{eq:beta_cross_small} (or more conservatively if \(\beta \le \tfrac{1}{2}\beta_{\mathrm{cross}}\)), then \(a_{l_1}(\beta)\gg a_{l_2}(\beta)\)and the first bifurcation is dominated by mode \(\ell_1\) (two-mode interactions unlikely).

 % \item 
%  \item \emph{Practical rule-of-thumb:} for moderate indices (say \(\ell_1,l_2\le 10\)) the small-\(\beta\) crossing estimate \eqref{eq:beta_cross_small} gives a reasonable first guess for the \(\beta\)-scale where the single-mode to the two-mode transition may occur; for larger \(\beta\) expect fractional differences of \(O(\beta^{-2})\) and hencemulti-mode phenomena become typical.
%\end{itemize}

%The following remark includes some explicit calculations for two-mode branches in the large $\beta$ regime.

%\begin{remark}[Explicit \(k=2\) asymptotics for large \(\beta\)]
%\label{rem:k2-asymptotics}~~
%\textbf{1) Large-\(\beta\) (Laplace expansion).}  
\medskip

\iffalse 
\emph{Example:} For \((\ell_1,l_2)=(2,3)\) we have $k_1(2,3)=\frac{138}{31}$ and $k_2(2,3)=\frac{189}{124}$. Hence the modal amplitudes satisfy 
\[
b_1(\beta)=\frac{\sqrt{138/31}}{\sqrt{\beta}} + O(\beta^{-1}),\qquad
b_2(\beta)=\frac{\sqrt{189/124}}{\sqrt{\beta}} + O(\beta^{-1}).
\]
For \((\ell_1,l_2)=(1,4)\) one obtains $k_1(1,4)=\frac{189}{34}$ and $k_2(1,4)=-\frac{120}{17}$.
The minus sign for \(k_2\) indicates the non-occurrence of the two-mode pitchfork. 
\fi

%for \((\ell_1,l_2)=(1,2)\) one finds \(k_1=15/8\) while \(k_2\) vanishes at the displayed order (so \(x_2\) is higher-order small); hence a two-mode branch at very large \(\beta\) has amplitude \(b_1\sim\sqrt{15/8}\,\beta^{-1/2}\) and \(b_2\) is of smaller order.

\medskip

\subsection{Stationary density representation in the supercritical bifurcation regime}
Now, we give a much more detailed description of non-trivial bifurcating solutions of \eqref{eq:stat} for a large class of potentials $W$ which contains all $W \in H^1_s(\mathbb{S}^1)$, treated in \cite{carrillo2020long}, in addition to a broad class of more singular potentials (see \cref{rem:singpot}). This description provides leading small-amplitude coefficients for arbitrarily many fixed Fourier modes in the exponent. As a test case, this is shown to `almost match' the explicitly solvable stationary solutions for the \emph{Kuramoto model} (see \cref{rem:Kur}).

%The following theorem characterizes all the bifurcation points of the stationary landscape of~\eqref{eq:mveq}, namely solutions to \eqref{eq:stat}, as the intensity of the interaction varies and also gives a detailed representation of the stationary solutions locally around these points.

Before proceeding, recall that we defined $p_0\coloneqq \text{Unif}(\mathbb{S}^1)$. For $\delta >0$, we let $\mathbb{B}_{\delta}(p_0,H_s^1(\mathbb{S}^1))$ denote the $\delta-$ball in $H_s^1(\mathbb{S}^1)$ around $p_0$. Also recall $\kappa_m = 2/a_m$.

\begin{theorem}\label{thm:main_thm}
Assume that $\sup_{\ell \ge 1} \ell |a_\ell|< \infty$. Suppose that there exists $m \in \mathbb{N}$ such that $a_m>0$, $\sum_{\ell=1}^\infty |a_{\ell m}| < \infty$ and $a_{\ell m} \neq a_{m}$ for all $\ell >1$. Then the following hold. If  the $m$-signature of $W$ is strictly positive, i.e., $R_m(W) = {(a_m - 2a_{2m})}/{(a_m - a_{2m}})>0$, there exists $\eta_m,\delta_m >0$ such that for $\kappa \in (\kappa_m, \kappa_m+\delta_m]$ there are two non-trivial solutions to \eqref{eq:stat} in $\mathbb{B}_{\eta_m}(p_0,H_s^1(\mathbb{S}^1))$ given by
    \begin{align}\label{eq:densitybifur}
        p_\kappa^{\pm}(\theta) \coloneqq \frac{1}{Z} \exp \bigg( \kappa \sum_{\ell=1}^\infty s^{\pm}(\kappa)^\ell  (z_{\ell m}(\kappa)+r_{\ell m}(\kappa)) a_{\ell m} \cos \ell m\theta \bigg),
    \end{align}
    where:
    \begin{itemize}
        \item $Z$ is the normalizing factor,
        \item $\{z_{\ell m}(\kappa): \ell \geq 1 \}$ satisfy the following recursive convolution equations:
    \begin{align}\label{eq:zseq}
        z_m(\kappa)=1, \quad \ell m (2-\kappa a_{\ell m}) z_{\ell m}(\kappa) = \kappa \sum_{j=1}^{\ell-1} j m\, a_{jm}\,z_{jm}(\kappa)\,z_{(\ell-j)m}(\kappa), \, \ell \geq 2,
    \end{align}
    \item $r_{\ell m}(\kappa) = O_\ell(|\kappa a_m-2|)$ for each fixed $\ell$, with constants that may depend on $W,m,\ell$,
    \item and
        \begin{align*}
        s^{\pm}(\kappa) \coloneqq \pm \sqrt{\frac{2(\kappa a_m-2)(2-\kappa a_{2m})}{\kappa^2 a_m (a_m - 2a_{2m})}} + O(|\kappa a_m-2|^{3/2}).
    \end{align*}
    \end{itemize} 
In particular, this corresponds to a supercritical branch of non-trivial bifurcating solutions.
    
    The same representation holds when $R_m(W)<0$, but now for $\kappa \in [\kappa_m-\delta_m, \kappa_m)$ (a subcritical branch). Moreover, if $a_\ell \neq a_m$ for all $\ell \neq m$, these are the only possible non-trivial solutions in $\mathbb{B}_{\eta_m}(p_0,H_s^1(\mathbb{S}^1)) \times [\kappa_m-\delta_m, \kappa_m + \delta_m]$.
 \end{theorem}

\begin{remark}\label{rem:Kur}
    For the Kuramoto model corresponding to $W(\theta) = \cos\theta$, the above results naturally characterize the stationary behavior around the first and only bifurcation point at $\kappa_1=2$. The associated  $z_\ell$'s are given by $z_1=1$ and for $\ell\geq 2$, $2\ell z_\ell = \kappa  z_{\ell-1}$, which gives $z_\ell = (\kappa/2)^{\ell-1}/\ell!$, and $s^{\pm}(\kappa) = \pm 2\sqrt{\frac{\kappa-2}{\kappa^2}}+O((\kappa-2)^{3/2})$. Thus, the Fourier coefficients of $p^{\pm}_\kappa$ are given by
    \begin{align}\label{chilexp}
    \chi_\ell = \bigg( \pm 2 \sqrt{\frac{\kappa-2}{\kappa^2}}\bigg)^\ell \bigg((\frac{\kappa}{2})^\ell \frac{1}{\ell!} + O(\kappa-2)\bigg).
    \end{align}
    We now show that this matches the Fourier coefficients of the explicitly known density in this case. For the Kuramoto model, it is known (see, for example, \cite{bertini2010dynamical} and \citet[Section 6.1]{carrillo2020long}) that the uniform distribution is the only stationary distribution for $\kappa \le 2$. $\kappa=2$ is the unique bifurcation point and the unique (up to rotation) non-trivial branch for $\kappa>2$ is given by
    $$
    \pi^{\text{Kur}}(\theta) := \frac{1}{I_0(\kappa a)}\exp\left(\kappa  a\cos \theta\right), \quad \theta \in [0, 2\pi),
    $$
    where $a$ is the unique positive solution to
    $$
    a = \frac{I_1(\kappa a)}{I_0(\kappa a)}.
    $$
    The Fourier coefficients of $\pi^{\text{Kur}}$ are given by $\chi^{\text{Kur}}_\ell = I_\ell(\kappa a)/I_0(\kappa a)$. Note that we have by Taylor expansion
    \begin{align*}
      I_\ell(x) = \frac{1}{\ell!}(\frac{x}{2})^\ell + \frac{1}{(\ell+1)!}  (\frac{x}{2})^{\ell+2} + O(|x|^{\ell+4}).
    \end{align*}
  For $\kappa \in [2, 2+ \delta)$, for sufficiently small $\delta>0$, $a$ is small and by the above Taylor expansion, we see that $a$ satisfies
  $$
  a = \frac{\kappa a/2 + \frac{1}{2}(\kappa a/2)^3 + O(|a|^5)}{1 + (\kappa a/2)^2 + O(|a|^4)},
  $$
which implies two non-zero solutions for $a$ given by
$$
a = \pm 2 \sqrt{ \frac{\kappa -2}{\kappa^2}} + O((\kappa - 2)^{3/2}).
$$
Substituting this above, we conclude
$$
\chi^{\text{Kur}}_\ell = I_\ell(\kappa a)/I_0(\kappa a) = \bigg( \pm 2 \sqrt{\frac{\kappa-2}{\kappa^2}}\bigg)^\ell \bigg((\frac{\kappa}{2})^\ell \frac{1}{\ell!} + O(\kappa-2)\bigg),
$$
which agrees with our expression of $\chi_\ell$ in \eqref{chilexp}.
\end{remark}

\subsubsection{Application to the noisy mean-field Transformers}\label{rem:transformer_explicit}

For the Noisy Transformer model the coefficients are $
a_\ell=\frac{2I_\ell(\beta)}{\beta}$.  For any \(\beta>0\) the sequence \((a_\ell)_{\ell\ge1}\) is positive, strictly decreasing in \(\ell\), and summable:
$$
    \sum_{\ell=1}^\infty a_\ell = W_\beta(0) = \beta^{-1}(e^\beta - I_0(\beta)) < \infty.
$$
So, the hypotheses of \cref{thm:main_thm} are satisfied.~\cref{fig:denvis} provides a visualization of approximate density profiles in~\cref{eq:densitybifur} for small and large $\beta$ values. %A more interactive visualization is available \href{https://claude.ai/public/artifacts/e60b4806-a5eb-410f-812d-ada834f46120}{here}. 
We can also provide an explicit description of a two-term approximation of~\cref{eq:densitybifur}, as follows, with detailed justification in~\cref{sec:adddetails}:
\begin{itemize}
\item Small-$\beta$ regime, with $\Delta=\kappa a_m-2>0$ and the exact $\kappa_m=2/a_m$:
\[
\begin{split}
p_\kappa^{\pm}(\theta)=\frac{1}{Z}\exp\Bigg(&\pm 2\sqrt{\frac{\beta^{m-1}}{2^{m-1}m!}\,(\kappa-\kappa_m)}\cos(m\theta)
+O_m\big((\beta^m+\beta^2)\sqrt\Delta+\beta^m\Delta\big)
+O_{\beta,m}(\Delta^{3/2})\Bigg).
\end{split}
\]
\item Large-$\beta$ regime, with $\delta=2-\kappa a_m>0$:
\[
\begin{split}
p_\kappa^{\pm}(\theta)=\frac{1}{Z}\exp\Bigg(&\pm2\sqrt{\frac{3m^2\delta}{2\beta}}\cos(m\theta)
+\delta\cos(2m\theta)+O_m\big(\sqrt\delta\,\beta^{-3/2}+\delta/\beta\big)
+O_{\beta,m}(\delta^{3/2})\Bigg).
\end{split}
\]
\end{itemize}

%Finally, in the large-$\beta$ regime, the bifurcation point itself satisfies
%\[
%\kappa_m=\frac{2}{a_m}
%\sim \frac{2}{C(\beta)}
%=\beta^{3/2}\sqrt{2\pi}\,e^{-\beta}\big(1+O(\beta^{-1})\big),
%\]
%which decays exponentially fast in $\beta$. 
We take $\kappa\to\kappa_m$ along the bifurcating branch, with $\beta$ and $m$ fixed and $a_m\ne2a_{2m}$. The exponent errors are uniform in $\theta$, and $Z$ normalizes the full expression. The subscripts in $O$ highlight the dependence of constants on parameters. Constants in $O_{\beta,m}$ need not be uniform as $\beta\to\infty$.
Thus, in the large-$\beta$ regime, the bifurcation is subcritical, with a primary cosine mode of frequency $m$ and amplitude $\sim\sqrt{\delta/\beta}$, and a frequency $2m$ second-harmonic contribution of order $\delta$.

\begin{figure}[t]
  \centering
  \begin{subfigure}[b]{\linewidth}
    \centering
    \includegraphics[width=0.50\linewidth,height=3.2cm]{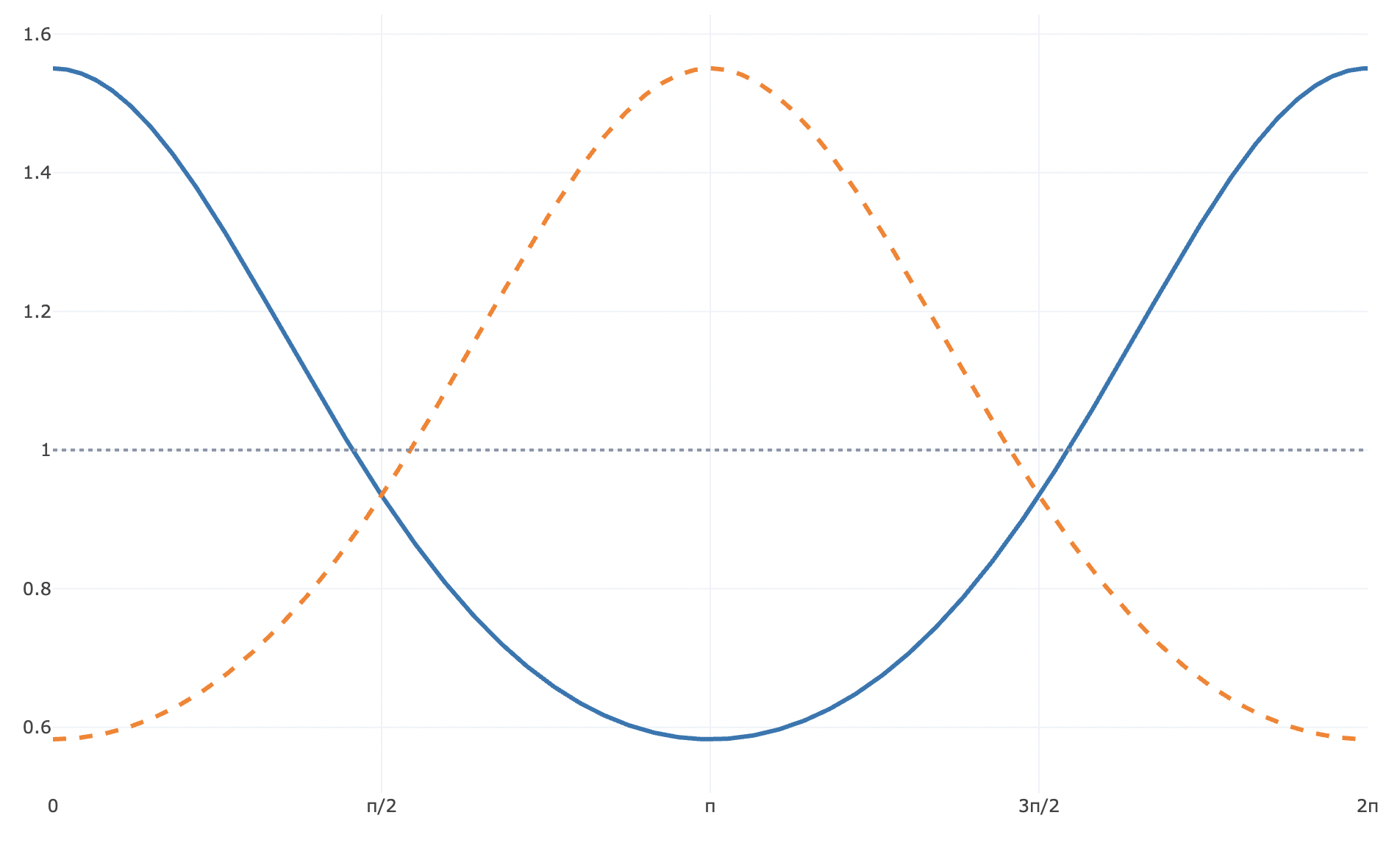}
    \caption{$\beta=0.5$, and $\kappa-\kappa_1=0.05$}
    \label{fig:4}
  \end{subfigure}

   % \vspace{6pt}

  \begin{subfigure}[b]{\linewidth}
    \centering
  \includegraphics[width=0.50\linewidth,height=3.2cm]{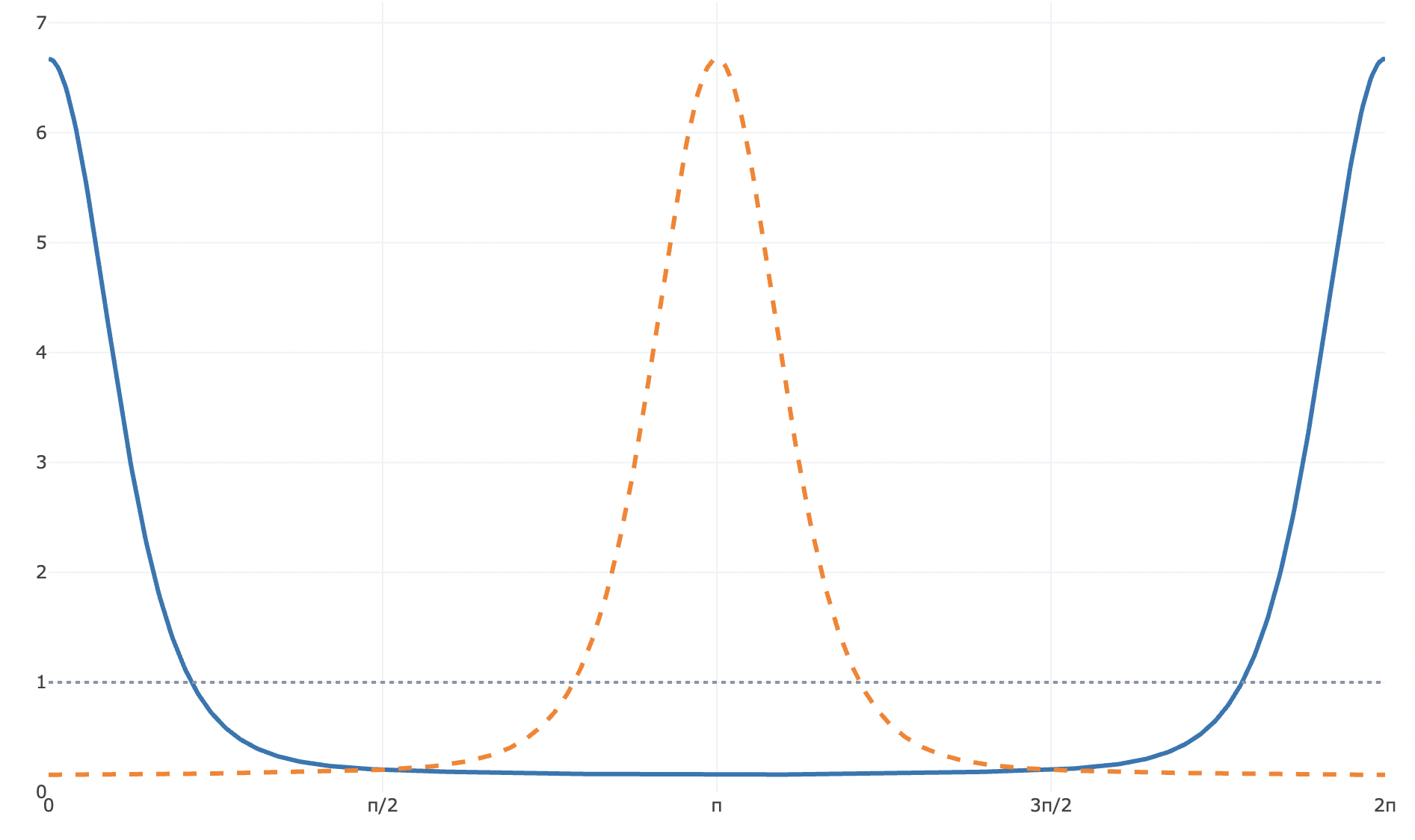}
    \caption{$\beta=5$, and $\kappa-\kappa_1=-0.05$}
    \label{fig:5}
  \end{subfigure}

  \caption{Approximate profiles based on the density representation in~\cref{eq:densitybifur} for a few \emph{representative}  small and large $\beta$ and $\kappa$ values, for $m=1$, computed numerically by truncating the infinite series at 10. Solid and dashed densities correspond to $p^+_\kappa$ and $p^{-}_\kappa$ respectively. }
  \label{fig:denvis}
\end{figure}

%\section{Phase transition}
\subsection{Bifurcations and discontinuous phase transition}\label{sec:dpt}
With the advent of Optimal Transport, the analysis of the McKean-Vlasov equation has benefited significantly from the viewpoint of \eqref{eq:mveq} as a `gradient flow' for the \emph{free energy functional} (see \cite{chayes2010mckean,carrillo2020long}). Let $\mathcal{P}_{ac}^+(\mathbb{S}^1)$ denote the space of absolutely continuous probability measures with a strictly positive density. For densities for which the following integrals are finite, the free energy functional is a map $\mathcal{F}: \mathcal{P}_{ac}^+(\mathbb{S}^1) \times \mathbb{R}_+ \rightarrow \mathbb{R}$ given by
\begin{align*}
    \mathcal{F}(\rho, \kappa) := \frac{1}{2\pi}\int_0^{2\pi} \rho(\theta) \log \rho(\theta) \d \theta - \frac{\kappa}{8\pi^2}\int_0^{2\pi}\int_0^{2\pi}W(\theta - \phi)\rho(\theta)\rho(\phi)\d\theta\d\phi.
\end{align*}
Owing to this gradient flow structure, stationary solutions of~\eqref{eq:mveq}, that is, solutions of~\eqref{eq:stat}, correspond to the critical points of the free energy functional $\mathcal{F}(\cdot, \kappa)$; see \citet[Proposition~2.4]{carrillo2020long}. Their stability is determined by the local geometry of the energy landscape around these critical points.

%As a consequence of this gradient flow structure, stationary solutions of \eqref{eq:mveq}, that is, solutions of \eqref{eq:stat}, can be identified with the critical points of $\mathcal{F}(\cdot, \kappa)$ (see \citet[Proposition 2.4]{carrillo2020long}), and the stability properties of these solutions arise from the energy landscape in a local neighborhood of the solution.

The following definition is taken from \cite{chayes2010mckean,carrillo2020long}. 

\begin{figure}[t]
\centering
\begin{tikzpicture}
\begin{axis}[
    width=8cm,
    height=7cm,
    axis lines=center,
    axis line style={gray!30},
    xlabel={$s$},
    ylabel={$\kappa$},
    xmin=-0.2, xmax=2.2,
    ymin=0, ymax=2.3,
    xtick=\empty,
    ytick=\empty,
    thick,
    samples=200,
]
% Parameters
\pgfmathsetmacro{\kstar}{1.2}  % \kappa^*
\pgfmathsetmacro{\ksn}{0.95}   % minimum of the blue curve
\pgfmathsetmacro{\kc}{1.00}    % \kappa_sn < \kappa_c < \kappa^*
\pgfmathsetmacro{\szero}{0.7}  % location of the blue minimum

% Ensures that the blue curve passes through (0,\kappa^*)
\pgfmathsetmacro{\a}{(\kstar-\ksn)/(\szero*\szero)}

% Right-hand intersection of the blue curve with \kappa=\kappa_c
\pgfmathsetmacro{\sc}{
    \szero + sqrt((\kc-\ksn)/\a)
}

% --- Vertical lines at s=0 ---
\addplot[blue, thick]
    coordinates {(0,0) (0,\kstar)};
\addplot[red, dashed, very thick]
    coordinates {(0,0) (0,\kc)};

% Mark \kappa_c and \kappa^*
\addplot[
    only marks,
    mark=*,
    mark size=1.5pt,
    blue
] coordinates {(0,\kstar)};

\addplot[
    only marks,
    mark=*,
    mark size=1.5pt,
    red
] coordinates {(0,\kc)};

% Labels
\node[anchor=west, blue]
    at (axis cs:0,\kstar) {$\ \kappa^*$};

\node[anchor=west, red]
    at (axis cs:0,\kc) {$\ \kappa_c$};

% --- Blue bifurcation curve ---
\addplot[
    blue,
    thick,
    domain=0:2.2
] {\ksn + \a*(x-\szero)^2};

% --- Free-energy minimizing part of the branch ---
\addplot[
    red,
    dashed,
    very thick,
    domain=\sc:2.2
] {\ksn + \a*(x-\szero)^2};

% --- Co-existence label and arrows ---
\node[
    anchor=south,
    font=\footnotesize\bfseries
] (label) at (axis cs:0.8,0.05) {co-existence};

% Arrow to the uniform state at \kappa_c
\draw[->, thick, gray!70]
    (label.north west)
    .. controls +(150:0.8) and +(-90:0.3)
    .. (axis cs:0,\kc);

% Arrow to the nonuniform state at the same \kappa_c
\draw[->, thick, gray!70]
    (label.north east)
    .. controls +(30:0.8) and +(-90:0.3)
    .. (axis cs:\sc,\kc);

\end{axis}
\end{tikzpicture}
\caption{The dashed (red) curve represents the free energy minimizing
branch. The solid (blue) curve represents the bifurcation branch with
$\kappa<\kappa^*$. The global continuation is schematic; the theorem proves local bending and existence of a competing minimizer.}
\label{fig:kappa-quadratic-overlay}
\end{figure}

\begin{definition}[Transition point and continuity of phase transition]\label{def:contpt}
    A point $\kappa_c>0$ is said to be a transition point for the McKean-Vlasov equation if (i) for $\kappa<\kappa_c$, the uniform distribution $p_0$ is the unique global minimizer of $\mathcal{F}(\cdot, \kappa)$, (ii) at $\kappa = \kappa_c$, $p_0$ is a global minimizer of $\mathcal{F}(\cdot, \kappa)$, and (iii) for $\kappa > \kappa_c$, there exists some $\rho_\kappa \in \mathcal{P}_{ac}^+(\mathbb{S}^1)$ with $\rho_\kappa \neq p_0$ such that $\rho_\kappa$ is a global minimizer of $\mathcal{F}(\cdot, \kappa)$.

    If (i) $p_0$ is the unique minimizer at $\kappa = \kappa_c$, and (ii) for any family of minimizers $\{\rho_\kappa: \kappa > \kappa_c\}$ we have $\lim_{\kappa \downarrow \kappa_c}\|\rho_\kappa - p_0\|_{H^1} = 0$, where $\|\cdot\|_{H^1}$ is as defined in~\cref{sec:solutinspace}, $\kappa_c$ is said to be a continuous transition point. If at least one of these conditions is violated, $\kappa_c$ is said to be a discontinuous transition point.
    
    %there exists a family of minimizers $\{\rho_\kappa: \kappa > \kappa_c\}$ such that $\lim_{\kappa \downarrow \kappa_c}\|\rho_\kappa - \rho_0\| \neq 0$, then $\kappa_c$ is said to be a discontinuous transition point. In this case, there exists $\rho_{\kappa_c}\neq \rho_0$ such that $\mathcal{F}(\rho_{\kappa_c}, \kappa_c) = \mathcal{F}(\rho_0, \kappa_c)$.
\end{definition}

The following theorem connects subcritical bifurcations with discontinuous phase transitions.

\begin{theorem}\label{thm:bifdpt}
   Assume that $W\in H^1_s(\mathbb{S}^1)$. Let $a := \max_{\ell \ge 1} a_\ell >0$ and let $\Lambda := \{ \ell \in \mathbb{N}: a_\ell = a\}$. Suppose that there exist a nonempty subset $\Lambda' \subseteq \Lambda$, $\delta>0$ and a nontrivial $C^1$ branch of bifurcating solutions locally around $\kappa^*:= 2/a$ of the form
   $$
   \underline{p}(s) = s\sum_{i \in \Lambda'} y_i e_i + O (s^2), \quad \underline{p}'(s) = \sum_{i \in \Lambda'} y_i e_i + O (s), \quad \kappa(s) \in (0, \kappa^*) \text{ for all } s \in (0,\delta),
   $$
   where $y_i \neq 0$ for all $i \in \Lambda'$. Then $\kappa_c < \kappa^*$ and $\kappa_c$ is a discontinuous phase transition point. 

   Moreover, there exists $\epsilon>0$ such that for any potential $\tilde{W}(\theta) = \sum_{\ell=1}^\infty \tilde{a}_\ell \cos \ell \theta$ in $H^1_s(\mathbb{S}^1)$, with $\sum_{\ell = 1}^{\infty} |a_\ell - \tilde{a}_\ell| < \epsilon$, $\tilde{a} := \max_{\ell \ge 1} \tilde{a}_\ell >0$ and the free energy for the associated McKean-Vlasov equation has a discontinuous transition point $\tilde{\kappa}_c < 2/\tilde{a}$. 
\end{theorem}

\cref{fig:kappa-quadratic-overlay} visually depicts the situation described in \cref{thm:bifdpt} in the case of a bifurcation with $\kappa < \kappa^*$ and a single dominant Fourier mode, with modal amplitude plotted along the horizontal axis and $\kappa$ along the vertical axis. The transition point $\kappa_c$ corresponds to the $\kappa$ value where the global free energy minimizing branch discontinuously transitions from the vertical axis (uniform distribution) to a non-trivial stationary distribution. At $\kappa = \kappa_c$, there is `coexistence' in the sense that there are two minimizing solutions with the same free energy but different values of the interaction energy; see \cref{thm:global} for more details.

The following corollary follows immediately from \cref{thm:CRthm}, \cref{thm:high} and \cref{thm:bifdpt}, and identifies conditions for bifurcations with a $\kappa < \kappa^*$ branch, and hence discontinuous phase transitions, to hold.

\begin{corollary}\label{cor:dpt}
    Suppose $W\in H^1_s(\mathbb{S}^1)$ satisfies $a := \max_{\ell \ge 1} a_\ell >0$ and one of the following:
    \begin{itemize}
        \item[(i)] $a_1>0$, $a_\ell < a_1$ for all $\ell \ge 2$ and $a_1/2 < a_2 < a_1$;
        \item[(ii)] The hypotheses of \cref{thm:high}(a) with $\operatorname{sgn}\left(\tilde{B}^{-1}\mathbf{1}\right)<0$;
        \item[(iii)] The hypotheses of \cref{thm:high}(b).
    \end{itemize}
    Then the associated McKean-Vlasov equation has a discontinuous transition point $\kappa_c < 2/a$. An analogous assertion holds for any potential $\tilde{W}(\theta) = \sum_{\ell=1}^\infty \tilde{a}_\ell \cos \ell \theta$ in $H^1_s(\mathbb{S}^1)$ with $\sum_{\ell = 1}^{\infty} |a_\ell - \tilde{a}_\ell| < \epsilon$ for sufficiently small $\epsilon>0$.
\end{corollary}

\begin{remark}\label{rem:dptCGPS}
    \cref{cor:dpt}(iii) can be seen as the analogue of \citet[Theorem 5.11]{carrillo2020long}. The above theorem and corollary show that a three-mode resonance is just one way in which a discontinuous transition point occurs. The more general cause seems to be a bifurcating branch with $\kappa$ values less than $2/a$.
\end{remark}

\subsection{Phase transition for the Noisy Transformer}
Recall the Noisy Transformer model described in~\eqref{eq:NMFTpotential}.
The following result characterizes the continuity of phase transition for the Noisy Transformer model and establishes a sharp dichotomy. Before stating the result, we recall from Section~\ref{sec:221} that the function
$\beta \mapsto \frac{I_2(\beta)}{I_1(\beta)}$ is strictly increasing for $\beta>0$. 

\begin{theorem}[Phase transition for the noisy Transformer]
\label{thm:phasetr}
Let $\beta_0>0$ be the unique solution of
\[
\frac{I_2(\beta_0)}{I_1(\beta_0)}=\frac12,
\qquad \beta_0\approx2.446918,
\]
and write $\kappa^*(\beta)=\beta/I_1(\beta)$.
For the potential $W_\beta$ in \eqref{eq:NMFTpotential}, the following hold. Minimization is over all probability densities, with no symmetry or evenness restriction.
\begin{enumerate}[label=(\roman*)]
\item If $0<\beta\le\beta_0$, then
$\kappa_c(\beta)=\kappa^*(\beta)$ is a continuous transition point
in the sense of \cref{def:contpt}.
The uniform density is the unique minimizer at $\kappa^*(\beta)$.
\item If $\beta>\beta_0$, there is a discontinuous transition point
$0<\kappa_c(\beta)<\kappa^*(\beta)$.
At $\kappa_c(\beta)$, the uniform density and at least one nonuniform
density both minimize the free energy.
\item At $\beta=\beta_0$, parameterize the local even branch from
\cref{thm:CRthm} by its first Fourier coefficient $p_1(s)=s$.
Then
\begin{equation}\label{eq:transformer-quartic}
\kappa(s)=\kappa^*(\beta_0)
+\frac{4(3-\beta_0)}{3a_1(\beta_0)}s^4+O(s^6).
\end{equation}
The coefficient of $s^4$ is positive. Thus the bifurcation is
degenerate and supercritical, with amplitude of order
$(\kappa-\kappa^*(\beta_0))^{1/4}$.
\end{enumerate}
\end{theorem}

\section{Towards global properties}\label{sec:global}

In general, going beyond local properties for infinite-dimensional non-linear systems like ours is a formidable task, and little can be said. In this section, we record some global properties. We hope to address this in greater detail in a subsequent article.

\subsection{A `singular' example: log-sine potential and Poisson kernel} The next theorem gives an example of a `singular' potential in the critical bifurcation regime which leads to an explicitly computable global picture with infinitely many bifurcation points. 
    \begin{theorem}\label{cor:explicitbif}
        Let the interaction potential be given by
        \begin{align*}
        W_{Poi}(\theta) := -\log\left(2\sin(\theta/2)\right) = \sum_{\ell=1}^{\infty}\frac{1}{\ell}\cos \ell \theta, \quad \theta \in [0,2\pi). 
        \end{align*}
        Then the bifurcation points of \eqref{eq:stat} are precisely $\kappa \in \{2\ell : \ell \in \mathbb{N}\}$. The associated bifurcating non-trivial solutions at $\kappa^* := 2\ell^*$ correspond to the \textbf{Poisson kernel with wave number} $\ell^*$ given by 
\begin{align}\label{eq:PS}
    P_{r,\ell^*}(\theta) := \frac{1-r^2}{1-2r\cos (\ell^*\theta) + r^2}, \quad \theta \in [0,2\pi), \, r \in (-1,1). 
\end{align}
In particular, as $r \uparrow 1$, the associated measure $\mu_{r,\ell^*}(\d\theta) := P_{r,\ell^*}(\theta)\,\frac{\d\theta}{2\pi}$ weakly converges to the sum of Dirac masses $\frac{1}{\ell^*}\sum_{j=0}^{\ell^*-1}\delta_{2j\pi/\ell^*}$ and, as $r \downarrow -1$, $\mu_{r,\ell^*}(\d\theta)$ weakly converges to $\frac{1}{\ell^*}\sum_{j=0}^{\ell^*-1}\delta_{(\pi + 2j\pi)/\ell^*}$. Every nonuniform stationary solution in $H_s^1(\mathbb{S}^1)$ is one of these Poisson kernels.
\end{theorem}

\begin{figure}[htbp]
    \centering
    %----- First subfigure -----
    \begin{subfigure}[b]{0.7\textwidth}
        \centering
        \includegraphics[width=\textwidth]{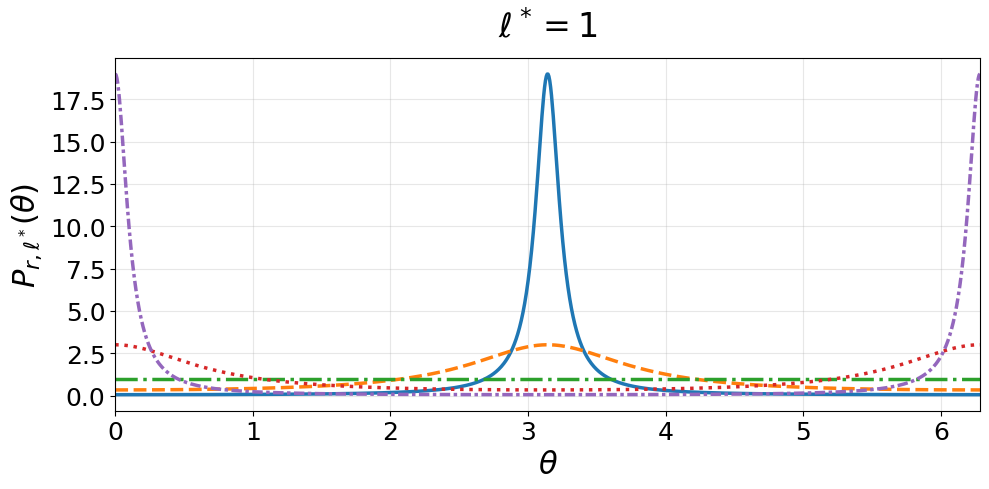}
     %   \label{fig:l1}
    \end{subfigure}
    \\
    %----- Second subfigure -----
    \begin{subfigure}[b]{0.7\textwidth}
        \centering
        \includegraphics[width=\textwidth]{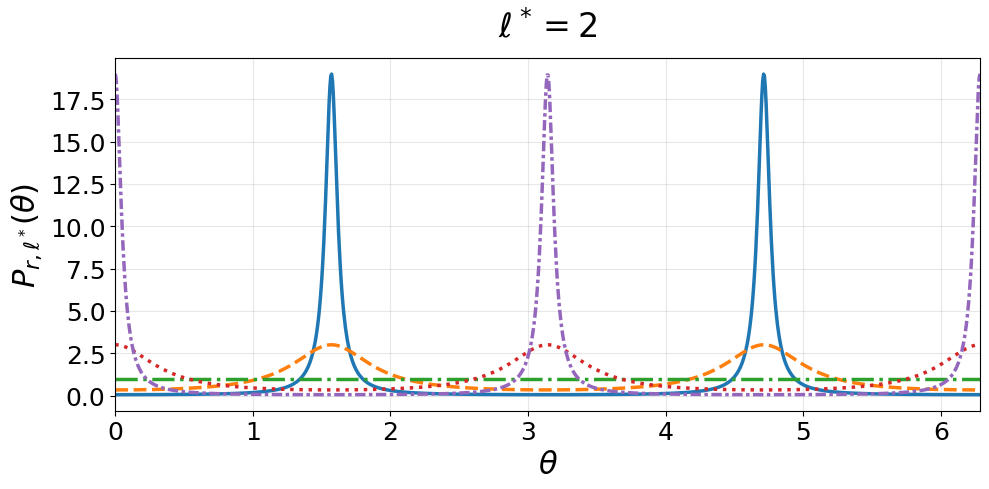}
  %      \label{fig:l2}
    \end{subfigure}
    \\
    %----- Third subfigure -----
    \begin{subfigure}[b]{0.7\textwidth}
        \centering
        \includegraphics[width=\textwidth]{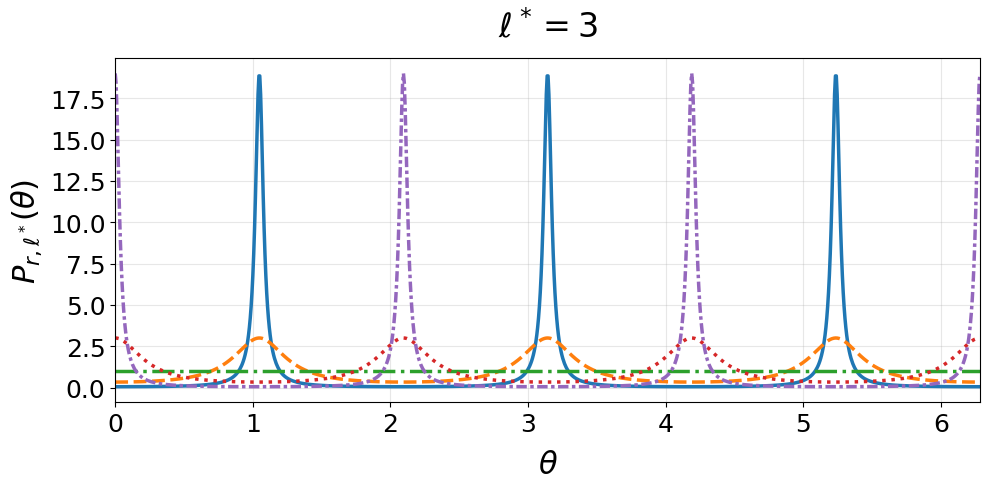}
    %    \label{fig:l3}
    \end{subfigure}
    \begin{subfigure}[b]{0.7\textwidth}
    %\centering
    \hspace{0.25in}\includegraphics[scale=0.5]{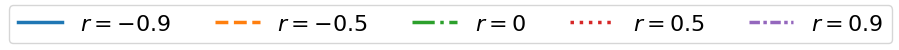}
    \end{subfigure}
    \caption{Visualization of the Poisson kernel in~\cref{eq:PS} for various wave numbers.}
    \label{fig:poissonkernel}
\end{figure}

The ($\ell$-periodic version of the) Poisson kernel $P_{r,\ell}(\theta)$ in~\cref{eq:PS} is a classical object in harmonic analysis, representing the density of the harmonic measure on the unit circle corresponding to the interior point $re^{i\ell\theta}$ of the unit disk (\citet[Section 11.5]{rudin1987real}). Equivalently, it describes how harmonic functions in the disk can be reconstructed from their boundary data via the Poisson integral. In the present context, the appearance of $P_{r,\ell}$ as a stationary solution of the nonlinear mean-field equation reveals an intriguing connection between collective stochastic dynamics and harmonic analytic structures. Indeed, for the logarithmic sine potential $W_{Poi}$, the stationary McKean-Vlasov equation reduces to a nonlinear self-consistency relation for the Fourier modes of the Poisson kernel. Each bifurcation point $\kappa = 2\ell$ thus marks the onset of a new harmonic mode $\cos(\ell\theta)$, and the corresponding branch of stationary measures $\mu_{r,\ell}$ can be viewed as a nonlinear deformation of the harmonic extension problem on the unit disk. 

From the dynamical viewpoint, the limiting regimes $r \uparrow 1$ and $r \downarrow -1$ of $\mu_{r,\ell}$ correspond to extreme ``synchronization'' phases of the interacting system, where the probability mass concentrates on a finite number of evenly spaced points on the circle. The limiting Dirac masses capture perfect phase alignment (for $r \uparrow 1$) and perfect anti-alignment (for $r \downarrow -1$), respectively; see~\cref{fig:poissonkernel}. Hence, the family $\{\mu_{r,\ell}\}_{r \in (-1,1)}$ provides a complete global picture interpolating between the uniform (disordered) state and fully clustered (ordered) configurations. In this sense, the Poisson kernel serves as a canonical bridge between potential theory and the nonlinear geometry of stationary states in McKean-Vlasov dynamics, demonstrating how singular interactions can lead to an infinite cascade of analytically tractable bifurcations. 

\subsection{A global `Kuramoto' branch}
The following result is a simple application of \cref{lem:lemma1} to potentials with finitely many non-zero Fourier coefficients. Stability properties of the associated stationary solutions has been studied recently, both analytically and numerically, by \cite{bertoli2024stability}. Furthermore, \cite{lucia2010exact} and \cite{vukadinovic2023phase} study the special cases of the Onsager model with a finite mode interaction potential and the Hodgkin-Huxley oscillators, respectively.

\begin{theorem}\label{thm:finkur}
Suppose $W(\theta)=\sum_{\ell=1}^L a_\ell\cos\ell\theta$,
$\theta\in[0,2\pi)$, where $a_1>a_2>\cdots>a_L>0$.
Then there exists a global two-sided `Kuramoto'-type branch
whose two non-trivial arms, for $\kappa>2/a_L$, are given by
\[
\pi_\kappa^\pm(\theta)
:=\frac{1}{I_0(\kappa a_Lr_\kappa)}
\exp\!\left(\pm\kappa a_Lr_\kappa\cos L\theta\right),
\qquad \theta\in[0,2\pi),
\]
where $r_\kappa$ is the unique positive solution to
\[
r_\kappa
=\frac{I_1(\kappa a_Lr_\kappa)}
       {I_0(\kappa a_Lr_\kappa)}.
\]
The two arms meet at the uniform density $1$ at $\kappa=2/a_L$,
and satisfy
\(
\pi_\kappa^-(\theta)
=\pi_\kappa^+(\theta+\pi/L).
\)
\end{theorem}

Theorem~\ref{thm:finkur} thus provides an explicit analytic characterization of the global `Kuramoto'-type branch for interaction potentials with finitely many Fourier modes, extending the classical single-mode theory to any dominant harmonic component.

\subsection{Global free energy minimizing solutions}
The next result collects some observations about the `globally free energy minimizing' solutions of the McKean-Vlasov equation. The proof techniques are quite general and, with natural modifications, the results apply to McKean-Vlasov equations on any compact smooth manifold (e.g. the torus or sphere). Some results in this direction were already obtained in \cite{chayes2010mckean}. These observations paint a global picture of the geometry of such solutions. Define for $\kappa>0$,
\begin{align*}
    m(\kappa) &:= \inf_{\rho \in \mathcal{P}^+_{ac}(\mathbb{S}^1)}\mathcal{F}(\rho, \kappa),\\
    M(\kappa) &:= \{\rho \in \mathcal{P}^+_{ac}(\mathbb{S}^1) : \mathcal{F}(\rho, \kappa) = m(\kappa)\}.
\end{align*}
 Writing $\int$ for $\frac{1}{2\pi}\int_0^{2\pi}$, we define the \emph{interaction energy} of $\rho \in \mathcal{P}^+_{ac}(\mathbb{S}^1)$ as
 $$
 \mathcal{E}(\rho,\rho) := \iint W(\theta - \phi)\rho(\theta)\rho(\phi)\d\theta\d\phi.
 $$
\begin{theorem}\label{thm:global}
    Assume $W$ is continuous and $\|W\|_\infty< \infty$. The minimum energy map $\kappa \mapsto m(\kappa)$ and the set map of minimizers $\kappa \mapsto M(\kappa)$ satisfy the following.
    
    (i) For every $\kappa>0$, $M(\kappa)$ is non-empty and compact in $L^1(\mathbb{S}^1)$. Moreover, if $\kappa_n \to \kappa^\circ$ and $\rho_{\kappa_n} \in M(\kappa_n)$, then there exists a subsequence $\{n_j\}$ such that $\rho_{n_j} \to \rho^\circ \in M(\kappa^\circ)$ in $L^1(\mathbb{S}^1)$ (sequential compactness in $\mathbb{R}_+ \times L^1(\mathbb{S}^1)$).
        
    (ii) The map $\kappa \mapsto m(\kappa)$ is concave and globally Lipschitz. Its left and right derivatives $ m'_{-},  m'_{+}$ exist for all $\kappa>0$. They further satisfy the `envelope identities'
        \begin{align*}
            m'_{-}(\kappa) = -\frac{1}{2}\inf_{\rho \in M(\kappa)}\mathcal{E}(\rho,\rho),\quad
            m'_{+}(\kappa) = -\frac{1}{2}\sup_{\rho \in M(\kappa)}\mathcal{E}(\rho,\rho).
        \end{align*}
The maps $\kappa \mapsto \inf_{\rho \in M(\kappa)}\mathcal{E}(\rho,\rho)$ and $\kappa \mapsto \sup_{\rho \in M(\kappa)}\mathcal{E}(\rho,\rho)$ are non-decreasing in $\kappa$.
%Further, if $a_\ell>0$ for all $\ell \in \mathbb{N}$, $m(\cdot)$ is non-increasing.

(iii) There are at most countably many points of non-differentiability (`kinks') of $\kappa \mapsto m(\kappa)$. At such a point $\kappa$, $M(\kappa)$ has at least two elements (called `coexistence' points) which both minimize the free energy but have different values of the interaction energy. 
\end{theorem}

\begin{remark}\label{rem:globalpic}
    Suppose $W\in H_s^1(\mathbb{S}^1)$ and $\max_\ell a_\ell>0$. At the first transition $\kappa_c$ from the uniform state, discontinuity is equivalent to a kink of $m$. Indeed, a nonuniform co-minimizer has $\mathcal{E}(\rho,\rho)=2H(\rho)/\kappa_c>0$, where $H(\rho)=\int\rho\log\rho$, whereas the uniform density has zero interaction energy; the envelope identities in \cref{thm:global}(ii), compactness in part (i), and the Gibbs fixed-point equation (see \eqref{eq:CGPSfpe}) for the stationary density give the equivalence. If the uniform density is the unique minimizer at $\kappa^*=2/\max_\ell a_\ell$, the transition is continuous; if it is not a minimizer there, the transition is discontinuous and occurs earlier (\citet[Proposition 5.8]{carrillo2020long}). The upgrade from $L^1$ to $H^1$ continuity uses the Gibbs formula (see \eqref{eq:l1toh1} in the proof of \cref{thm:phasetr}), which requires only $W\in H^1$. Subsequent kinks may describe coexistence among nonuniform minimizers.
\end{remark}

\section{Proofs: A Fourier representation and bifurcations}

\begin{proof}[Proof of \cref{lem:lemma1}]
Note that $p_0=1/2$ is a consequence of the fact that $p$ is a probability density, and thus $(1/2\pi) \int_0^{2\pi} p(\theta) \d\theta = 1$. For $\ell \geq 1$, we have
\begin{align*}
p_\ell & = \frac{1}{2\pi} \int_0^{2\pi} \cos \ell\theta p(\theta) \d \theta = -\frac{1}{2\pi\ell} \int_0^{2\pi} \sin \ell\theta \, p'(\theta) \d \theta \\
& = -\frac{1}{2\pi\ell} \int_0^{2\pi} \sin \ell\theta \, \kappa(W\star p)'(\theta) p(\theta )\d \theta \qquad\text{(as $p$ is a solution to~\eqref{eq:stat})}\\
& = -\frac{1}{2\pi\ell} \int_0^{2\pi} \sin\ell\theta \bigg[ -\sum_{j=1}^\infty \kappa j a_j p_j\sin j\theta  \bigg]\bigg[  2\sum_{j=0}^\infty p_j\cos j\theta   \bigg] \d\theta \\
& = \frac{1}{\pi\ell}\sum_{j,k\geq 0} \kappa j a_j p_jp_k \int_0^{2\pi} \sin j\theta \cos k\theta \sin \ell\theta \d\theta \qquad\text{(by taking $a_0=0$)}\\
& = \frac{1}{2\pi\ell}\sum_{j,k\geq 0} \kappa j a_j p_jp_k \int_0^{2\pi} \big[\sin((j+k)\theta) + \sin((j-k)\theta)  \big]  \sin \ell\theta \d\theta \\ 
& = \frac{1}{4\pi\ell}\sum_{j,k\geq 0} \kappa j a_j p_jp_k \int_0^{2\pi} \big[ \cos((j+k-\ell)\theta) - \cos((j+k+\ell)\theta)\\
&\qquad\qquad\qquad + \cos((j-k-\ell)\theta) - \cos((j-k+\ell)\theta)\big]  \d\theta \\ 
& = \frac{1}{2\ell}\sum_{j,k\geq 0} \kappa j a_j p_jp_k \big[  \delta_{j+k-\ell} - \delta_{j+k+\ell} + \delta_{j-k-\ell} - \delta_{j-k+\ell}  \big]\\
& = \frac{1}{2\ell}\sum_{j,k: j+k=\ell} \kappa j a_j p_jp_k + \frac{1}{2\ell}\sum_{j,k: j-k=\ell} \kappa j a_j p_jp_k - \frac{1}{2\ell}\sum_{j,k: j-k=-\ell} \kappa j a_j p_jp_k\\
&=\frac{1}{2\ell} \sum_{j=1}^\ell \kappa j a_j p_j p_{\ell-j} +  \frac{1}{2\ell} \sum_{j=\ell}^\infty \kappa j a_j p_j p_{j-\ell} - \frac{1}{2\ell} \sum_{j=1}^\infty \kappa j a_j p_j p_{j+\ell}.
\end{align*}
The infinite sums above are all absolutely convergent as $p \in H_s^1(\mathbb{S}^1)$ and $\sup_{\ell \ge 1} |a_\ell| < \infty$, which is a consequence of $W \in L^2_s(\mathbb{S}^1)$.
Therefore, we have
\begin{align*}
    2\ell(1-\frac{\kappa a_\ell}{2}) p_\ell & = \kappa \sum_{j=1}^{\ell-1} j a_j p_j p_{\ell-j} + \kappa \sum_{j=\ell+1}^\infty j a_j p_jp_{j-\ell} - \kappa \sum_{j=\ell+1}^\infty (j-\ell) a_{j-\ell} p_{j-\ell} p_j\\
    & = \kappa \sum_{j=1}^{\ell-1} j a_j p_j p_{\ell-j} + \kappa \sum_{j=\ell+1}^\infty (ja_j - (j-\ell)a_{j-\ell}) p_jp_{j-\ell}.
\end{align*}
To prove the converse, suppose that $p \in H_s^1(\mathbb{S}^1)$ and its Fourier coefficients satisfy \eqref{eq:Fid} and $p_0=1/2$. Define the function
$$
g(\theta) := p'(\theta) - \kappa(W\star p)'(\theta)p(\theta), \, \theta \in \mathbb{S}^1.
$$
We claim that $g$ is almost surely zero with respect to $\text{Unif}(\mathbb{S}^1)$. Note that $g$ is an odd function, that is, $g(\theta - \pi) = -g(\pi-\theta), \forall \theta \in [0, 2\pi)$. Moreover, by the assumption $p \in H_s^1(\mathbb{S}^1)$, and as $\sup_{\ell \ge 1} |a_\ell| < \infty$, $p(\cdot), p'(\cdot)$ and $(W\star p)'(\cdot) =  -\sum_{j=1}^\infty j a_j p_j\sin (j\cdot)$ are square integrable. Further, $p \in H_s^1(\mathbb{S}^1)$ implies that $p$ is uniformly bounded on $\mathbb{S}^1$ and, combining this with the previous observation, we obtain the square integrability of $g$. Thus, to prove the claim, it suffices to show that for any odd square integrable function $\phi(\theta) = \sum_{l=1}^\infty c_\ell \sin(l\theta)$, $\int_0^{2\pi} \phi(\theta)g(\theta)d\theta = 0$. This follows from the above calculation, which shows that \eqref{eq:Fid}, along with $p_0=1/2$, is equivalent to $\int_0^{2\pi}\sin \ell \theta \, p'(\theta)\d\theta = \int_0^{2\pi}\sin \ell \theta \, \kappa(W\star p)'(\theta)p(\theta)\d\theta$ for every $\ell \ge 1$.

The positivity and exponential representation of $p$ follows from the observation that 
$$\theta \mapsto p(\theta) e^{-\kappa(W\star p)(\theta)}$$ 
is a constant function (along with the fact that $p$ is a probability density), which is a consequence of $g=0$ almost surely with respect to $\text{Unif}(\mathbb{S}^1)$.

This proves the result. 
\end{proof}

\begin{proof}[Proof of \cref{lem:Freg}]
    For sequences $\underline{u}, \underline{v}$ in $\mathbb{R}^{\mathbb{Z}}$ satisfying $\sum_{j \in \mathbb{Z}} u_j^2 < \infty, \, \sum_{j \in \mathbb{Z}} v_j^2 < \infty$, define the discrete convolution $\underline{u} \star \underline{v}$ with entries given by $\underline{u} \star \underline{v}(\ell) := \sum_{j \in \mathbb{Z}}u_jv_{\ell-j}, \, \ell \in \mathbb{Z}$. For a sequence $\underline{u}$ in $\ell^2$, naturally extend it to a sequence in $\mathbb{R}^{\mathbb{Z}}$ by setting $u_j=0$ for $j \le 0$. Then, using the assumption $\sup_{\ell \ge 1}\ell|a_\ell|<\infty$, we obtain $C>0$ such that for any $\underline{p} \in \ell^2_w$ and $\kappa >0$,
     \begin{align*}
\|F(\underline{p},\kappa)\|^2_{\ell^2}\le C(1+\kappa^2)\left(\|\underline{p}\|^2_{\ell^2_w}+\|\underline{q}\star\underline{q}\|_{\ell^2}^2+\|\underline{q}\star\underline{q}^-\|_{\ell^2}^2\right),
    \end{align*}
where $q_j:=|p_j|$ and $q_j^-:=q_{-j}$ for $j\in\mathbb{Z}$. By Young's inequality for convolutions,
    $$
\|\underline{q}\star\underline{q}\|_{\ell^2}^2\le\|\underline{q}\|_{\ell^1}^2\|\underline{q}\|_{\ell^2}^2
    $$
    where
    $$
\|\underline{q}\|_{\ell^1}:=\sum_{j\in\mathbb{N}}|p_j|\le\left(\sum_{j\in\mathbb{N}}(1+j^2)p_j^2\right)^{1/2}\left(\sum_{j\in\mathbb{N}}\frac{1}{1+j^2}\right)^{1/2}<\infty
    $$
applying the Cauchy--Schwarz inequality and the fact that $\underline{p}\in\ell_w^2$. Thus $\|\underline{q}\star\underline{q}\|_{\ell^2}^2<\infty$. A similar argument shows $\|\underline{q}\star\underline{q}^-\|_{\ell^2}^2<\infty$. Hence $F$ is a well-defined map from $\ell_w^2\times\mathbb{R}_+$ to $\ell^2$.

    The existence and representation of the Fr\'echet derivatives of $F$ follow upon simply observing that $F$ can be written as a sum $F(\underline{p},\kappa) = L(\underline{p},\kappa) + Q(\underline{p},\kappa)$ where $L(\cdot,\kappa): \ell^2_w \rightarrow \ell^2$ is a bounded linear map and $Q(\cdot,\kappa): \ell^2_w \times \ell^2_w \rightarrow \ell^2$ is a bounded bilinear map, where the boundedness follows similarly as above.
\end{proof}

\subsection{Lyapunov-Schmidt reduction}\label{subsection:LSdesc} The proofs of our results rely on the so-called \emph{Lyapunov-Schmidt reduction} (\citet[Theorem I.2.3]{kielhofer2012bifurcation}) which transforms the problem of characterizing zero sets in infinite-dimensional Banach spaces into essentially a finite-dimensional one by careful applications of the (Banach space version of the) Implicit Function Theorem (\citet[Theorem I.1.1]{kielhofer2012bifurcation}). We first cast our problem in this general set-up, where we mainly follow \cite{kielhofer2012bifurcation}. First, note that $F(\underline{0},\kappa) = 0$ for all $\kappa>0$. Moreover, from \cref{lem:Freg},
\begin{align}\label{eq:derF}
    \left[D_{\underline{p}}F(\underline{0}, \kappa)\underline{h}\right]_\ell = \ell(2-\kappa a_\ell) h_\ell, \quad \ell \in \mathbb{N}.
\end{align}
Suppose, for some $\kappa^* \in \mathbb{R}_+$, there are exactly $k$ Fourier modes $\{a_{\ell_1}, \dots, a_{\ell_k}\}$ such that $a_{\ell_j} = 2/\kappa^*$ for $1 \le j \le k$.
Thus, the dimension of the kernel and co-dimension of the range of $D_{\underline{p}}F(\underline{0}, \kappa^*)$ are both $k$. Moreover, as $D_{\underline{p}}F(\underline{0}, \kappa^*)$ is a diagonal mapping, it is easy to check that its range is closed. Thus, $F$ is a nonlinear Fredholm operator of index zero (\citet[Definition I.2.1]{kielhofer2012bifurcation}). Consequently, we choose the coordinate complements $X_0=\{x\in\ell_w^2:x_{\ell_j}=0,\ 1\le j\le k\}$ and $Z_0=\operatorname{span}\{e_{\ell_1},\ldots,e_{\ell_k}\}\subset\ell^2$, so that
\begin{align*}
    \ell^2_w = N \oplus X_0,\qquad
    \ell^2 = R \oplus Z_0,
\end{align*}
where $N = \operatorname{ker}D_{\underline{p}}F(\underline{0}, \kappa^*) = \operatorname{span}[\{e_{\ell_1}, \dots, e_{\ell_k}\}]:= \{\underline{x} \in \ell_w^2 : x_{\ell} = 0 \ \forall \ \ell \neq \ell_1,\dots, \ell_k\}$ and $R = \operatorname{Ran}D_{\underline{p}}F(\underline{0}, \kappa^*) = \{\underline{x} \in \ell^2 : x_{\ell_j} = 0 \text{ for all } 1 \le j \le k\}$. Let $P$ denote the continuous projection of $\ell^2_w$ onto $N$ along $X_0$ and $Q$ denote the continuous projection of $\ell^2$ onto $Z_0$ along $R$. Then proceeding as in \citet[Theorem I.2.3]{kielhofer2012bifurcation}, the problem
$$
F(\underline{p},\kappa) = 0
$$
can be split into an equivalent system of a `kernel equation' and a `range equation':
\begin{align*}
QF(u + w,\kappa) = 0 \ \text{ and } \ (I-Q) F(u + w,\kappa) = 0,
\end{align*}
where $u=P\underline{p}$ and $w=(I-P)\underline{p}$. Using the invertibility of $D_{\underline{p}}F(\underline{0},\kappa^*)|_{X_0}:X_0\to R$ and the Implicit Function Theorem, we can obtain open sets $\tilde U \subset N$, $\tilde W \subset X_0$, $V \subset \mathbb{R}_+$ and a $C^1$ function $\Psi: \tilde U \times V \rightarrow \tilde W$ satisfying $\Psi(\underline{0},\kappa^*) = 0$ such that
\begin{align*}
    (I-Q) F(u + w,\kappa) = 0 \text{ for } (u,w,\kappa) \in \tilde U \times \tilde W \times V \text{ is equivalent to } w = \Psi(u,\kappa).
\end{align*}
Using this in the kernel equation $QF(u+w,\kappa) = 0$, we thus obtain a neighborhood $U \times V$ of $(\underline{0},\kappa^*)$ and $\tilde{U}$ of $0$ in $N$ such that the problem
$$
F(\underline{p},\kappa) = 0 \qquad (\underline{p},\kappa) \in U \times V
$$
is equivalent to
$$
\Phi(u,\kappa) = 0 \qquad (u,\kappa) \in \tilde{U} \times V,
$$
where $\Phi: \tilde{U} \times V \rightarrow Z_0$ is the $C^1$ function given by
\begin{align}\label{eq:bf}
    \Phi(u,\kappa) = Q F(u + \Psi(u,\kappa), \kappa), 
\end{align}
$\Phi$ satisfies $\Phi(0, \kappa^*)=0$ and is called the \emph{bifurcation function}. Since $F$ is $C^\infty$ by Lemma \ref{lem:Freg}, \citet[Equation (I.1.7)]{kielhofer2012bifurcation}, in fact, implies that $\Psi,\Phi$ are both $C^\infty$.

\subsection{Proof of \cref{thm:CRthm}} As in \cite{carrillo2020long}, the proof relies on the Crandall-Rabinowitz Theorem \citep{crandall1971bifurcation}, which can be viewed as a specialization of the Lyapunov-Schmidt setup in the case where the kernel of the derivative map $\operatorname{ker}D_{\underline{p}}F(\underline{0}, \kappa^*)$ is one-dimensional. In this special case, by \citet[Theorem I.5.1 and Equations (I.6.3) and (I.6.11)]{kielhofer2012bifurcation}, we have the following.
\begin{proposition}\label{prop:CRstatement}
    Suppose the hypotheses of \cref{thm:CRthm} hold. Moreover, suppose the following `transversality' condition is satisfied:
    \begin{align}\label{eq:trans}
        D^2_{\underline{p}\kappa}F(\underline{0}, \kappa^*)e_{\ell^*} \notin R.
    \end{align}
    Then, there exists a non-trivial branch of solutions $(\underline{p}(s), \kappa(s)) : (-\delta, \delta) \rightarrow  \ell^2_w \times \mathbb{R}_+$ such that
    \begin{align*}
        \underline{p}(s) = se_{\ell^*} + o(s)
    \end{align*}
    and $\kappa(\cdot)$ is twice differentiable and satisfies $\kappa(0) = \kappa^*$ and
    \begin{align*}
    \kappa'(0) &= -\frac{\langle D^2_{\underline{p}\underline{p}}F(\underline{0}, \kappa^*)[e_{\ell^*},e_{\ell^*}], e_{\ell^*}\rangle}{2\langle D^2_{\underline{p}\kappa}F(\underline{0}, \kappa^*)e_{\ell^*}, e_{\ell^*}\rangle},\\
        \kappa''(0) &= -\frac{\langle D^3_{uuu}\Phi(0, \kappa^*)[e_{\ell^*},e_{\ell^*},e_{\ell^*}], e_{\ell^*}\rangle}{3\langle D^2_{\underline{p}\kappa}F(\underline{0}, \kappa^*)e_{\ell^*}, e_{\ell^*}\rangle}.
    \end{align*}
    Moreover, $(\underline{p}(s), \kappa(s))$ is the only non-trivial solution in a neighborhood of $(\underline{0},\kappa^*)$ in $\ell^2_w \times \mathbb{R}_+$.
\end{proposition}
Now, we are ready to prove \cref{thm:CRthm}.

\begin{proof}[Proof of \cref{thm:CRthm}]
    By \cref{lem:Freg}, we get the following:
    \begin{align}\label{eq:cur0}
        \langle D^2_{\underline{p}\underline{p}}F(\underline{0}, \kappa^*)[e_{\ell^*},e_{\ell^*}], e_{\ell^*}\rangle = 0, \quad 
        \langle D^2_{\underline{p}\kappa}F(\underline{0}, \kappa^*)e_{\ell^*}, e_{\ell^*}\rangle = - \ell^*a_{\ell^*}.
    \end{align}
    From the last equation, we conclude that the transversality condition is satisfied, and the above two equations give $\kappa'(0)=0$. To compute the curvature $\kappa''(0)$, we leverage \citet[Equation (I.6.9)]{kielhofer2012bifurcation}:
    \begin{align}\label{eq:cur1}
        D^3_{uuu}\Phi(0, \kappa^*)[e_{\ell^*},e_{\ell^*},e_{\ell^*}] &= QD^3_{\underline{p}\underline{p}\underline{p}}F(\underline{0}, \kappa^*)[e_{\ell^*},e_{\ell^*},e_{\ell^*}]\notag\\
        &-3QD^2_{\underline{p}\underline{p}}F(\underline{0}, \kappa^*)[e_{\ell^*},(I-P)\left(D_{\underline{p}}F(\underline{0}, \kappa^*)\right)^{-1}(I-Q)D^2_{\underline{p}\underline{p}}F(\underline{0}, \kappa^*)[e_{\ell^*},e_{\ell^*}]],
    \end{align}
    where $P$ and $Q$ are the projection operators defined earlier. As $F$ is a quadratic function, the first term on the right hand side above is zero. Now we compute the second term.

    From the explicit description of $R$, observe that $[(I-Q)\underline{x}]_\ell = x_\ell \mathbbm{1}_{\{\ell \neq \ell^*\}}$. Thus, we obtain $\left[(I-Q)D^2_{\underline{p}\underline{p}}F(\underline{0}, \kappa^*)[e_{\ell^*},e_{\ell^*}]\right]_{\ell^*}=0$ and for $\ell \neq \ell^*$,
    \begin{align*}
        \left[(I-Q)D^2_{\underline{p}\underline{p}}F(\underline{0}, \kappa^*)[e_{\ell^*},e_{\ell^*}]\right]_{\ell} &= - 2\kappa^* \sum_{j < \ell}  j a_j \mathbbm{1}_{\{j = \ell^*\}} \mathbbm{1}_{\{\ell-j = \ell^*\}}\\
        &\qquad - 2\kappa^* \sum_{j > \ell} (ja_j - (j-\ell)a_{j-\ell})\mathbbm{1}_{\{j = \ell^*\}} \mathbbm{1}_{\{j - \ell = \ell^*\}}.
    \end{align*}
    The only non-zero coordinate above comes from the case $\ell = 2\ell^*$ and $j=\ell^*$ in the first sum and we get
    \begin{align*}
        \left[(I-Q)D^2_{\underline{p}\underline{p}}F(\underline{0}, \kappa^*)[e_{\ell^*},e_{\ell^*}]\right]_{\ell} = -2\kappa^*\ell^*a_{\ell^*}\mathbbm{1}_{\{\ell = 2\ell^*\}}.
    \end{align*}
    Consequently,
    \begin{align*}
        \left[\left(D_{\underline{p}}F(\underline{0}, \kappa^*)\right)^{-1}(I-Q)D^2_{\underline{p}\underline{p}}F(\underline{0}, \kappa^*)[e_{\ell^*},e_{\ell^*}]\right]_{\ell} &= -\frac{2\kappa^*\ell^*a_{\ell^*}}{2\ell^*(2-\kappa^*a_{2\ell^*})}\mathbbm{1}_{\{\ell = 2\ell^*\}}\\
        &= -\frac{\kappa^*a_{\ell^*}}{2-\kappa^*a_{2\ell^*}}\mathbbm{1}_{\{\ell = 2\ell^*\}}.
    \end{align*}
    As the $\ell^*$-th entry of the above is zero and $I-P$ projects onto $X_0$, we obtain from the above
    \begin{align}\label{eq:cur2}
        \underline{\phi} &:= (I-P)\left(D_{\underline{p}}F(\underline{0}, \kappa^*)\right)^{-1}(I-Q)D^2_{\underline{p}\underline{p}}F(\underline{0}, \kappa^*)[e_{\ell^*},e_{\ell^*}]\notag\\
        &= \left(D_{\underline{p}}F(\underline{0}, \kappa^*)\right)^{-1}(I-Q)D^2_{\underline{p}\underline{p}}F(\underline{0}, \kappa^*)[e_{\ell^*},e_{\ell^*}] = -\frac{\kappa^*a_{\ell^*}}{2-\kappa^*a_{2\ell^*}}e_{2\ell^*}.
    \end{align}
    As $Q$ projects along the $e_{\ell^*}$ direction,
    \begin{align*}
        QD^2_{\underline{p}\underline{p}}F(\underline{0}, \kappa^*)[e_{\ell^*},\underline{\phi}] = \left[D^2_{\underline{p}\underline{p}}F(\underline{0}, \kappa^*)[e_{\ell^*},\underline{\phi}]\right]_{\ell^*}e_{\ell^*},
    \end{align*}
    where, using \cref{lem:Freg},
    \begin{align*}
        \left[D^2_{\underline{p}\underline{p}}F(\underline{0}, \kappa^*)[e_{\ell^*},\underline{\phi}]\right]_{\ell^*} &= - \kappa^* \sum_{j < \ell^*}  j a_j (\mathbbm{1}_{\{j = \ell^*\}} \phi_{\ell^*-j} + \phi_j\mathbbm{1}_{\{\ell^*-j = \ell^*\}})\\ 
        &\qquad - \kappa^* \sum_{j > \ell^*} (ja_j - (j-\ell^*)a_{j-\ell^*})(\mathbbm{1}_{\{j = \ell^*\}} \phi_{j - \ell^*} + \phi_j\mathbbm{1}_{\{j - \ell^* = \ell^*\}}).
    \end{align*}
    The only non-zero contribution to the right hand side above comes from taking $j=2\ell^*$ in the second sum, giving
    \begin{align*}
         QD^2_{\underline{p}\underline{p}}F(\underline{0}, \kappa^*)[e_{\ell^*},\underline{\phi}] = -\kappa^*(2\ell^*a_{2\ell^*} - \ell^*a_{\ell^*})\phi_{2\ell^*}\, e_{\ell^*} = \frac{(\kappa^*)^2a_{\ell^*}(2\ell^*a_{2\ell^*} - \ell^*a_{\ell^*})}{2-\kappa^*a_{2\ell^*}} \, e_{\ell^*}.
    \end{align*}
    Using this in \eqref{eq:cur1}, we obtain
    \begin{align}\label{eq:cur3}
        D^3_{uuu}\Phi(0, \kappa^*)[e_{\ell^*},e_{\ell^*},e_{\ell^*}] = -\frac{3(\kappa^*)^2a_{\ell^*}(2\ell^*a_{2\ell^*} - \ell^*a_{\ell^*})}{2-\kappa^*a_{2\ell^*}} \, e_{\ell^*}.
    \end{align}
    Using \eqref{eq:cur0} and \eqref{eq:cur3} in the formula for $\kappa''(0)$ in \cref{prop:CRstatement}, and using $\kappa^* = 2/a_{\ell^*}$, we conclude that
    \begin{align*}
        \kappa''(0) = -\frac{(\kappa^*)^2(2a_{2\ell^*} - a_{\ell^*})}{2-\kappa^*a_{2\ell^*}} = \frac{2(a_{\ell^*} - 2a_{2\ell^*})}{a_{\ell^*}(a_{\ell^*} - a_{2\ell^*})}.
    \end{align*}
    Finally, if $\kappa\notin\{2/a_j:a_j>0\}$, then by \eqref{eq:derF}, $D_{\underline{p}}F(\underline{0}, \kappa)$ is invertible. By the Implicit Function Theorem for Banach spaces (\citet[Theorem I.1.1]{kielhofer2012bifurcation}), there exists a neighborhood $U' \times V'$ of $(\underline{0}, \kappa)$ and a continuous mapping $f:V' \rightarrow U'$ such that the zeros of $F$ in $U' \times V'$ form the set $\{(f(\kappa),\kappa) : \kappa \in V'\}$. But this curve clearly has to be $f \equiv 0$ as $F(\underline{0},\kappa') = 0$ for any $\kappa' \in \mathbb{R}_+$. Thus, such a $\kappa$ cannot be a bifurcation point.
\end{proof}

\subsection{Proof of \cref{thm:period}}
\begin{proof}[Proof of \cref{thm:period}(a)]
    It suffices to show that $\{(\underline{p}(s), \kappa(s)): s \in (-\delta, \delta)\}$ as described in the theorem satisfies $F(\underline{p}(s), \kappa(s)) = 0$ for all $s \in (-\delta, \delta)$. Recall
    \begin{align*}
    F_{\ell}(\underline{p}, \kappa) = \ell(2-\kappa a_\ell) p_\ell - \kappa \sum_{j < \ell}  j a_j p_j p_{\ell-j} - \kappa \sum_{j > \ell} (ja_j - (j-\ell)a_{j-\ell})p_jp_{j-\ell}, \ \ell \ge 1.
\end{align*}
If $m \nmid \ell$, then for any $j$ in one of the above sums, either $m \nmid j$ or $m \nmid |\ell - j|$. This gives $F_{\ell}(\underline{p}(s), \kappa(s)) = 0$ for any $\ell \in \mathbb{N}$ such that $m \nmid \ell$. Moreover, writing $a^{(m)}_j := a_{jm}$, we have for any $\ell \in \mathbb{N}$,
\begin{align*}
    F_{\ell m}(\underline{p}(s), \kappa(s)) &= \ell m(2-\kappa (s) a^{(m)}_{\ell}) p_{\ell m}(s) - \kappa(s) \sum_{j < \ell}  j m a^{(m)}_j p_{jm}(s) p_{\ell m-jm}(s)\\
    &\qquad - \kappa(s)\sum_{j > \ell} m(ja^{(m)}_{j} - (j-\ell)a^{(m)}_{j-\ell})\, p_{jm}(s)p_{jm-\ell m}(s)\\
    &=  m\left[\ell(2-\kappa^{(m)}(s) a^{(m)}_{\ell}) p^{(m)}_{\ell}(s) - \kappa^{(m)}(s) \sum_{j < \ell}  j a^{(m)}_j p^{(m)}_{j}(s) p^{(m)}_{\ell-j}(s)\right.\\
    &\left. \qquad - \kappa^{(m)}(s)\sum_{j > \ell} (ja^{(m)}_{j} - (j-\ell)a^{(m)}_{j-\ell}) \, p^{(m)}_{j}(s)p^{(m)}_{j-\ell}(s)\right] = 0,
\end{align*}
as $\{(\underline{p}^{(m)}(s), \kappa^{(m)}(s)): s \in (-\delta, \delta)\}$ is the non-trivial solution branch at $2/a_m$ for the stationary equation \eqref{eq:stat} with potential $W^{(m)}$. This proves the result.
\end{proof}

To prove \cref{thm:period}(b), we need the following lemmas. In the following, we suppress the dependence of the solution on $\kappa$ and write $\{\chi_\ell : \ell \in \mathbb{N}\}$ for the Fourier coefficients of a generic solution to \eqref{eq:stat} for a prescribed value of $\kappa$. We also denote the $\ell^2$-norm by $\|\underline{\chi}\| := \sqrt{\sum_{j=1}^{\infty}\chi_j^2}$ and the $\ell_w^2$-norm by $\|\underline{\chi}\|_w := \sqrt{\sum_{j=1}^{\infty} (1 + j^2)\chi_j^2}$.

Let $S \subset \mathbb{N}$ be a finite set such that $a_\ell = a>0$ for all $\ell \in S$ and $a_\ell \neq a$ for all $\ell \notin S$. Let $P_{S}$ (respectively $P_{S^c}$) denote the projection onto the coordinates in $S$ (respectively $S^c$).

From the Lyapunov-Schmidt reduction described in \cref{subsection:LSdesc}, observe that for $\eta, \delta>0$  sufficiently small, $(\underline{\chi},\kappa)$ solving $F(\underline{\chi},\kappa)=0$, namely,
\begin{align}\label{eq:chieq}
\ell(2-\kappa a_\ell) \chi_\ell = \kappa \sum_{j < \ell} j a_j \chi_j \chi_{\ell-j} + \kappa \sum_{j > \ell} (j a_j - (j-\ell)a_{j-\ell} ) \chi_j \chi_{j-\ell}, \quad \ell\in\mathbb{N},
\end{align} 
with $\|\underline{\chi}\|_w \le \eta$ and $\kappa \in \left(\frac{2}{a} - \delta, \frac{2}{a} + \delta\right)$, is equivalent to
$$
\underline{w} = \Psi(\underline{u}, \kappa) \ \ \text{and} \ \ F_\ell(\underline{u} + \underline{w}, \kappa) = 0 \ \forall \ \ell \in S,
$$
where $\Psi: P_S\underline{\chi} \times \mathbb{R}_+ \rightarrow P_{S^c}\underline{\chi}$ is a $C^1$ function with $\Psi(\underline{0},2/a) = 0$ and $\underline{u} = P_S\underline{\chi}$, $\underline{w} = P_{S^c}\underline{\chi}$. The next lemma controls $\|\underline{w}\|_w$ in terms of $\|\underline{u}\|_w$ and will be used several times in the sequel.
\begin{lemma}\label{lem:lemma2}
Assume $\sup_{\ell \ge 1} \ell |a_\ell|< \infty$. In the Lyapunov-Schmidt reduction setup above, there exist $\eta, \delta>0$ and $C>0$ such that if $(\underline{\chi},\kappa)$ solves \eqref{eq:chieq} with $\|\underline{\chi}\|_w \le \eta$ and $\kappa \in \left(\frac{2}{a} - \delta, \frac{2}{a} + \delta\right)$, then
$$
\|\underline{w}\|_w \le C \|\underline{u}\|_w^2.
$$
\end{lemma}

\begin{proof}[Proof of \cref{lem:lemma2}]
Note that, as $a_j\to0$ as $j\to\infty$, for $\delta>0$ sufficiently small,
\[
\nu:=\inf_{\kappa\in I,\ j\in S^c}|2-\kappa a_j|>0,
\qquad I=[2/a-\delta,2/a+\delta].
\]
Proceeding as in the proof of \cref{lem:Freg}, with $q_j=|\chi_j|$ (extended by zero for $j\le0$) and $q_j^-=q_{-j}$, gives constants $C',C''>0$ such that
\begin{align*}
\|\underline{w}\|_w
&\le C'\nu^{-1}\kappa\left(\|\underline{q}\star\underline{q}\|_{\ell^2}^2+\|\underline{q}\star\underline{q}^-\|_{\ell^2}^2\right)^{1/2}\\
&\le C''\|\underline{\chi}\|_w^2
=C''\|\underline{u}\|_w^2+C''\|\underline{w}\|_w^2.
\end{align*}
Moreover, as $\underline{w}=\Psi(\underline{u},\kappa)$, $\Psi$ is continuous and $\Psi(\underline{0},2/a)=0$, choosing $\delta,\eta$ small enough,
$$
C''\|\underline{w}\|_w < 1/2 \ \ \forall \ \ \|\underline{\chi}\|_w \le \eta, \kappa \in \left(\frac{2}{a} - \delta, \frac{2}{a} + \delta\right).
$$
From the above, we conclude
$$
\|\underline{w}\|_w \le 2C'' \|\underline{u}\|_w^2,
$$
which completes the proof of the lemma.
\end{proof}

\begin{lemma}\label{lem:lemma4}
Suppose $\sum_{j\ge 1}j^2 a_j^2< \infty$, $\{\chi_\ell:\ell \in \mathbb{N} \}$ satisfies \eqref{eq:chieq} and $\|\underline{\chi}\| <\infty$.
Assume there exists $m \in \mathbb{N}$ such that $a_m>0$ and, writing (the necessarily finite set) $S:= \{\ell \in \mathbb{N} : a_\ell = a_m\}$, we have
$$
\operatorname{gcd}(S) = g>1.
$$
Moreover, assume that there exists $C>0$ such that $|\chi_\ell| \leq C\sum_{k\in S}|\chi_k|, \forall \ell \in \mathbb{N}$. Then there exist $\delta'>0, \eta'>0$ such that if
\begin{align*}
   \kappa \in \left[\frac{2}{a_m} - \delta', \frac{2}{a_m} + \delta'\right] \ \text{ and } \ \sum_{k\in S}|\chi_k| \leq \eta',
\end{align*}
we must necessarily have $\chi_\ell = 0$, $\forall \ell\in\mathbb{N}$ such that $g \nmid \ell$. 
\end{lemma}

\begin{proof}[Proof of \cref{lem:lemma4}]
For $\ell \in \mathbb{N}$ with $g \nmid \ell$, 
\begin{align}\label{eq:temp1}
    |\ell(2-\kappa a_\ell) \chi_\ell| \leq \kappa \sum_{j<\ell} j|a_j| |\chi_j| |\chi_{\ell-j}| + \kappa \sum_{j>\ell} \left(j|a_j| + (j-\ell)|a_{j-\ell}|\right) |\chi_j| |\chi_{j-\ell}|. 
\end{align}
Write $\underline{z}\coloneqq (|\chi_\ell|: g \nmid \ell)$. If $g \nmid \ell$, then for any $j \in \mathbb{N}$, $g$ cannot simultaneously divide both $j$ and $\ell-j$ (or equivalently, $j$  and $j-\ell)$. Hence, from~\eqref{eq:temp1}, we obtain $\underline{z} \leq M \underline{z}$ (entrywise) where for $i,j$ such that $g \nmid i$, $g \nmid j$, 
\begin{align*}
    M_{ij} &\coloneqq \frac{\kappa C}{i|2-\kappa a_i|}\Big((2j|a_j| + (i-j)|a_{i-j}| + (i+j)|a_{i+j}|)\mathbbm{1}_{\{ j <i\}}\\
    & \qquad + (2j|a_j| + (j-i)|a_{j-i}| + (i+j)|a_{i+j}|)\mathbbm{1}_{\{ j >i\}}\Big) \sum_{k\in S}|\chi_k|\\
    &= \frac{C}{i|\frac{2}{\kappa}- a_i|}|\left(2j|a_j| + |i-j||a_{|i-j|}| + (i+j)|a_{i+j}|\right)\sum_{k\in S}|\chi_k|,
\end{align*}
for all such $i,j$, including $i=j$, with $a_0=0$.
Note that if $g \nmid j$, then $a_j \neq a_m$. Thus, as $a_n \to 0$ as $n \to \infty$ and $a_m>0$, $\inf_{j: g \nmid j}|a_j - a_m|>0$. Hence, there exist $\delta'>0, \gamma>0$ such that for all $i$ with $g \nmid i$,
$$
\left|\frac{2}{\kappa}- a_i\right| \ge \gamma \ \text{ for all} \ \kappa \in \left[\frac{2}{a_m} - \delta', \frac{2}{a_m} + \delta'\right].
$$
From the above observations and Cauchy-Schwarz inequality, $\|\underline{z}\| \le \|M\|\|\underline{z}\|$, where
\begin{align*}
    \|M\|^2 & = \sum_{i,j: g \nmid i,\, g \nmid j} M^2_{ij} \\
    & \leq 21 C^2\gamma^{-2} \left(\sum_{k\in S}|\chi_k|\right)^2\sum_{i=1}^\infty \frac{1}{i^2} \big( \sum_{j=1}^\infty j^2a_j^2\big).
\end{align*}
Thus, we conclude that if
\begin{align*}
    \sum_{k\in S}|\chi_k| \leq \frac{1}{\sqrt{42 C^2 \gamma^{-2} \sum_i \frac{1}{i^2} \sum_j j^2 a_j^2}},
\end{align*}
then $\|M\|^2\leq 1/2$ and consequently $\underline{z}=0$, thereby proving the desired result.

\end{proof}

\begin{proof}[Proof of \cref{thm:period}(b)]
   Recall $S:= \{\ell \in \mathbb{N} : a_\ell = a_m\}$. Consider any non-trivial branch of solutions $(\underline{p}(s), \kappa(s)) : (-\delta, \delta) \rightarrow  \ell^2_w \times \mathbb{R}_+$ at bifurcation point $2/a_m$. 
   %By the Cauchy-Schwarz inequality, $\sum_{j=1}^\infty |a_j|< \infty$ is implied by the assumption $\sum_{j=1}^\infty j^2a_j^2< \infty$.
   Using \cref{lem:lemma2}, we obtain $C_1>0, \delta_1>0$ such that 
\begin{align*}
    |p_\ell(s)| \leq C_1\sum_{j \in S} |p_j(s)|, \ \forall \ell \in \mathbb{N}, \ s \in (-\delta_1, \delta_1).
\end{align*}
By choosing $\eta', \delta'$ as in \cref{lem:lemma4} with $C_1$ in place of $C$, we obtain $\delta_2 \in (0,\delta_1]$ such that
$$
\kappa(s) \in \left[\frac{2}{a_m} - \delta', \frac{2}{a_m} + \delta'\right] \ \text{ and } \ \sum_{j\in S} |p_j(s)| \le \eta'\ \forall s \in (-\delta_2, \delta_2).
$$
By \cref{lem:lemma4}, the result follows.
\end{proof}

\subsection{Proof of \cref{thm:high}} 
The proof depends on making the Lyapunov-Schmidt reduction described in \cref{subsection:LSdesc} more explicit under the assumptions of the theorem. Recall the projection map $P$ mapping $\ell_w^2$ to the subspace spanned by the coordinates indexed by $\Lambda$. By \cref{lem:lemma2}, for $\underline{\chi}$ solving \eqref{eq:chieq} with $\|\underline{\chi}\|_w$ sufficiently small 
\begin{align}\label{eq:projcon}
\|(I - P)\underline{\chi}\|_w \le O\left(\|\underline{\chi}\|_w^2\right).
\end{align}
This implies that $\Lambda$ correspond to the `dominant modes' around the bifurcation point $\kappa = 2/a$. Now, we carefully track the interactions between the dominant modes to obtain the \emph{reduced Lyapunov-Schmidt equation} given in the following lemma.
\begin{lemma}\label{lem:redLS}
    Let the assumptions in \cref{thm:high}(a) hold. For any non-trivial branch of solutions to \eqref{eq:chieq} locally around $\kappa = 2/a$, we have the following reduced Lyapunov-Schmidt equation:
    \begin{align*}
        (2-\kappa a)\chi_l + B_{ll} \, \chi_l^3 + \sum_{j \in \Lambda, j \neq l} B_{lj} \, \chi_l \,\chi_j^2 + O\left(\|\underline{\chi}\|_w^4\right) + O\left(\|\underline{\chi}\|_w^3|\kappa a - 2|\right) = 0, \quad l \in \Lambda,
    \end{align*}
    for sufficiently small $\|\underline{\chi}\|_w$.
\end{lemma}

\begin{proof}[Proof of \cref{lem:redLS}]
We assume without loss of generality that $\kappa$ is such that $k a_j \neq 2$ for any $j \notin \Lambda$.

We abbreviate the coefficients of \eqref{eq:chieq} as $\gamma(i,j) := ja_j, j <i,$ and $\gamma(i,j) := ja_j - (j-i)a_{j-i}, j >i$. We set $\gamma(i,i)=0$ for every $i$ and also take $\chi_0=0$.
For $l \in \Lambda$, we write
\begin{align}\label{dec}
    l(2-\kappa a)\chi_l &= \kappa\sum_{j=1}^\infty \gamma(l,j) \chi_j \chi_{|l-j|}\notag\\
    & = \kappa\sum_{j \in \Lambda} \gamma(l,j) \chi_j \chi_{|l-j|} + \kappa\sum_{j \notin \Lambda, |l-j| \in \Lambda} \gamma(l,j) \chi_j \chi_{|l-j|} + \kappa\sum_{j \notin \Lambda, |l-j| \notin \Lambda} \gamma(l,j) \chi_j \chi_{|l-j|}.
\end{align}
    By \eqref{eq:projcon},
    \begin{align}\label{red1}
        \sum_{j \notin \Lambda, |l-j| \notin \Lambda} \gamma(l,j) \chi_j \chi_{|l-j|} = O\left(\|\underline{\chi}\|_w^4\right).
    \end{align}
    Now, we address the first term in \eqref{dec} when $j \in \Lambda$, and obtain the contributions of the indices in $\Lambda$ to $\chi_{|l-j|}$. By the `no resonance' assumption, $|l-j| \notin \Lambda$. From \eqref{eq:chieq} and \eqref{eq:projcon},
    \begin{align*}
        |l-j|(2-\kappa a_{|l-j|}) \chi_{|l-j|} &= \kappa\sum_{r=1}^\infty \gamma(|l-j|,r) \chi_r\chi_{||l-j|-r|}\\
        &= \kappa \sum_{r: r \in \Lambda, ||l-j|-r| \in \Lambda} \gamma(|l-j|,r) \chi_r\chi_{||l-j|-r|} + O\left(\|\underline{\chi}\|_w^3\right).
    \end{align*}
    
    We have the following cases.

    \textit{Case (i).} $j < l, r < l-j$.
As $r,l-j-r,j,l$ are all in $\Lambda$ and $l = (l-j-r) + j + r$, the second assumption on $\Lambda$ in the theorem is violated. Thus, there is no such $r \in \Lambda$.

 \textit{Case (ii).} $j < l, r > l-j$.
  As $r,r+j - l,j,l$ are all in $\Lambda$ and $l + (r+j-l) = j + r$, the third assumption on $\Lambda$ gives $r=l$ as the only such index. 

  \textit{Case (iii).} $j > l, r < j-l$. No such index $r$ as $j = (j-l-r) + \ell + r$.
  
 \textit{Case (iv).} $j > l, r > j-l$. From $l+r = j + (r + l-j)$, the only such index is $r=j$.

  Combining the above observations, we obtain
\begin{align}\label{case1}
      |l-j|(2-\kappa a_{|l-j|}) \chi_{|l-j|} = \kappa \gamma(|l-j|, j \vee l) \chi_j\chi_\ell + O\left(\|\underline{\chi}\|_w^3\right).
  \end{align}

  We now address the second term in \eqref{dec}, when $j \notin \Lambda, |l-j| \in \Lambda$. Note that
  \begin{align*}
      j(2-\kappa a_j) \chi_j = &= \kappa\sum_{r=1}^\infty \gamma(j,r) \chi_r\chi_{|j-r|}\\
        &= \kappa \sum_{r: r \in \Lambda, |j-r| \in \Lambda} \gamma(j,r) \chi_r\chi_{|j-r|} + O\left(\|\underline{\chi}\|_w^3\right).
  \end{align*}

Again, we have four cases.

\textit{Case (i).} $j < l, r < j$. No such index as $(l-j) + (j-r) + r = l$.

\textit{Case (ii).} $j < l, r > j$. In this case, $(l-j) + r = (r-j) + l$ gives $r=l$.

\textit{Case (iii).} $j > l, r < j$. For this case, $(j-l) + l = (j-r) + r$ gives $r=l$ and an additional index $r=j-l$ when $j \neq 2l$.

\textit{Case (iv).} $j > l, r > j$. No such index as $(r-j) + (j-l) + l = r$.

Combining the above,
\begin{align}\label{case2}
    j(2-\kappa a_j) \chi_j = \kappa \gamma(j,l) \, \chi_{|l-j|}\chi_l + \kappa \gamma(j,j-l) \mathbf{1}[j>l,\, j \neq 2l] \, \chi_{|l-j|}\chi_l + O\left(\|\underline{\chi}\|_w^3\right).
\end{align}

  Using \eqref{case1} and \eqref{case2} in \eqref{dec}, we obtain
  \begin{align}\label{dec2}
      l(2-\kappa a) \chi_l &= \kappa^2\sum_{j \in \Lambda, j < l} \frac{\gamma(l,j)\gamma(l-j,l)}{(l-j)(2-\kappa a_{l-j})}\chi_l\chi_j^2 + \kappa^2\sum_{j \in \Lambda, j > l} \frac{\gamma(l,j)\gamma(j-l,j)}{(j-l)(2-\kappa a_{j-l})}\chi_l\chi_j^2\notag\\
      &\quad + \kappa^2\sum_{j \notin \Lambda, |l-j| \in \Lambda} \frac{\gamma(l,j)\left[\gamma(j,l) + \gamma(j,j-l)\mathbf{1}[j>l,\, j \neq 2l]\right]}{j(2-\kappa a_j)}\chi_l \chi_{|l-j|}^2 + O\left(\|\underline{\chi}\|_w^4\right).
  \end{align}
  For the third term above, we now re-index $t = |l-j|$ to get
  \begin{align*}
      &\kappa^2\sum_{j \notin \Lambda, |l-j| \in \Lambda} \frac{\gamma(l,j)\left[\gamma(j,l) + \gamma(j,j-l)\mathbf{1}[j>l,\, j \neq 2l]\right]}{j(2-\kappa a_j)}\chi_l \chi_{|l-j|}^2 = \kappa^2\sum_{t \in \Lambda, t < l} \frac{\gamma(l,l-t)\gamma(l-t,l)}{(l-t)(2-\kappa a_{l-t})}\chi_l \chi_t^2\\
      &\qquad + \kappa^2\sum_{t \in \Lambda} \frac{\gamma(l,l+t)\left[\gamma(l+t,l) + \gamma(l+t,t)\mathbf{1}[t \neq l]\right]}{(l+t)(2-\kappa a_{l+t})}\chi_l \chi_t^2 + O\left(\|\underline{\chi}\|_w^4\right).
  \end{align*}
  Using this in \eqref{dec2}, we conclude
  \begin{align}\label{dec3}
      l(2-\kappa a) \chi_l &= \kappa^2\sum_{j \in \Lambda, j < l} \frac{\left[\gamma(l,j) + \gamma(l,l-j)\right]\gamma(l-j,l)}{(l-j)(2-\kappa a_{l-j})}\chi_l\chi_j^2 + \kappa^2\sum_{j \in \Lambda, j > l} \frac{\gamma(l,j)\gamma(j-l,j)}{(j-l)(2-\kappa a_{j-l})}\chi_l\chi_j^2\notag\\
      &\quad + \kappa^2\sum_{j \in \Lambda} \frac{\gamma(l,l+j)\left[\gamma(l+j,l) + \gamma(l+j,j)\mathbf{1}[j \neq l]\right]}{(l+j)(2-\kappa a_{l+j})}\chi_l \chi_j^2 + O\left(\|\underline{\chi}\|_w^4\right).
  \end{align}
  Now, we express the coefficients above in terms of the Fourier coefficients $\{a_j\}$ of $W$. For $j = l$ the coefficient of $\chi_l^3$ is
  \begin{align*}
  \kappa^2 \frac{\gamma(l,2l)\gamma(2l,l)}{2l(2-\kappa a_{2l})} &= \kappa^2\frac{(2l a_{2l} - la_l)(la_l)}{2l(2-\kappa a_{2l})}\\
  &= \frac{\kappa^2 a l(2a_{2l} - a)}{2(2-\kappa a_{2l})} = -l B_{ll} + O(|\kappa a - 2|).
  \end{align*}
  For $j<l$, the coefficient of $\chi_l\chi_j^2$ is
  \begin{align*}
      &\kappa^2\frac{\left[\gamma(l,j) + \gamma(l,l-j)\right]\gamma(l-j,l)}{(l-j)(2-\kappa a_{l-j})} + \kappa^2 \frac{\gamma(l,l+j)\left[\gamma(l+j,l) + \gamma(l+j,j)\right]}{(l+j)(2-\kappa a_{l+j})}\\
      & = \kappa^2 a \frac{ja + (l-j)a_{l-j}}{2-\kappa a_{l-j}} + \kappa^2 a \frac{(l+j)a_{l+j} - ja}{2-\kappa a_{l+j}} = - l B_{lj} + O(|\kappa a - 2|).
  \end{align*}
  Similarly, for $j>l$, the coefficient of $\chi_l\chi_j^2$ is
  \begin{align*}
  &\kappa^2 \frac{\gamma(l,j)\gamma(j-l,j)}{(j-l)(2-\kappa a_{j-l})} + \kappa^2 \frac{\gamma(l,l+j)\left[\gamma(l+j,l) + \gamma(l+j,j)\right]}{(l+j)(2-\kappa a_{l+j})}\\
  &= \kappa^2 a \frac{ja - (j-l)a_{j-l}}{2-\kappa a_{j-l}} + \kappa^2 a \frac{(l+j)a_{l+j} - ja}{2-\kappa a_{l+j}} = - l B_{lj} + O(|\kappa a - 2|).
  \end{align*}
  The lemma now follows upon using these expressions in \eqref{dec3}.
\end{proof}

\begin{proof}[Proof of Theorem \ref{thm:high}(a)]
    By the Lyapunov-Schmidt reduction and Lemma \ref{lem:redLS}, the problem
    $F(\underline{\chi},\kappa)=0$ is (locally) equivalent to
    $\Phi(\underline{u},\kappa)=0$, where $\Phi$ can be seen as a map from $\mathbb{R}^k \times \mathbb{R}_+$ to $\mathbb{R}^k$ given (locally) by
    $$
    \Phi_\ell(\underline{u},\kappa)
    = (2-\kappa a)u_\ell
    + B_{\ell\ell}u_\ell^3
    + \sum_{j\in\Lambda,\ j\neq\ell}B_{\ell j}u_\ell u_j^2
    + O\left(\|\underline{u}\|_w^4\right)
    + O\left(\|\underline{u}\|_w^3|\kappa a-2|\right),
    \quad \ell\in\Lambda.
    $$
    Since $F$ is $C^\infty$ by \cref{lem:Freg}, the range map
    $\Psi$ and the reduced map $\Phi$ are also $C^\infty$ (see Section \ref{subsection:LSdesc}).

    Write $S:=\{\ell_{i_1},\dots,\ell_{i_m}\}$,
    $\underline{y}:=\tilde{B}^{-1}\mathbf{1}$ and
    $\sigma:=\operatorname{sgn}(\underline{y})$.
    Define $\underline{\tilde{u}}^0\in\mathbb{R}^k$ by
    $$
    \tilde{u}^0_{\ell_{i_r}}=b_r,\quad 1\le r\le m,
    \qquad
    \tilde{u}^0_\ell=0,\quad \ell\in\Lambda\setminus S.
    $$
    Thus $b_r^2=\sigma y_r$ and
    $\tilde{B}\underline{y}=\mathbf{1}$.
    Make the change of variables
    $$
    \kappa(s)=\frac{2+\sigma s^2}{a},
    \qquad
    \underline{u}=s\underline{\tilde{u}},
    $$
    and, for $s\neq0$, set
    $$
    G_\ell(\underline{\tilde{u}},s)
    :=s^{-3}\Phi_\ell
    \bigl(s\underline{\tilde{u}},\kappa(s)\bigr),
    \qquad \ell\in\Lambda.
    $$
    By Taylor's formula and the absence of quadratic terms in
    $\Phi$, this map extends smoothly to $s=0$, with
    $$
    G_\ell(\underline{\tilde{u}},0)
    =\tilde{u}_\ell\left[
    -\sigma+\sum_{j\in\Lambda}B_{\ell j}\tilde{u}_j^2
    \right].
    $$
    In particular, $G(\underline{\tilde{u}}^0,0)=0$.

    Ordering the coordinates to put the ones in $S$ first, the derivative with
    respect to $\underline{\tilde{u}}$ at
    $(\underline{\tilde{u}}^0,0)$ is
    $$
    D_{\underline{\tilde{u}}}G(\underline{\tilde{u}}^0,0)
    =\begin{pmatrix}
    2D\tilde{B}D & 0\\
    0 & \operatorname{diag}(d_\ell)_{\ell\in\Lambda\setminus S}
    \end{pmatrix},
    $$
    where $D:=\operatorname{diag}(b_1,\dots,b_m)$ and
    $$
    d_\ell
    =-\sigma+\sum_{r=1}^m B_{\ell\ell_{i_r}}b_r^2
    =\sigma\left[
    \sum_{r=1}^m B_{\ell\ell_{i_r}}y_r-1
    \right],
    \qquad \ell\in\Lambda\setminus S.
    $$
    The first block is invertible since $\tilde{B}$ is invertible
    and every $b_r$ is positive. The non-degeneracy
    assumption \eqref{eq:nondege} gives $d_\ell\neq0$ for every
    $\ell\in\Lambda\setminus S$. Hence the full derivative is
    invertible.

    By the Implicit Function Theorem
    (\citet[Theorem I.1.1]{kielhofer2012bifurcation}),
    there exist $\delta>0$ and a smooth curve
    $\underline{\tilde{u}}(s)$, $s\in(-\delta,\delta)$, such that
    $\underline{\tilde{u}}(0)=\underline{\tilde{u}}^0$ and
    $G(\underline{\tilde{u}}(s),s)=0$.
    Reconstructing the full solution by
    $$
    \underline{p}(s)
    =s\underline{\tilde{u}}(s)
    +\Psi\bigl(s\underline{\tilde{u}}(s),\kappa(s)\bigr)
    $$
    and using \eqref{eq:projcon} and smoothness, we obtain
    $$
    \underline{p}(s)
    =s\sum_{r=1}^m b_r e_{\ell_{i_r}}+O(s^2),
    \qquad
    \underline{p}'(s)
    =\sum_{r=1}^m b_r e_{\ell_{i_r}}+O(s).
    $$
    Together with the expression for $\kappa(s)$, this proves
    part (a).
\end{proof}

\begin{lemma}\label{lem:resLS}
    Let the hypotheses in \cref{thm:high}(b) hold. For any non-trivial branch of solutions to \eqref{eq:chieq} locally around $\kappa = 2/a$, we have the following reduced system of Lyapunov-Schmidt equations:
    \begin{align*}
        &(2-\kappa a)\chi_l -\kappa  a \chi_{m}\chi_{l+m} +  O\left(\|\underline{\chi}\|_w^3\right) = 0,\\
        &(2-\kappa a)\chi_m -\kappa  a \chi_{l}\chi_{l+m} +  O\left(\|\underline{\chi}\|_w^3\right) = 0,\\
        &(2-\kappa a)\chi_{l+m} -\kappa  a \chi_{l}\chi_m +  O\left(\|\underline{\chi}\|_w^3\right) = 0.
    \end{align*}
    for sufficiently small $\|\underline{\chi}\|_w$.
\end{lemma}

\begin{proof}[Proof of \cref{lem:resLS}]
    Recall \eqref{eq:chieq}:
    $$
    \ell(2-\kappa a_\ell) \chi_\ell = \kappa \sum_{j < \ell} j a_j \chi_j \chi_{\ell-j} + \kappa \sum_{j > \ell} (j a_j - (j-\ell)a_{j-\ell} ) \chi_j \chi_{j-\ell}, \quad \ell \in \mathbb{N}.
    $$
    Consider $\ell = l$ above, where $l,m$ are as in the theorem. As $l \neq 2m$, there is no $j \in \mathbb{N}$ such that $\{j, l -j\} \in \{m,l,m+l\}$. Thus, when $\ell=l$, the only contribution purely from the dominant modes at the quadratic level comes from the second sum when $j=l+m$, corresponding to the only $j \in \mathbb{N}$ such that $\{j, j-l\}\in \{m,l,m+l\}$. Similarly, the only such quadratic contribution in \eqref{eq:chieq} when $\ell = m$ also comes from $j=l+m$. Finally, when $\ell = l+m$, the only contributions now come from the first sum when $j=l,m$. 
    
    The remaining terms are $O\left(\|\underline{\chi}\|_w^3\right)$ by \eqref{eq:projcon}. The lemma follows.
\end{proof}

\begin{proof}[Proof of \cref{thm:high}(b)]
 Write $\Lambda := \{l,m,l+m\}$. Similarly as in the proof of \cref{thm:high}(a), by the Lyapunov-Schmidt reduction and \cref{lem:resLS}, the problem $F(\underline{\chi},\kappa)=0$ is (locally) equivalent to $\Phi(\underline{u},\kappa)=0$ where $\Phi:\mathbb{R}^3 \times \mathbb{R}_+ \rightarrow \mathbb{R}^3$ is given by
    $$
    \Phi_\ell(\underline{u}, \kappa) = (2-\kappa a)u_\ell - \kappa a \prod_{j \in \Lambda, j \neq \ell} u_j + O\left(\|\underline{u}\|_w^3\right), \quad \ell \in \Lambda.
    $$
    Make the change of variables $(\underline{u},\kappa)$ to $(\underline{\tilde{u}},\mu)$ by substituting
    $$
    \kappa = (2 - \mu)/a, \quad  \underline{u} = \mu\underline{\tilde{u}}.
    $$
    Then,
    \begin{align*}
    \Phi_\ell(\underline{u}, \kappa) &= \mu^2\left(\tilde{u}_\ell - \kappa a \prod_{j \in \Lambda, j \neq \ell} \tilde{u}_j + O\left(|\mu|\|\underline{\tilde{u}}\|_w^3\right)\right),\\
    & = \mu^2\left(\tilde{u}_\ell - 2 \prod_{j \in \Lambda, j \neq \ell} \tilde{u}_j + R(\underline{\tilde{u}},\mu)\right) =: \mu^2G(\underline{\tilde{u}},\mu), \quad \ell \in \Lambda,
    \end{align*}
    where $R(\underline{\tilde{u}},\mu) = O\left(|\mu|\|\underline{\tilde{u}}\|_w^2\right)$ and $R(\underline{\tilde{u}},0)=0$. Now, we solve the system of equations
    $$
    z_\ell = 2 \prod_{j \in \Lambda, j \neq \ell} z_j, \quad \ell \in \Lambda,
    $$
    in $(z_l,z_m, z_{l+m})$, for solutions other than $(0,0,0)$. Note that, from the above equations, any other solution must necessarily have all the variables non-zero, and we must necessarily have $\frac{z_l}{z_m} = \frac{z_m}{z_l}$ from which we obtain $z_l = \pm z_m$. If $z_l = z_m$, we obtain two solutions
    $$
    \left(\frac{1}{2},\frac{1}{2}, \frac{1}{2} \right) \ \text{ and } \ \left(-\frac{1}{2},-\frac{1}{2}, \frac{1}{2} \right),
    $$
    and two more solutions when $z_l = -z_m$:
    $$
    \left(\frac{1}{2},-\frac{1}{2}, -\frac{1}{2} \right) \ \text{ and } \ \left(-\frac{1}{2},\frac{1}{2}, -\frac{1}{2} \right).
    $$
    More compactly, the solutions of the above equations are given by $\frac{1}{2}(\sigma_1, \sigma_2, \sigma_3)$ where $\sigma_i = \pm 1$ and $\sigma_1\sigma_2\sigma_3=1$. These give the solutions for $G(\underline{\tilde{u}},0)=0$. Moreover, the derivative of the map $\underline{\tilde{u}} \mapsto G(\underline{\tilde{u}},\mu)$ at $(\frac{1}{2}(\sigma_1, \sigma_2, \sigma_3), 0)$ for any such collection of $\{\sigma_i\}$ is given by
    \[
J=\begin{pmatrix}
1 & -\sigma_{3} & -\sigma_{2} \\
-\sigma_{3} & 1 & -\sigma_{1} \\
-\sigma_{2} & -\sigma_{1} & 1
\end{pmatrix},
\qquad
\det J = -4 \quad \text{as } \sigma_{1}\sigma_{2}\sigma_{3}=+1.
\]
Thus, the Implicit Function Theorem applies around any such solution and we obtain the bifurcating branches as given in the theorem.
    \end{proof}

\section{Proofs: Stationary density representation in the supercritical bifurcation regime}
Write $\kappa_1 = 2/a_1$ and $\alpha_\ell := |a_\ell|, \, \ell \in \mathbb{N}$.

\begin{lemma}\label{lem:lemma3}
Assume $\sup_{j\ge1}j|a_j|<\infty$, $\sum_{j=1}^\infty |a_j|<\infty$, $a_1>0$ and $a_j \neq a_1$ for all $j \ge 2$.

There exist $D_1, D_2 >0,\eta>0$ and $\delta \in (0,\kappa_1)$ such that for any $\{\chi_\ell : \ell \in \mathbb{N}\}$ with $\|\underline{\chi}\|_w \le \eta$ satisfying \eqref{eq:chieq} for some $\kappa \in [\kappa_1 - \delta, \kappa_1 + \delta]$, we must have
\begin{align*}
   & |\chi_\ell| \leq \frac{D_1}{\ell} (D_2|\chi_1|)^\ell,\\
    |\ell(2-\kappa a_\ell)\chi_\ell &- \kappa \sum_{j=1}^{\ell-1} j a_j \chi_j \chi_{\ell-j}| \leq D_1(D_2|\chi_1|)^{\ell+2}, \quad \ell \in \mathbb{N}.
\end{align*}
\end{lemma}

\begin{proof}[Proof of \cref{lem:lemma3}]
We denote $y_{\ell} := \ell |\chi_{\ell}|, \, \ell \in \mathbb{N}$. By \cref{lem:lemma2}, there exist $B_1 \ge 1, \eta>0$ and $\delta \in (0,\kappa_1)$ such that for any $\{\chi_\ell : \ell \in \mathbb{N}\}$ with $\|\underline{\chi}\|_w \le \eta$ satisfying \eqref{eq:chieq} for some $[\kappa_1 - \delta, \kappa_1 + \delta]$,
$$
y_\ell \leq B_1|\chi_1|, \quad \ell \in \mathbb{N}.
$$
This implies that for $\ell\geq 2$, we have by~\cref{lem:lemma1}:
\begin{align*}
  |\ell(2-\kappa a_\ell)\chi_\ell |& \leq \sum_{j < \ell} \kappa \alpha_j y_j  y_{\ell-j} +    \sum_{j > \ell} \kappa (\alpha_j + \alpha_{j-\ell}) y_j  y_{j-\ell} \\
  & \leq (2\kappa  s ) B_1^2 |\chi_1|^2,
\end{align*}
where $s\coloneqq \sum_{j=1}^\infty \alpha_j$. 

Choosing $\delta>0$ sufficiently small, we obtain $\gamma_1 \in (0,1]$ such that $|2 - \kappa a_\ell| \ge \gamma_1>0$ for all $\ell \ge 2$ and $\kappa \in [\kappa_1 - \delta, \kappa_1 + \delta]$. Thus, we have that
\begin{align*}
    y_\ell \leq 2\kappa s \gamma_1^{-1} B_1^2 |\chi_1|^2 \le 2\bar\kappa s\gamma_1^{-1} B_1^2 |\chi_1|^2,\quad \forall \ell \geq 2,
\end{align*}
where we write $\bar\kappa := \kappa_1 + \delta$. (\cref{lem:lemma2} already gives an $O(|\chi_1|^2)$ bound but the above calculation is illustrative of the induction strategy below.)

For future use, we assume $\eta>0$ is small enough such that $|\chi_1| \le \|\underline{\chi}\| \le (2\bar\kappa s B_1\gamma_1^{-1})^{-1}$.
%For later use, we pick $B_1$ large enough such that $4\kappa_2 s B_1\gamma_1^{-1} >1$.

Now, we apply induction. We pose the induction statement for $N\geq 2$ as follows:
\begin{align*}
    \mathcal{I}_N: &|\chi_\ell|\leq \frac{(2\bar\kappa s B_1\gamma_1^{-1})^{\ell-1} B_1 |\chi_1|^\ell}{\ell},\quad \forall 1\leq \ell \leq N-1.\\
    &|\chi_\ell|\leq  \frac{(2\bar\kappa s B_1\gamma_1^{-1})^{N-1} B_1 |\chi_1|^N}{\ell}, \quad \forall  \ell \geq N.
\end{align*}

Clearly, $\mathcal{I}_2$ is true from the above calculations. Now, suppose that $\mathcal{I}_N$ is true for some $N\geq 2$. For $\ell \geq N+1$, we have
\begin{align*}
    |2-\kappa a_\ell| y_\ell \leq T_1^\ell + T_2^\ell,
\end{align*}
where
\begin{align*}
T_1^\ell:=\kappa\sum_{j<\ell}\alpha_j y_jy_{\ell-j}
&\le\kappa \left(\sum_{j<\ell}\alpha_j\right) (2\bar\kappa s B_1\gamma_1^{-1})^{N-1} B_1^2 |\chi_1|^{N+1}.
\end{align*}
Here, we have used the induction hypothesis, which gives $y_j\le B_1|\chi_1|(2\bar\kappa s B_1\gamma_1^{-1}|\chi_1|)^{\min(j,N)-1}$ and, for $\ell\ge N+1$ and $1\le j<\ell$,
\[
\min(j,N)+\min(\ell-j,N)\ge N+1.
\]
For the other contribution,
\begin{align*}
    T_2^\ell &= \kappa \sum_{j >\ell} \alpha_{j} y_j y_{j-\ell} + \kappa \sum_{j >\ell} \alpha_{j-\ell} y_j y_{j-\ell}\\
    &\leq \kappa \left(\sum_{j >\ell} \alpha_{j} + \sum_{j=1}^{\infty}\alpha_j\right) \big((2\bar\kappa s B_1\gamma_1^{-1})^{N-1}B_1|\chi_1|^N\big) (B_1|\chi_1|)\\
    &= \kappa \left(\sum_{j >\ell} \alpha_{j} + \sum_{j=1}^{\infty}\alpha_j\right) (2\bar\kappa s B_1\gamma_1^{-1})^{N-1} B_1^2 |\chi_1|^{N+1}.
\end{align*}
From the above bounds, we obtain that
\begin{align*}
   |2-\kappa a_\ell| y_\ell \leq  2B_1\kappa s(2\bar\kappa s B_1\gamma_1^{-1})^{N-1} B_1 |\chi_1|^{N+1} \le (2B_1\bar\kappa s)(2\bar\kappa s B_1\gamma_1^{-1})^{N-1} B_1 |\chi_1|^{N+1}. %\quad \forall \ell \geq N+1, \kappa \in \big[0, \frac{(1+\delta)\kappa_{2}}{4}\big]
\end{align*}
Since $|2 - \kappa a_\ell| \ge \gamma_1>0$ for all $\ell \ge 2$, we conclude
\begin{align*}
    |\chi_\ell| \leq  \frac{(2\bar\kappa s B_1\gamma_1^{-1})^N B_1 |\chi_1|^{N+1}}{\ell} \quad \forall \ell \geq N+1.
\end{align*}
Thus, $\mathcal{I}_{N+1}$ is true. This proves the first assertion of the lemma. 

To prove the second assertion, we use the first part to conclude that for $\ell \in  \mathbb{N}$, 
\begin{align*}
    &|\ell(2-\kappa a_\ell)\chi_\ell - \kappa \sum_{j=1}^{\ell-1} ja_j \chi_j \chi_{\ell-j}| \leq  \sum_{j >\ell} \kappa (\alpha_j +\alpha_{j-\ell})y_j y_{j-\ell}\\
    &\leq \sum_{j >\ell}\kappa (\alpha_j +\alpha_{j-\ell}) (2\bar\kappa s B_1\gamma_1^{-1})^{j-1} B_1|\chi_1|^j (2\bar\kappa s B_1\gamma_1^{-1})^{j-\ell-1} B_1 |\chi_1|^{j-\ell}\\
    &= \sum_{j >\ell}\kappa (\alpha_j +\alpha_{j-\ell}) (2\bar\kappa s B_1\gamma_1^{-1})^{2j-\ell-2} B^2_1|\chi_1|^{2j-\ell}\\
    & \leq \sum_{j >\ell}\kappa (\alpha_j +\alpha_{j-\ell}) (2\bar\kappa s B_1\gamma_1^{-1})^{\ell} B^2_1|\chi_1|^{\ell+2}\\
    &\le (2\kappa B_1 s)(2\bar\kappa s B_1\gamma_1^{-1})^{\ell} B_1|\chi_1|^{\ell+2}\\
    &\le (2\bar\kappa s B_1 \gamma_1^{-1})(2\bar\kappa s B_1\gamma_1^{-1})^{\ell} B_1|\chi_1|^{\ell+2}\\
    &\leq \frac{1}{2\bar\kappa s \gamma_1^{-1}} (2\bar\kappa s B_1 \gamma_1^{-1} |\chi_1|)^{\ell+2}.
\end{align*}
This proves the desired result.

\end{proof}

We are now ready to prove \cref{thm:main_thm}.

\begin{proof}[Proof of~\cref{thm:main_thm}]
%The supercriticality of bifurcations about $p_0$ follows on noting that $\kappa''(0)>0$ in \cref{thm:CRthm} by \cref{ass:onW} on $W$. 
By \cref{thm:period}(a), there exists an $m$-periodic non-trivial branch of solutions around $\kappa_m = 2/a_m$. This branch is explicitly related to those of the stationary equation \eqref{eq:stat} with potential $W^{(m)}(\theta) = \sum_{\ell=1}^\infty a_{\ell m}\cos \ell \theta, \, \theta \in [0,2\pi)$. In particular, the $\ell$-th Fourier mode $p_\ell= 0$ if $m \nmid \ell$ for any stationary solution $p$ on this branch. The supercriticality/subcriticality of this solution according to the sign of $R_m(W) = \frac{a_m - 2a_{2m}}{a_m - a_{2m}}$ then follow from \cref{thm:CRthm}. We now provide a detailed characterization of this branch. We work with the supercritical case, namely $R_m(W) >0$. The subcritical case follows similarly.

By \cref{lem:lemma1}, $p$ is a solution to \eqref{eq:stat} in $H_s^1(\mathbb{S}^1)$ if and only if, writing $\chi_\ell = p_\ell$, $p$ takes the form
 \begin{align}\label{eq:star}
     p(\theta) = \frac{1}{Z}\exp\bigg(\kappa \sum_{\ell=1}^\infty \chi_\ell a_\ell \cos \ell \theta\bigg).
 \end{align}
where $\{\chi_\ell: \ell \in \mathbb{N}\}$ satisfy \eqref{eq:chieq}.
%Conversely, by representing the test functions in the Fourier basis, it follows that every $p$ of the form in~\eqref{eq:star} with $\{\chi_\ell\}$ satisfying~\eqref{eq:chieq} is a solution to \eqref{eq:stat} in $L_s^2(\mathbb{S}^1)$. 

It is convenient to view the `single branch' obtained from \cref{thm:period}(a) as two separate branches $\{(\underline{p}(s), \kappa(s)) : s >0\}$ (the `plus branch') and $\{(\underline{p}(s), \kappa(s)) : s <0\}$ (the `minus branch'). In each such branch, we can identify a unique solution $p_\kappa$ with a given value of $\kappa$ in the range of $\kappa(\cdot)$. Thus, to obtain the representation given in the theorem, it suffices to exhibit $\eta_m, \delta_m>0$ and, for $\kappa \in (\kappa_m,\kappa_m +\delta_m]$, two non-trivial solutions of~\eqref{eq:stat} in $\mathbb{B}_{\eta_m}(p_0,H_s^1(\mathbb{S}^1))$ whose Fourier coefficients are given by
\begin{align*}
\chi^{\pm}_{\ell m}(\kappa)=(s^{\pm}(\kappa))^\ell(z_{\ell m}(\kappa)+r_{\ell m}(\kappa)),\quad\ell\in\mathbb{N},\qquad\chi_j^{\pm}=0\text{ if }m\nmid j,    
\end{align*}
which correspond to the plus and minus branches.

%It will be conv one with $\chi_m >0$ and one with $\chi_m <0$ (the latter two are obtained by viewing the solution sets for $s<0$ and $s>0$ as separate solutions). So, it suffices to characterize them as in the statement of the theorem. 

%By \cref{thm:period}(a), the non-trivial solutions are explicitly related to those of the stationary equation \eqref{eq:stat} with potential $W^{(m)}$ defined in that theorem. In particular, $\chi_{\ell}= 0$ if $m \nmid \ell$.

Using~\cref{lem:lemma3} applied to the stationary equation \eqref{eq:stat} with potential $W^{(m)}$, we obtain positive constants $D_1, D_2, \delta$ such that for all $\kappa\in [\kappa_m - \delta, \kappa_m + \delta]$,
\begin{align}
    &|\chi_{\ell m}| \leq \frac{D_1}{\ell m} (D_2 |\chi_m|)^\ell, \label{eq:thm1A}\\
    &|\ell m (2-\kappa a_{\ell m}) \chi_{\ell m} - \kappa \sum_{j=1}^{\ell-1} jm \, a_{jm}\chi_{jm}\chi_{(\ell-j)m} | \leq D_1(D_2 |\chi_m|)^{\ell+2},~~~\forall \ell \in \mathbb{N}. \label{eq:thm1B}
\end{align}
Recall that by~\cref{lem:lemma1}, we have
\begin{align*}
    m(2-\kappa a_m)\chi_m = \sum_{j=2}^\infty \kappa (jm\,a_{jm} - (j-1)m\,a_{(j-1)m})\chi_{jm} \chi_{(j-1)m}.
\end{align*}
Thus, by~\eqref{eq:thm1A}, 
\begin{align*}
    m(2-\kappa a_m)\chi_m = \kappa (2m a_{2m} - m a_{m})\chi_{m} \chi_{2m} + O(|\chi_m|^5).
\end{align*}
Now, by using~\eqref{eq:thm1B} (with $\ell=2$), we have
\begin{align*}
    2m(2-\kappa a_{2m})\chi_{2m} = \kappa m a_{m} \chi^2_{m} + O(|\chi_m|^4).
\end{align*}
Using this in the previous equation, we obtain
\begin{align*}
   m(2-\kappa a_m)\chi_m = \frac{\kappa^2 (2m a_{2m}-m a_m) m a_m \chi_m^3}{2m(2-\kappa a_{2m})} + O(|\chi_m|^5), 
\end{align*}
which gives the solutions $\chi_m=0$ and
\begin{align*}
    \chi_m = \underbrace{\pm \sqrt{\frac{2(\kappa a_m-2)(2-\kappa a_{2m})}{\kappa^2a_m (a_m-2a_{2m})}} + O((\kappa a_m - 2)^{3/2}).}_{s^{\pm}(\kappa)}
\end{align*}

If $\chi_m=0$, then by \cref{lem:lemma2} applied to potential $W^{(m)}$, $\underline{\chi} \equiv 0$ corresponding to stationary solution $p=p_0$.

Now suppose $\chi_m\neq 0$. Recall the sequence $\{z_{\ell m} = z_{\ell m}(\kappa)\}$ in~\eqref{eq:zseq}, and define $\hat{\chi}_{\ell m}\coloneqq \chi_m^{\ell} z_{\ell m}, \ell \in \mathbb{N}$. We claim that there exist $\{C'_\ell >0 : \ell \in \mathbb{N}\}$ such that
\begin{align}\label{eq:thm1C}
    |\chi_{\ell m} - \hat{\chi}_{\ell m}| = C'_\ell|\chi_m|^{\ell+2}, \ \forall \ell \in \mathbb{N}.
\end{align}
The above claim is clearly true for $\ell=1$, as $\chi_m=\hat{\chi}_m$. Suppose that it is true for all $\ell \leq L-1$, for $L\geq 2$. By~\eqref{eq:thm1B} and the definition of $\{ \hat{\chi}_{\ell m}\}$, we have that
\begin{align*}
    |L m (2-\kappa a_{Lm})| |\chi_{Lm} - \hat{\chi}_{Lm}| &\leq  \kappa \sum_{j=1}^{L-1}jm\,|a_{jm}||\chi_{jm} \chi_{(L-j)m} - \hat{\chi}_{jm}\hat{\chi}_{(L-j)m}| + D_1(D_2|\chi_m|)^{L+2}\\
    &\leq  C_L \sum_{j=1}^{L-1} jm \,|a_{jm}| \big( |\chi_{jm} - \hat{\chi}_{jm}| |\chi_m|^{L-j} + |\chi_{(L-j)m} - \hat{\chi}_{(L-j)m}||\chi_m|^j   \big)\\
    &\qquad + D_1 (D_2|\chi_m|)^{L+2} \\
    & \leq  C_L \sum_{j=1}^{L-1} jm\, |a_{jm}| C'_j |\chi_m|^{L+2} + D_1 (D_2|\chi_m|)^{L+2},
\end{align*}
for constant $C_L>0$, where the second inequality follows by triangle inequality, ~\eqref{eq:thm1A} and the fact that $|\hat{\chi}_{\ell m}| \leq \tilde{C}_\ell |\chi_m|^\ell$ for some constant $\tilde{C}_\ell>0, \forall \ell$, which can be checked by induction as in the proof of~\cref{lem:lemma3}. The last inequality uses the assertion that \eqref{eq:thm1C} holds for $\ell \le L-1$. This proves the claim in~\eqref{eq:thm1C} for all $\ell \in \mathbb{N}$.

Thus, we conclude that 
\begin{align*}
    \chi_{\ell m} = \chi_m^\ell(z_{\ell m} + O(\chi^2_m)) = \chi_m^\ell (z_{\ell m} + O(\kappa a_m-2)),
\end{align*}
thereby completing the proof of the representation given in the theorem. For each sufficiently small $\eta_m>0$, shrink $\delta_m>0$ so both branches over the stated parameter interval lie in $\mathbb{B}_{\eta_m}(p_0,H_s^1)$. This is possible since their $H^1$ distance from $p_0$ is $O(\sqrt{|\kappa-\kappa_m|})$.

Finally, if $a_\ell \neq a_m$ for all $\ell \neq m$, the uniqueness of the even non-trivial solutions locally around $(p_0,\kappa_m)$ in $H^1_s(\mathbb{S}^1) \times \mathbb{R}_+$ follows from the Fourier characterization of solutions in \cref{lem:lemma1} and the uniqueness of solutions, expressed in terms of Fourier modes, in $\ell_w^2 \times \mathbb{R}_+$ obtained from \cref{thm:CRthm}. 

%Now, we prove part (b). Suppose that $\exists m\in\mathbb{N}$ and $\hat\kappa \in (\kappa_{m-1},\kappa_m)$ (with $\kappa_0\coloneqq 0$) such that for any $\delta>0$, there exists a non-trivial solution in $\mathbb{B}_\delta(p_0,L_s^2(\mathbb{S}^1))$ for some $\kappa \in (\hat\kappa-\delta, \hat\kappa + \delta)$. Then, there exists $\kappa^{(n)} \to \hat{\kappa}$ such that the solution $p_n \neq p_0$ with intensity $\kappa^{(n)}$ with $\underline{\chi}^{(n)} \to 0$ in $L^2(\mathbb{R}^\infty)$. By~\cref{lem:lemma2}, $\exists C>0$ such that 
%\begin{align*}
%    |\chi_\ell^{(n)}| \leq C \sup_{j\leq m} |\chi_j^{(m)}| \ \forall \ell \geq m+1.
%\end{align*}
%We can obtain $\ell\leq m$ and a subsequence $\kappa^{(n_j)}$ such that  $|\chi_\ell^{(n_j)}| =   \max_{i\leq m} |\chi_i^{(n_j)}|,~\forall j$. From~\eqref{eq:chieq}, we thus obtain $C'>0$ such that
%\begin{align*}
%    |\ell (2-\kappa^{(n_j)} a_\ell)||\chi_\ell(\kappa^{(n_j})|\leq C' |\chi_\ell(\kappa^{(n_j)})|^2.
%\end{align*}
%As $|\ell (2-\kappa^{(n_j)} a_\ell)|$ is bounded away from zero for sufficiently large $n_j$ and $\chi_\ell(\kappa^{(n_j)})\to 0$, there exists $j_0\in\mathbb{N}$ such that $\chi_\ell(\kappa^{(n_j})=0~,\forall j \geq j_0$. This implies  $\chi_{\ell}^{(n_j)}\equiv 0, \forall \ell \in \mathbb{N}, j \geq j_0$, which further implies $p_{n_j}=p_0 \, \forall j \geq j_0$, providing a contradiction. This proves the claim in part (b).
 
\end{proof}

\section{Proofs: Phase Transition}

   \begin{proof}[Proof of \cref{thm:bifdpt}]
        We abbreviate $\int$ for $\frac{1}{2\pi}\int_0^{2\pi}$ and write $(p_s,\kappa_s)$ for the element in $H_s^1(\mathbb{S}^1) \times\mathbb{R}_+$ corresponding to $(\underline{p}(s),\kappa(s)) \in \ell^2_w \times \mathbb{R}_+$. Differentiating the free energy $\mathcal{F}(\cdot, \kappa^*)$ along the curve $s \mapsto p_s$, we obtain
        \begin{align*}
            \partial_s\mathcal{F}(p_s,\kappa^*) &= \int\partial_s\left(p_s(\theta) \log p_s(\theta)\right)\d\theta - \frac{\kappa^*}{2}\int\partial_s[W\star p_s (\theta) p_s(\theta)]\d\theta\\
            & = \int(\partial_s p_s(\theta)) \log p_s(\theta)\d\theta - \kappa^*\int W\star p_s (\theta) (\partial_s p_s(\theta))\d\theta,
        \end{align*}
        where the first term in the second equality follows from the fact that $\int p_s(\theta)\partial_s (\log p_s(\theta))\d\theta = \int\partial_sp_s(\theta)\d\theta = 0$ as $p_s$ is a probability density for all $s$, and the second term follows from the observation $\int\int p_s(\theta)W\star \partial_sp_s(\theta)\d\theta = \int\int W\star p_s (\theta) (\partial_s p_s(\theta))\d\theta$ as $W$ is an even function. As, $p_s$ is a solution to \eqref{eq:stat} with parameter $\kappa_s$, we have
        $$
        \log p_s(\theta) = \kappa_s W\star p_s (\theta) + C_s, \quad \theta \in [0, 2\pi),
        $$
        almost everywhere, for some constant $C_s$ possibly depending on $s$ but not $\theta$. Using this above, we get
        \begin{align}\label{pt1}
            \partial_s\mathcal{F}(p_s,\kappa^*) = (\kappa_s - \kappa^*)\int W\star p_s (\theta) (\partial_s p_s(\theta))\d\theta, \quad s \in [0,\delta).
        \end{align}
        From the Fourier representation of $p_s$, we obtain
        \begin{align*}
            \int W\star p_s (\theta) (\partial_s p_s(\theta))\d\theta = \sum_{\ell=1}^{\infty}a_\ell p_\ell(s)p_\ell'(s)
            =as\sum_{i \in \Lambda'}y_i^2 + O(s^2),
        \end{align*}
    which is strictly positive for all $s \in (0, \eta]$ for some $\eta \in (0, \delta]$. Moreover, by the assumed nature of the bifurcation, $\kappa_s < \kappa^*$ for all $s \in (0, \delta)$. Using this in \eqref{pt1}, we obtain
    $$
    \mathcal{F}(p_\eta, \kappa^*) - \mathcal{F}(p_0, \kappa^*) = \int_0^\eta (\kappa_s - \kappa^*)\int W\star p_s (\theta) (\partial_s p_s(\theta))\d\theta < 0.
    $$
    Thus, $p_0$ is not a global minimizer of
$\mathcal{F}(\cdot,\kappa^*)$.
Since $W\in H_s^1(\mathbb{S}^1)$, the potential $-W$
satisfies Assumption (A2) in \cite{carrillo2020long}.
Moreover, $a>0$ implies that $-W$ has a negative Fourier
coefficient. Hence \citet[Proposition 5.3]{carrillo2020long} guarantees a finite transition point
$\kappa_c$. By \citet[Proposition 5.8(b)]{carrillo2020long}, $\kappa_c<\kappa^*$ and the transition
is discontinuous, proving the first assertion of the theorem.

    To prove the second assertion, we highlight the dependence of the potential $W$ on $\mathcal{F}$ and denote it by $\mathcal{F}_{W}$. Observe that, for any $(\rho, \kappa) \in \mathcal{P}_{ac}^+(\mathbb{S}^1) \times \mathbb{R}_+$,
    \begin{align*}
    |\mathcal{F}_W(\rho, \kappa) - \mathcal{F}_{\tilde{W}}(\rho, \kappa)| &\le \frac{\kappa}{2}\int\int|W(\theta - \phi) - \tilde{W}(\theta- \phi)| \rho(\theta)\rho(\phi)\d\theta\d\phi\\
    &\le \frac{\kappa}{2}\|W - \tilde{W}\|_\infty \le \frac{\kappa}{2}\sum_{\ell=1}^\infty|a_\ell - \tilde{a}_\ell|.
    \end{align*}
    Thus,
    \begin{align*}
        \mathcal{F}_{\tilde{W}}(p_\eta, \kappa^*) - \mathcal{F}_{\tilde{W}}(p_0, \kappa^*) &\le \mathcal{F}_W(p_\eta, \kappa^*) - \mathcal{F}_W(p_0, \kappa^*) + \kappa^* \sum_{\ell=1}^\infty|a_\ell - \tilde{a}_\ell|\\
        & < \mathcal{F}_W(p_\eta, \kappa^*) - \mathcal{F}_W(p_0, \kappa^*) + \kappa^* \epsilon.
    \end{align*}
    The above bound and continuity in $\kappa$,
together with
$\mathcal{F}_W(p_\eta,\kappa^*)
<\mathcal{F}_W(p_0,\kappa^*)$,
therefore imply, for $\epsilon \in (0, (\mathcal{F}_W(p_0,\kappa^*) - \mathcal{F}_W(p_\eta, \kappa^*))/\kappa)$,
$$
\mathcal{F}_{\tilde{W}}(p_\eta,\tilde{\kappa}^*)
<
\mathcal{F}_{\tilde{W}}(p_0,\tilde{\kappa}^*).
$$
    Write $\tilde{a}:=\max_{\ell\ge1}\tilde{a}_\ell$. Taking $\epsilon<a/2$,
$
|\tilde{a}-a|
\le \sum_{\ell=1}^\infty|a_\ell-\tilde{a}_\ell|
<\epsilon
$
implies $\tilde{a}>0$ and, setting
$\tilde{\kappa}^*:=2/\tilde{a}$,
$
|\tilde{\kappa}^*-\kappa^*|
\le \frac{4\epsilon}{a^2}.
$
Applying \citet[Propositions 5.3 and 5.8(b)]{carrillo2020long} with interaction potential
$-\tilde{W}$ proves the second assertion.
\end{proof}

    \begin{proof}[Proof of \cref{thm:phasetr}]
As discussed in \cref{sec:221}, the ratio
$I_2(\beta)/I_1(\beta)$ increases strictly for $\beta>0$.
This gives the existence and uniqueness of $\beta_0$.
In the proof below, write $a_n=2I_n(\beta)/\beta$ and
$\kappa^*=2/a_1$.
Let $d\nu(\theta)=d\theta/(2\pi)$ and, for a probability density $\rho$,
define
\[
\widehat\rho_n=\int_{\mathbb S^1}e^{in\theta}\rho(\theta)\,d\nu(\theta),
\qquad H(\rho)=\int_{\mathbb S^1}\rho\log\rho\,d\nu.
\]
For even densities, $\widehat\rho_n$ agrees with the cosine coefficient
used elsewhere in the paper.

We use the sharp logarithmic Hardy--Littlewood--Sobolev inequality
on the circle \citet[Theorem~1, equation~(7)]{carlenloss1992loghls}, which gives 
\begin{equation}\label{hls}
H(\rho)\ge
-\iint\log|e^{i\theta}-e^{i\phi}|\,
\rho(\theta)\rho(\phi)\,d\nu(\theta)d\nu(\phi)
\end{equation}
for every probability density of finite entropy.
To rewrite the right-hand side, use
\[
-\log|e^{i\theta}-e^{i\phi}|
=-\log\!\left(2\left|\sin\frac{\theta-\phi}{2}\right|\right)
=\sum_{n\ge1}\frac{\cos(n(\theta-\phi))}{n}.
\]
For each \(n\),
\[
\iint\cos(n(\theta-\phi))\rho(\theta)\rho(\phi)
\,d\nu(\theta)d\nu(\phi)
=
\left(\int\rho\cos(n\theta)\,d\nu\right)^2
+\left(\int\rho\sin(n\theta)\,d\nu\right)^2
=|\widehat\rho_n|^2.
\]
Thus, in the Fourier formulation, \eqref{hls} becomes
\begin{equation}\label{eq:transformer-loghls}
H(\rho)\ge\sum_{n=1}^{\infty}\frac{|\widehat\rho_n|^2}{n}.
\end{equation}

\medskip
\noindent\emph{The case $0<\beta\le\beta_0$.}
Tur\'an's inequality for the modified Bessel functions implies that
$I_{j+1}(\beta)/I_j(\beta)$ decreases with $j$.
Consequently,
\[
\frac{a_n}{a_1}=\frac{I_n(\beta)}{I_1(\beta)}
=\prod_{j=1}^{n-1}\frac{I_{j+1}(\beta)}{I_j(\beta)}
\le2^{-(n-1)}\le\frac1n,\qquad n\ge1,
\]
where the product is interpreted as one when $n=1$.
The final inequality is strict for $n\ge3$.
Using \eqref{eq:transformer-loghls} and the Fourier representation
of the interaction energy gives
\begin{align}\label{eq:transformer-energy-bound}
\mathcal F(\rho,\kappa^*)
&=H(\rho)-\sum_{n\ge1}\frac{a_n}{a_1}|\widehat\rho_n|^2\notag\\
&\ge\sum_{n\ge1}\left(\frac1n-\frac{a_n}{a_1}\right)
|\widehat\rho_n|^2\ge0
=\mathcal F(1,\kappa^*).
\end{align}
Thus the uniform density is a minimizer at $\kappa^*$.

We next prove uniqueness, including at $\beta=\beta_0$.
If $\rho$ is any minimizer, equality in
\eqref{eq:transformer-energy-bound} forces
$\widehat\rho_n=0$ for every $n\ge3$.
It is therefore a trigonometric polynomial of degree at most two.
On the other hand, integrating the stationary equation \eqref{eq:stat} once and using periodicity gives
\begin{equation}\label{statfp}
\rho'=\kappa^*\rho(W_\beta\star\rho)'.
\end{equation}
Suppose $\rho$ is nonconstant, and let $d\in\{1,2\}$ be its highest
nonzero Fourier index. 
After rotating the angular coordinate, we can write
\[
\rho(\theta)=1+2r\cos(d\theta)+q(\theta),\qquad r>0,
\]
where $q$ contains only frequencies strictly below $d$.
The convolution formula gives
\[
(W_\beta\star\rho)'(\theta)
=-d a_d r\sin(d\theta)+v(\theta),
\]
where $v$ also contains only frequencies strictly below $d$.
Multiplying the two terms of frequency $d$, we obtain
\[
(2r\cos(d\theta))(-d a_d r\sin(d\theta))
=-d a_d r^2\sin(2d\theta).
\]
Every other product involves at least one frequency below $d$,
so its frequency is at most $2d-1$. Therefore,
\[
\kappa^*\rho(W_\beta\star\rho)'
=-\kappa^*d a_d r^2\sin(2d\theta)+R(\theta),
\]
where $R$ contains only frequencies at most $2d-1$.
On the other hand, $\rho'$ contains only frequencies at most $d$.
Comparing the coefficients of $\sin(2d\theta)$ on the two sides of \eqref{statfp} thus gives $\kappa^*d a_d r^2=0$.
Since $\kappa^*>0$ and $a_d>0$, this contradicts $r>0$.
Hence $\rho$ is constant and, by normalization, $\rho\equiv1$.

By \citet[Proposition 5.8(a)]{carrillo2020long}, $\kappa^*$ is the
transition point with continuity in $L^1$. To obtain the $H^1$ continuity
required by \cref{def:contpt}, we use the following argument, which applies
to any $W\in H^1(\mathbb S^1)$ whenever the uniform density is the unique
minimizer at a finite transition point $\kappa_c$.
Let $\kappa_n\downarrow\kappa_c$ and let $\rho_n$ be any minimizer at
$\kappa_n$. By \cref{thm:global}(i) (whose statement and proof are independent of the current result), every subsequence has a further
subsequence converging in $L^1$ to a minimizer at $\kappa_c$.
Uniqueness therefore gives $\rho_n\to1$ in $L^1$.
Since $W\in H^1$ is bounded and $K:=\sup_n\kappa_n<\infty$,
the normalized Gibbs formula $\rho(\cdot)
=
\frac{\exp\!\big(\kappa(W\star\rho)(\cdot)\big)}
{\frac{1}{2\pi}\int_0^{2\pi}
\exp\!\big(\kappa(W\star\rho)(\phi)\big)\,\d\phi}$ gives
\[
\|\rho_n\|_\infty\le e^{2K\|W\|_\infty}=:C_0.
\]
Moreover, $W\star\rho_n\in H^1$, so the Gibbs formula and the Sobolev chain rule imply $\rho_n\in H^1$. Differentiating weakly and using $W'\star1=0$ yields $\rho_n'=\kappa_n\rho_n[W'\star(\rho_n-1)]$.
Consequently, using uniform boundedness of $\{\rho_n\}$ and Young's inequality respectively give
\begin{align}\label{eq:l1toh1}
\|\rho_n-1\|_2^2
&\le (C_0+1)\|\rho_n-1\|_1\longrightarrow0,\notag\\
\|\rho_n'\|_2
&\le C\|W'\|_2\|\rho_n-1\|_1\longrightarrow0,
\end{align}
where $C$ is independent of $n$.
Hence $\|\rho_n-1\|_{H^1}\to0$.
Applying this argument to $W=W_\beta$ and $\kappa_c=\kappa^*$ proves
continuity in the sense of \cref{def:contpt}.

\medskip

\noindent\emph{The case $\beta>\beta_0$.}
Here, by the choice of $\beta_0$ and the explicit representation $a_n=2I_n(\beta)/\beta$, $a_1/2<a_2<a_1$, so \cref{cor:dpt}(i) gives a discontinuous
transition at some $\kappa_c<\kappa^*$.
At $\kappa_c$, the uniform density is a minimizer by definition, and
it coexists with a nonuniform
minimizer; see \citet[Proposition~5.6]{carrillo2020long}.

\medskip
\noindent\emph{The local branch at $\beta=\beta_0$.}
Evaluate all coefficients at $\beta_0$ and put
$a=a_1$, $\gamma=a_3/a$. Since $a_2=a/2$, the curvature formula
in \cref{thm:CRthm} gives $\kappa''(0)=0$.
Translation by $\pi$ changes $p_1=s$ to $-s$; local uniqueness
therefore gives $\kappa(-s)=\kappa(s)$ and
$p_n(-s)=(-1)^np_n(s)$.
Smoothness of the branch implies $\kappa(s)=2/a+O(s^4)$.

Lemma \ref{lem:lemma3} gives geometric bounds on the Fourier
coefficients, also in this zero-curvature case.
The equations \eqref{eq:Fid} for the second and third coefficients
then yield
\[
p_2(s)=s^2+O(s^4),\qquad
p_3(s)=\frac{2}{3(1-\gamma)}s^3+O(s^5).
\]
In the first Fourier equation, the term involving $p_2p_1$
vanishes because $2a_2-a_1=0$.
The remaining terms give
\begin{align*}
(2-\kappa(s)a)s
&=\kappa(s)(3a_3-a)p_3(s)p_2(s)+O(s^7)\\
&=\frac{4(3\gamma-1)}{3(1-\gamma)}s^5+O(s^7),
\end{align*}
where the geometric bound controls the sum over higher modes.
Dividing by $s$ gives
\[
\kappa(s)=\frac2a+
\frac{4(1-3\gamma)}{3a(1-\gamma)}s^4+O(s^6).
\]
The Bessel recurrence $I_1(\beta_0)-I_3(\beta_0)=4I_2(\beta_0)/\beta_0$ (see \citep[p.~256]{baricz2010turan}), together with
$I_2(\beta_0)/I_1(\beta_0)=1/2$, gives $\gamma=1-2/\beta_0$.
Substitution proves \eqref{eq:transformer-quartic}.
Finally, strict Tur\'an inequality gives
$\gamma<(I_2/I_1)^2=1/4$, so $1-3\gamma>0$.
The quartic coefficient is therefore positive, completing the proof.
\end{proof}
\section{Proofs: Global properties}

\begin{proof}[Proof of \cref{cor:explicitbif}]
    Using the Fourier modes of $W_{Poi}(\theta) := -\log\left(2\sin(\theta/2)\right)$ in \cref{lem:lemma1}, we see that the second sum in \eqref{eq:Fid} drops out and the Fourier modes $\{p_\ell : \ell \ge 1\}$ of any solution $p$ in $H_s^1(\mathbb{S}^1)$ must satisfy
    \begin{align*}
        \ell(2-\kappa \ell^{-1}) p_\ell = \kappa \sum_{j < \ell} p_j p_{\ell-j}, \quad \ell \ge 1.
    \end{align*}
    In particular, if $\kappa\notin\{2k:k\in\mathbb{N}\}$, $p_1=0$ and by the above recursion, $p_\ell=0$ for all $\ell \ge 1$. Thus, for such $\kappa$, $p=p_0$ is the only stationary solution.

    If $\kappa = \kappa^* := 2\ell^*$ for some $\ell^* \in \mathbb{N}$, we see from the above recursion that $p_\ell = 0$ for all $\ell < \ell^*$, and the equation corresponding to $\ell = \ell^*$ takes a $0=0$ form. Thus, by induction, we conclude $p_\ell=0$ for all $\ell$ which is not a multiple of $\ell^*$ and for $k \ge 2$,
    \begin{align*}
        2k\ell^*(1-k^{-1})p_{k\ell^*} = 2\ell^*\sum_{j=1}^{k-1}p_{j\ell^*} p_{(k-j)\ell^*}
    \end{align*}
    which gives
    $$
    p_{k\ell^*} = p_{\ell^*}^k \ \forall \ k \ge 2.
    $$
    Further, $p$ lies in $H_s^1(\mathbb{S}^1)$ if and only if $p_{\ell^*} = r$ with $|r|<1$ and we obtain the one-parameter family of non-trivial solutions at $\kappa = 2\ell^*$ given by $\{P_{r,\ell^*} : r \in (-1,1)\}$ in the theorem, which follows from the identity
    $$
    1 + 2\sum_{n=1}^{\infty}r^n\cos(n\theta) = \operatorname{Re}\left(\frac{1+re^{i\theta}}{1-re^{i\theta}}\right) = \frac{1-r^2}{1-2r\cos(\theta) + r^2}, \quad \theta \in [0,2\pi), r \in (-1,1).
    $$
    The weak convergence of $P_{r,\ell^*}$ as $r \to \pm 1$ follows from the `approximate identity' property of the Poisson kernel \citet[Chapter 11]{rudin1987real}.
\end{proof}

\begin{proof}[Proof of \cref{thm:finkur}]
    As in \cref{thm:period}(a), consider the potential $W^{(L)}(\theta) = a_L \cos \theta, \ \theta \in [0, 2 \pi)$. By well-known results on the Kuramoto model (see \cite{bertini2010dynamical} and \citet[Section 6.1]{carrillo2020long}), there exists a unique global non-trivial branch for this potential given by
    $$
    \pi^{Kur}_\kappa(\theta) := \frac{1}{I_0(\kappa a_L r_\kappa)}\exp\left(\kappa a_L r_\kappa\cos \theta\right), \quad \theta \in [0, 2\pi), \quad \kappa \in (2/a_L, \infty),
    $$
    where $r_\kappa$ is the unique positive solution to
    $$
    r_\kappa = \frac{I_1(\kappa a_L r_\kappa)}{I_0(\kappa a_L r_\kappa)}.
    $$
    This one-sided branch can be extended to a global two-sided branch $\{(\pi^{Kur}(s), \kappa(s)): s \in (- \infty, \infty)\}$ by using a signed amplitude:
\[
\pi^{Kur}(s)(\theta)=\frac{e^{s\cos\theta}}{I_0(s)},\qquad
\kappa(s)=\frac{sI_0(s)}{a_LI_1(s)}\quad(s\ne0),\qquad
\kappa(0)=\frac2{a_L}.
\]
Since $I_1(s)/I_0(s)=s/2+O(s^3)$, this family extends smoothly through
$s=0$, where $\pi^{Kur}(0)=1$. For $s<0$, the density is the rotation
through $\pi$ of the positive-amplitude density, so it gives the opposite
arm of the branch. The result now follows by \cref{thm:period}(a).
\end{proof}

\begin{proof}[Proof of \cref{thm:global}]
By \citet[Theorem 2.2, equation (3), and Proposition 2.3]{chayes2010mckean}, free energy
minimizers exist over all nonnegative normalized densities and satisfy
the Gibbs equation
\begin{equation}\label{eq:globalGibbs}
\rho(\cdot)
=
\frac{\exp\!\big(\kappa(W\star\rho)(\cdot)\big)}
{\frac{1}{2\pi}\int_0^{2\pi}
\exp\!\big(\kappa(W\star\rho)(\phi)\big)\,\d\phi}.
\end{equation}
Since $W$ is continuous, every such minimizer is strictly positive
and continuous, with bounded logarithm. Thus $M(\kappa)$ is nonempty,
and allowing nonnegative densities does not change the infimum.

We now prove the Lipschitz continuity of $m(\cdot)$.
For any $\rho\in\mathcal P_{ac}^+(\mathbb S^1)$
with finite entropy and $\kappa_1,\kappa_2>0$,
    $$
    \left|\mathcal{F}(\rho,\kappa_1) - \mathcal{F}(\rho,\kappa_2)\right| = \frac{|\kappa_1 - \kappa_2|}{2}\left|\mathcal{E}(\rho,\rho)\right| \le \frac{|\kappa_1 - \kappa_2|}{2}\|W\|_\infty.
    $$
    Thus, taking infimum with respect to $\rho$,
    $$
    \inf_{\rho \in \mathcal{P}_{ac}^+(\mathbb{S}^1)}\mathcal{F}(\rho,\kappa_1) \le \inf_{\rho \in \mathcal{P}_{ac}^+(\mathbb{S}^1)}\mathcal{F}(\rho,\kappa_2) + \frac{|\kappa_1 - \kappa_2|}{2}\|W\|_\infty
    $$
    and a symmetric inequality holds for $\kappa_1$ and $\kappa_2$ interchanged. Thus,
    $$
    |m(\kappa_1) - m(\kappa_2)| \le \frac{|\kappa_1 - \kappa_2|}{2}\|W\|_\infty
    $$
    proving global Lipschitz continuity of $m(\cdot)$.
    
    To prove the sequential compactness in $\mathbb{R}_+ \times L^1(\mathbb{S}^1)$ (which also gives compactness of $M(\kappa)$ in $L^1(\mathbb{S}^1)$ for any fixed $\kappa$), assume $\kappa_n \to \kappa^\circ$ and $\rho_n \in M(\kappa_n)$. Without loss of generality that $\sup_n\kappa_n \le K < \infty$. By compactness of $\mathbb{S}^1$, which implies tightness of laws $\{\mu_n\}$ with densities given by $\{\rho_n\}$, extract a subsequence $\{n_j\}$ such that $\mu_{n_j}$ converge weakly to a probability measure $\mu^\circ$ on $\mathbb{S}^1$. 
    
    Note that, for any $n \in \mathbb{N}$, recalling $p_0$ denotes the (density of the) uniform measure on $\mathbb{S}^1$,
    $
    \mathcal{F}(p_0,\kappa_n) = 0.
    $
    Using the fact that $\rho_n$ minimizes $\mathcal{F}(\cdot, \kappa_n)$ and the above bound,
    \begin{align*}
        \int \rho_n(\theta) \log \rho_n(\theta)\d\theta = \mathcal{F}(\rho_n, \kappa_n) + \frac{\kappa_n}{2}\mathcal{E}(\rho_n, \rho_n)
        \le \mathcal{F}(p_0, \kappa_n) + \frac{K}{2} \|W\|_\infty \le K\|W\|_\infty.
    \end{align*}
    Thus, by lower semicontinuity of entropy established, for example, in \citet[Lemma 2.4(b)]{budhiraja2019analysis}, we conclude that $\mu^\circ$ has a density $\rho^\circ$ and
 \begin{align}\label{glo1}
 \int\rho^\circ(\theta)\log \rho^\circ(\theta) \d\theta \le \lim\inf_{j \to \infty}\int \rho_{n_j}(\theta) \log \rho_{n_j}(\theta)\d\theta \le K\|W\|_\infty.
 \end{align}
 Moreover, as $\mu_{n_j} \otimes \mu_{n_j}$ weakly converges to $\mu^\circ \otimes \mu^\circ$ and $W$ is bounded continuous,
 \begin{align}\label{glo2}
 \mathcal{E}(\rho_{n_j}, \rho_{n_j}) \to \mathcal{E}(\rho^\circ, \rho^\circ) \ \text{ as } j \to \infty.
 \end{align}
From \eqref{glo1} and \eqref{glo2}, we conclude that for any absolutely continuous nonnegative normalized density $\rho$,
\begin{align*}
    \mathcal{F}(\rho^\circ, \kappa^\circ) \le \liminf_{j \to \infty}\mathcal{F}(\rho_{n_j}, \kappa_{n_j}) \le \liminf_{j \to \infty}\mathcal{F}(\rho, \kappa_{n_j}) = \mathcal{F}(\rho, \kappa^\circ),
\end{align*}
which first shows that $\rho^\circ$ minimizes over all nonnegative normalized densities. Applying \eqref{eq:globalGibbs} then gives a strictly positive continuous density with finite entropy, so $\rho^\circ\in\mathcal P_{ac}^+(\mathbb S^1)$ and hence $\rho^\circ\in M(\kappa^\circ)$. 

To complete the proof of (i), it thus remains to prove that $\rho_{n_j} \to \rho^\circ$ in $L^1(\mathbb{S}^1)$. Combining $\rho^\circ \in M(\kappa^\circ)$ with the continuity of $m(\cdot)$ and \eqref{glo2},
 $$
 \int \rho_{n_j}(\theta) \log \rho_{n_j}(\theta)\d\theta = m(\kappa_{n_j}) + \frac{\kappa_{n_j}}{2}\mathcal{E}(\rho_{n_j}, \rho_{n_j}) \to m(\kappa^\circ) + \frac{\kappa^\circ}{2}\mathcal{E}(\rho^\circ, \rho^\circ) = \int\rho^\circ(\theta)\log \rho^\circ(\theta) \d\theta
 $$
 as $j \to \infty$. By the Gibbs equation \eqref{eq:globalGibbs}, $\log\rho^\circ=\kappa^\circ W\star\rho^\circ+C$ for some constant $C$. Thus $\rho^\circ$ is strictly positive and $\log\rho^\circ$ is bounded and continuous, using only the assumed continuity of $W$.
 
 The above observations give convergence in relative entropy:
 \begin{align*}
     R\left(\rho_{n_j} || \rho^\circ\right) &:= \int \rho_{n_j}(\theta) \log\left(\frac{\rho_{n_j}(\theta)}{\rho^\circ(\theta)}\right)\d\theta = \int \rho_{n_j}(\theta) \log \rho_{n_j}(\theta)\d\theta - \int \rho_{n_j}(\theta) \log \rho^\circ(\theta)\d\theta\\
     &\rightarrow \int\rho^\circ(\theta)\log \rho^\circ(\theta) \d\theta - \int\rho^\circ(\theta)\log \rho^\circ(\theta) \d\theta=0.
 \end{align*}
 Thus, by Pinsker's inequality (\citet[Page 88]{tsybakov2008nonparametric}),
 $$
 \|\rho_{n_j} - \rho^\circ\|_{L^1(\mathbb{S}^1)} \le 2\pi\sqrt{2R\left(\rho_{n_j} || \rho^\circ\right)} \to 0,
 $$
completing the proof of (i).

Now, we prove (ii). We have already shown that $m(\cdot)$ is globally Lipschitz. Recalling that $m(\kappa) = \inf_{\rho \in \mathcal{P}_{ac}^+(\mathbb{S}^1)}\mathcal{F}(\rho, \kappa)$ and $\mathcal{F}(\rho, \kappa)$ is affine in $\kappa$, concavity of $m(\cdot)$ follows on noting for any $0<\kappa_1 < \kappa_2$ and $\alpha \in [0,1]$,
\begin{align*}
m(\alpha\kappa_1 + (1 - \alpha)\kappa_2) &= \inf_{\rho \in \mathcal{P}_{ac}^+(\mathbb{S}^1)}\mathcal{F}(\rho,\alpha\kappa_1 + (1 - \alpha)\kappa_2) = \inf_{\rho \in \mathcal{P}_{ac}^+(\mathbb{S}^1)}\left(\alpha\mathcal{F}(\rho, \kappa_1) + (1-\alpha)\mathcal{F}(\rho, \kappa_2)\right)\\ 
&\ge \alpha\inf_{\rho \in \mathcal{P}_{ac}^+(\mathbb{S}^1)}\mathcal{F}(\rho, \kappa_1) + (1-\alpha)\inf_{\rho \in \mathcal{P}_{ac}^+(\mathbb{S}^1)}\mathcal{F}(\rho, \kappa_2) = \alpha m(\kappa_1) + (1-\alpha)m(\kappa_2).
\end{align*}
To get the existence and representation for the left and right derivatives, observe that for any $\rho \in \mathcal{P}_{ac}^+(\mathbb{S}^1)$, $\kappa>0$ and $h>0$,
\begin{align*}
    m(\kappa + h) \le \mathcal{F}(\rho, \kappa) - \frac{h}{2}\mathcal{E}(\rho, \rho).
\end{align*}
Taking $\rho \in M(\kappa)$, the above gives
\begin{align}\label{en1}
\limsup_{h \downarrow 0}\frac{ m(\kappa + h) - m(\kappa)}{h} \le -\frac{1}{2}\sup_{\rho \in M(\kappa)}\mathcal{E}(\rho, \rho).
\end{align}
Now, consider $h_n \downarrow 0$ such that
$$
\liminf_{h \downarrow 0}\frac{ m(\kappa + h) - m(\kappa)}{h} = \lim_{n \to \infty}\frac{ m(\kappa + h_n) - m(\kappa)}{h_n}
$$
and $\rho_n \in M(\kappa + h_n)$. Note that
$$
m(\kappa + h_n) - m(\kappa) \ge \mathcal{F}(\rho_n, \kappa + h_n) - \mathcal{F}(\rho_n, \kappa) = -\frac{h_n}{2}\mathcal{E}(\rho_n, \rho_n).
$$
Moreover, by part (i), there is a subsequence $\{n_j\}$ such that $\rho_{n_j} \rightarrow \rho^* \in M(\kappa)$ in $L^1(\mathbb{S}^1)$ (and thus, weakly), and as $W$ is bounded and continuous,
$
\mathcal{E}(\rho_{n_j}, \rho_{n_j}) \to \mathcal{E}(\rho^*, \rho^*).
$
Thus, 
\begin{align}\label{en2}
    \liminf_{h \downarrow 0}\frac{ m(\kappa + h) - m(\kappa)}{h} = \lim_{j \to \infty}\frac{ m(\kappa + h_{n_j}) - m(\kappa)}{h_{n_j}} \ge -\frac{1}{2}\mathcal{E}(\rho^*, \rho^*) \ge -\frac{1}{2}\sup_{\rho \in M(\kappa)}\mathcal{E}(\rho, \rho).
\end{align}
From \eqref{en1} and \eqref{en2}, we conclude that $\lim_{h \downarrow 0}\frac{ m(\kappa + h) - m(\kappa)}{h}$ exists and
$$
m'_+(\kappa) := \lim_{h \downarrow 0}\frac{ m(\kappa + h) - m(\kappa)}{h} = -\frac{1}{2}\sup_{\rho \in M(\kappa)}\mathcal{E}(\rho, \rho).
$$
The left derivative can be handled similarly. The monotonicity properties of $\sup_{\rho \in M(\kappa)}\mathcal{E}(\rho, \rho)$ and $\inf_{\rho \in M(\kappa)}\mathcal{E}(\rho, \rho)$ follow exactly as in \citet[Proposition 2.4]{chayes2010mckean}. This completes the proof of (ii).

Finally, to prove (iii), note that by the concavity of $m$, for any $0<a<b$, 
$$
m'_+(a) \ge \frac{m(b) - m(a)}{b-a} \ge m'_-(b).
$$ 
Let $\kappa>0$ be a continuity point of both $m'_+$ and $m'_-$. Taking $a=\kappa - h$, $b=\kappa+h$ (for small enough $h>0$) above and taking a limit as $h \downarrow 0$, we obtain
$$
m'_+(\kappa) \ge \lim_{h \downarrow 0}\frac{m(\kappa + h) - m(\kappa-h)}{2h} \ge m'_-(\kappa).
$$
But, from part (ii), $m'_+(\kappa) \le m'_-(\kappa)$. Hence,
$$
m'_+(\kappa) = \lim_{h \downarrow 0}\frac{m(\kappa + h) - m(\kappa-h)}{2h} = m'_-(\kappa),
$$
which gives the differentiability of $m$ at $\kappa$. Thus, the points of non-differentiability of $m$ form a subset of the set of discontinuity points of either $m'_+$ or $m'_-$ or both. But this set is at most countable as $m'_+$ and $m'_-$ are both monotone, proving (iii).
 \end{proof}

 \textbf{Acknowledgements: }KB is supported in part by National Science Foundation (NSF) grant DMS-2413426. SB is supported in part by the NSF-CAREER award DMS-2141621 and the NSF-RTG award DMS-2134107. PR is supported in part by NSF grants DMS-2022448 and CCF-2106377.

\appendix
\section{Justification for~\cref{rem:transformer_explicit}}\label{sec:adddetails}
Here, we provide details about the explicit analytical description of the bifurcating solutions \(p_\kappa^{\pm}(\theta)\) near the bifurcation point via various approximations. Throughout, the small-amplitude limit is taken at fixed $\beta,m$ with $a_m-2a_{2m}\ne0$; the constants in amplitude remainders may depend on $\beta,m$ and need not be uniform as $\beta\to\infty$. All exponent errors below are uniform in $\theta$. Note that we have the following:
\begin{enumerate}
    \item {\it Amplitude term \(s^{\pm}(\kappa)\).} The leading order in \(\varepsilon=\kappa a_m-2\), is given by
    \[
    s^{\pm}(\kappa)=\pm\sqrt{\varepsilon\,\frac{a_m-a_{2m}}{a_m-2a_{2m}}}+O(|\varepsilon|^{3/2}).
    \]
    Substituting \(a_\ell=\tfrac{2I_\ell(\beta)}{\beta}\)  then yields
    \[
    s^{\pm}(\kappa)
    =\pm\sqrt{\left(\kappa\frac{2I_m(\beta)}{\beta}-2\right)\frac{I_m(\beta)-I_{2m}(\beta)}{I_m(\beta)-2I_{2m}(\beta)}}
    +O(|\kappa a_m-2|^{3/2}).
    \]
    The sign/branch and whether these appear for \(\kappa>\kappa_m\) (supercritical) or \(\kappa<\kappa_m\) (subcritical) is determined by the $m$-signature, i.e., the sign of $R_m(\beta)= (a_m-2a_{2m})/ (a_m-a_{2m})$, as described in \cref{thm:main_thm}.
    \item {\it Recursive coefficients \(z_{\ell m}(\kappa)\).} By definition \(z_m(\kappa)=1\). For \(\ell=2\) the recursion \eqref{eq:zseq} gives
    \[
    z_{2m}(\kappa)=\frac{\kappa a_m}{2(2-\kappa a_{2m})},
    \]
    and at \(\kappa=\kappa_m\) this simplifies to
    \[
    z_{2m}(\kappa_m)=\frac{a_m}{2(a_m-a_{2m})}
    =\frac{I_m(\beta)}{2\bigl(I_m(\beta)-I_{2m}(\beta)\bigr)}.
    \]
    \item {\it Explicit two-term approximation. }Keeping only the dominant \(\ell=1\) and \(\ell=2\) terms in the exponent and using the leading-order amplitudes above, one obtains the approximation
\[
p_\kappa^{\pm}(\theta)\approx \frac{1}{Z}\exp\Big(
\underbrace{\kappa a_m s^{\pm}(\kappa)\cos(m\theta)}_{\text{primary mode}}
+\underbrace{\kappa a_{2m}\,z_{2m}(\kappa_m)\,(s^{\pm}(\kappa))^2\cos(2m\theta)}_{\text{second harmonic}}
\Big),
\]
where \(Z\) normalizes the density. The omitted terms (higher harmonics, the remainders \(r_{\ell m}\), and error on replacing $z_{2m}(\kappa)$ by $z_{2m}(\kappa_m)$) contribute \(O_{\beta,m}(|\kappa a_m-2|^{3/2})\) or smaller corrections in the exponent, so the leading structure is a primary cosine wave of frequency \(m\) with a smaller second-harmonic contribution of frequency \(2m\) whose magnitude is quadratic in the primary amplitude.
\end{enumerate}
\medskip

\paragraph{\it Small-\(\beta\) explicit approximation.}
For small \(\beta>0\) use the leading small-\(\beta\) approximation
\[
I_\ell(\beta)\approx\frac{(\beta/2)^\ell}{\ell!}\quad(\beta\downarrow0),
\qquad
a_\ell=\frac{2I_\ell(\beta)}{\beta}\approx\frac{\beta^{\ell-1}}{2^{\ell-1}\,\ell!}.
\]
Set
\[
\kappa_m=\frac{2}{a_m},\qquad \Delta\coloneqq \kappa a_m-2>0 \quad(\text{so } \Delta = a_m(\kappa-\kappa_m)).
\]
Put $r=a_{2m}/a_m=O_m(\beta^m)$. For $0<\Delta\ll1$ and small $\beta$,
\[
s^{\pm}(\kappa)=\pm\sqrt\Delta
+O_m(\beta^m\sqrt\Delta)+O_{\beta,m}(\Delta^{3/2}),
\qquad z_{2m}(\kappa_m)=\frac{a_m}{2(a_m-a_{2m})}=\frac12+O_m(\beta^m).
\]
Using these estimates and $\kappa a_m=2+\Delta$, the two-term approximation becomes
\[
\begin{split}
p_\kappa^{\pm}(\theta)=\frac1Z\exp\Big(&\pm2\sqrt\Delta\cos(m\theta)
+\tfrac12\kappa a_{2m}\Delta\cos(2m\theta)
+O_m(\beta^m\sqrt\Delta)+O_{\beta,m}(\Delta^{3/2})\Big).
\end{split}
\]
Equivalently, using \(\Delta=a_m(\kappa-\kappa_m)\), and substituting the small-\(\beta\) leading form of \(a_m\),
\[
a_m\approx\frac{\beta^{m-1}}{2^{\,m-1}m!},\qquad
a_{2m}\approx\frac{\beta^{2m-1}}{2^{\,2m-1}(2m)!},
\]
gives the most explicit small-\(\beta\) form
\[
\begin{split}
p_\kappa^{\pm}(\theta)=\frac1Z\exp\Bigg(&\pm2\sqrt{\frac{\beta^{m-1}}{2^{m-1}m!}\,(\kappa-\kappa_m)}\cos(m\theta)
+O_m\big((\beta^m+\beta^2)\sqrt\Delta+\beta^m\Delta\big)
+O_{\beta,m}(\Delta^{3/2})\Bigg).
\end{split}
\]
\medskip

\paragraph{\it Large-\(\beta\) explicit approximation.} Note that for fixed $m$ and $\beta \to \infty$, we have the estimate from~\cref{aellbetalargebeta} for $a_m$. Hence, we have

\[
\begin{aligned}
a_m - a_{2m}
&= C(\beta)\!\left[\frac{3m^2}{2\beta} + O\!\left(\frac{1}{\beta^{2}}\right)\right], \\[6pt]
a_m - 2a_{2m}
&= C(\beta)\!\left[-1 + \frac{28m^2 - 1}{8\beta} + O\!\left(\frac{1}{\beta^{2}}\right)\right],
\end{aligned}
\qquad
\text{where}\quad C(\beta)=\frac{2e^{\beta}}{\beta^{3/2}\sqrt{2\pi}}.
\]
%\[
%a_m-a_{2m}\sim C(\beta)\frac{3m^2}{2\beta}+O(\beta^{-3}),\qquad
%a_m-2a_{2m}\sim -C(\beta)+O(\beta^{-1}),
%\]
%where $C(\beta)=\dfrac{2e^{\beta}}{\beta^{3/2}\sqrt{2\pi}}$. 
So $a_m-a_{2m}>0$ while $a_m-2a_{2m}<0$ for large $\beta$. Therefore, the ratio
\[
\frac{a_m-2a_{2m}}{a_m-a_{2m}}<0,
\]
and the bifurcation is \emph{subcritical}, i.e., non-trivial solutions appear for $\kappa<\kappa_m$.

\medskip
Let $\delta\coloneqq 2-\kappa a_m$ (small and positive on the subcritical branch). Substituting the above asymptotics in the general amplitude formula,
gives the explicit large-$\beta$ scaling
\[
s^{\pm}(\kappa)=\pm\sqrt{\frac{3m^2}{2\beta}\,\delta}
+O_m(\sqrt\delta\,\beta^{-3/2})+O_{\beta,m}(\delta^{3/2}).
\]
Similarly, for the recursive coefficient at the bifurcation point,
\[
z_{2m}(\kappa_m)=\frac{a_m}{2(a_m-a_{2m})}
\sim \frac{\beta}{3m^2}\big(1+O(\beta^{-1})\big).
\]
Keeping the dominant $\ell=m$ (primary) and $\ell=2m$ (second-harmonic) terms in the exponent and using $\kappa a_m=2-\delta$, we obtain the explicit two-term approximation for the density near the bifurcation as
\[
\begin{split}
p_\kappa^{\pm}(\theta)=\frac1Z\exp\Bigg(&\pm2\sqrt{\frac{3m^2\delta}{2\beta}}\cos(m\theta)
+\delta\cos(2m\theta)
+O_m\big(\sqrt\delta\,\beta^{-3/2}+\delta/\beta\big)
+O_{\beta,m}(\delta^{3/2})\Bigg).
\end{split}
\]
The total exponent error includes both the large-$\beta$ coefficient approximation and the higher-harmonic and remainder terms. The small-amplitude limit is taken at fixed $\beta$; the $O_{\beta,m}$ constant need not be uniform as $\beta\to\infty$.

 \bibliographystyle{plainnat}
\bibliography{citations}

\end{document}